\documentclass[10pt, reqno]{amsart}
\usepackage{hhline}

\usepackage{tikz}
\usepackage{verbatim}
\usepackage{amsmath}
\usepackage{amsxtra}
\usepackage{amscd}
\usepackage{amsthm}
\usepackage{amsfonts}
\usepackage{amssymb}
\usepackage{eucal}
\usetikzlibrary{arrows}
\usepackage{array}

\usepackage[all, cmtip]{xy}
\usepackage{stmaryrd}
\usepackage{color}

\usepackage{amsthm,amsfonts,amssymb,amsmath,amsxtra}
\usepackage[all]{xy}
\SelectTips{cm}{}
\usepackage{xr-hyper}
\usepackage[colorlinks=
   citecolor=Black,
   linkcolor=Red,
   urlcolor=Blue]{hyperref}
\usepackage{verbatim}

\usepackage{mathrsfs}

\RequirePackage{xspace}
\RequirePackage{etoolbox}
\RequirePackage{varwidth}
\RequirePackage{enumitem}
\RequirePackage{tensor}
\RequirePackage{mathtools}
\RequirePackage{longtable}
\RequirePackage{multirow}

\def\fl{\flat}
\def\l{\lambda}
\def\tW{\tilde W}
\def\o{\omega}
\def\a{\alpha}
\def\<{\langle}
\def\>{\rangle}
\def\e{\epsilon}
\def\a{\alpha}

\newcommand{\BA}{\ensuremath{\mathbb {A}}\xspace}

\newcommand{\BC}{\ensuremath{\mathbb {C}}\xspace}

\newcommand{{\BG}}{\ensuremath{\mathbb {G}}\xspace}

\newcommand{{\BK}}{\ensuremath{\mathbb {K}}\xspace}
\newcommand{\BL}{\ensuremath{\mathbb {L}}\xspace}
\newcommand{\BM}{\ensuremath{\mathbb {M}}\xspace}

\newcommand{\BP}{\ensuremath{\mathbb {P}}\xspace}
\newcommand{\BQ}{\ensuremath{\mathbb {Q}}\xspace}
\newcommand{\BR}{\ensuremath{\mathbb {R}}\xspace}

\newcommand{\BV}{\ensuremath{\mathbb {V}}\xspace}

\newcommand{\BX}{\ensuremath{\mathbb {X}}\xspace}

\newcommand{\BZ}{\ensuremath{\mathbb {Z}}\xspace}

\newcommand{\CA}{\ensuremath{\mathcal {A}}\xspace}
\newcommand{\CB}{\ensuremath{\mathcal {B}}\xspace}

\newcommand{\CF}{\ensuremath{\mathcal {F}}\xspace}
\newcommand{\CG}{\ensuremath{\mathcal {G}}\xspace}
\newcommand{\CH}{\ensuremath{\mathcal {H}}\xspace}

\newcommand{\CL}{\ensuremath{\mathcal {L}}\xspace}
\newcommand{\CM}{\ensuremath{\mathcal {M}}\xspace}

\newcommand{\CO}{\ensuremath{\mathcal {O}}\xspace}

\newcommand{\CS}{\ensuremath{\mathcal {S}}\xspace}
\newcommand{\CT}{\ensuremath{\mathcal {T}}\xspace}

\newcommand{\ad}{{\mathrm{ad}}}

\DeclareMathOperator{\diag}{diag}

\DeclareMathOperator{\Adm}{Adm}

\newcommand{\GL}{\mathrm{GL}}

\newcommand{\GU}{\mathrm{GU}}

\newcommand{\Int}{\ensuremath{\mathrm{Int}}\xspace}

\newcommand{\loc}{\ensuremath{\mathrm{loc}}\xspace}

\newcommand{\naive}{\ensuremath{\mathrm{naive}}\xspace}

\newcommand{\PGL}{{\mathrm{PGL}}}

\DeclareMathOperator{\Res}{Res}

\newcommand{\SL}{{\mathrm{SL}}}
\DeclareMathOperator{\Spec}{Spec}
\DeclareMathOperator{\Spf}{Spf}
\newcommand{\SO}{{\mathrm{SO}}}
\newcommand{\Sp}{{\mathrm{Sp}}}

\DeclareMathOperator{\sgn}{sgn}

\newcommand{\wt}{\widetilde}

\newcommand{\ov}{\overline}

 \newcommand{\br}{\breve}

\newtheorem{theorem}{Theorem}
\newtheorem{proposition}[theorem]{Proposition}
\newtheorem{lemma}[theorem]{Lemma}
\newtheorem {conjecture}[theorem]{Conjecture}
\newtheorem{corollary}[theorem]{Corollary}

\theoremstyle{definition}
\newtheorem{definition}[theorem]{Definition}
\newtheorem{example}[theorem]{Example}

\newtheorem{remark}[theorem]{Remark}
\newtheorem{remarks}[theorem]{Remarks}

\newtheorem{Notation}[theorem]{Notation}

\newenvironment{altenumerate}
   {\begin{list}
      {\textup{(\theenumi)} }
      {\usecounter{enumi}
       \setlength{\labelwidth}{0pt}
       \setlength{\labelsep}{0pt}
       \setlength{\leftmargin}{0pt}
       \setlength{\itemsep}{\the\smallskipamount}
       \renewcommand{\theenumi}{\roman{enumi}}
      }}
   {\end{list}}
\newenvironment{altitemize}
   {\begin{list}
      {$\bullet$}
      {\setlength{\labelwidth}{0pt}
	   \setlength{\itemindent}{5pt}
       \setlength{\labelsep}{5pt}
       \setlength{\leftmargin}{0pt}
       \setlength{\itemsep}{\the\smallskipamount}
      }}
   {\end{list}}

\numberwithin{equation}{section}
\numberwithin{theorem}{section}

\setitemize[0]{leftmargin=*,itemsep=\the\smallskipamount}
\setenumerate[0]{leftmargin=*,itemsep=\the\smallskipamount}

\renewcommand{\to}{%
   \ifbool{@display}{\longrightarrow}{\rightarrow}%
   }
\let\shortmapsto\mapsto
\renewcommand{\mapsto}{%
   \ifbool{@display}{\longmapsto}{\shortmapsto}%
   }
\newlength{\olen}
\newlength{\ulen}
\newlength{\xlen}
\newcommand{\xra}[2][]{%
   \ifbool{@display}%
      {\settowidth{\olen}{$\overset{#2}{\longrightarrow}$}%
       \settowidth{\ulen}{$\underset{#1}{\longrightarrow}$}%
       \settowidth{\xlen}{$\xrightarrow[#1]{#2}$}%
       \ifdimgreater{\olen}{\xlen}%
          {\underset{#1}{\overset{#2}{\longrightarrow}}}%
          {\ifdimgreater{\ulen}{\xlen}%
             {\underset{#1}{\overset{#2}{\longrightarrow}}}
             {\xrightarrow[#1]{#2}}}}%
      {\xrightarrow[#1]{#2}}
   }
\makeatother
\newcommand{\xyra}[2][]{%
   \settowidth{\xlen}{$\xrightarrow[#1]{#2}$}%
   \ifbool{@display}%
      {\settowidth{\olen}{$\overset{#2}{\longrightarrow}$}%
       \settowidth{\ulen}{$\underset{#1}{\longrightarrow}$}%
       \ifdimgreater{\olen}{\xlen}%
          {\mathrel{\xymatrix@M=.12ex@C=3.2ex{\ar[r]^-{#2}_-{#1} &}}}%
          {\ifdimgreater{\ulen}{\xlen}%
             {\mathrel{\xymatrix@M=.12ex@C=3.2ex{\ar[r]^-{#2}_-{#1} &}}}
             {\mathrel{\xymatrix@M=.12ex@C=\the\xlen{\ar[r]^-{#2}_-{#1} &}}}}}%
      {\mathrel{\xymatrix@M=.12ex@C=\the\xlen{\ar[r]^-{#2}_-{#1} &}}}%
   }
\makeatletter
\newcommand{\xla}[2][]{%
   \ifbool{@display}%
      {\settowidth{\olen}{$\overset{#2}{\longleftarrow}$}%
       \settowidth{\ulen}{$\underset{#1}{\longleftarrow}$}%
       \settowidth{\xlen}{$\xleftarrow[#1]{#2}$}%
       \ifdimgreater{\olen}{\xlen}%
          {\underset{#1}{\overset{#2}{\longleftarrow}}}%
          {\ifdimgreater{\ulen}{\xlen}%
             {\underset{#1}{\overset{#2}{\longleftarrow}}}
             {\xleftarrow[#1]{#2}}}}%
      {\xleftarrow[#1]{#2}}
   }
\newcommand{\isoarrow}{%
   \ifbool{@display}{\overset{\sim}{\longrightarrow}}{\xrightarrow\sim}%
   }

\newcommand{\Mloc}{\BM^{\rm loc}_K(G,\{\mu\})}

\begin{document}

\title{Good and semi-stable reductions of Shimura varieties}

\author{X. He}
\address{Department of Mathematics, University of Maryland, College Park, MD 20742, USA}
\curraddr{The Institute of Mathematical Sciences and Department of Mathematics, The Chinese University of Hong Kong, Shatin, N.T., Hong Kong.}
\email{xuhuahe@gmail.com}
\author{G. Pappas}
\address{Dept. of Mathematics, Michigan State University, E. Lansing, MI 48824, USA}
\email{pappasg@msu.edu}

\author{M. Rapoport}
\address{Mathematisches Institut der Universit\"at Bonn, Endenicher Allee 60, 53115 Bonn, Germany, and Department of Mathematics, University of Maryland, College Park, MD 20742, USA}
\email{rapoport@math.uni-bonn.de}

\date{\today}

\begin{abstract}
We study variants of the local models constructed by the second author and Zhu and consider corresponding integral models of Shimura varieties of abelian type. We determine all cases of good, resp. of semi-stable, reduction under  tame ramification hypotheses.  
\end{abstract}

\date{\today}
\maketitle

\tableofcontents

\section{Introduction}
 The problem of the reduction modulo $p$ of a Shimura variety has a long and complicated history, perhaps beginning with Kronecker. The case of the modular curve (the Shimura variety associated to $\GL_2$) is essentially solved after the work of Igusa, Deligne, Drinfeld and Katz-Mazur. In particular, it is known that the modular curve has good reduction at $p$ if the level structure is prime to $p$. If the level structure is of $\Gamma_0(p)$-type (in addition to some level structure prime to $p$), then the modular curve has semi-stable reduction (one even has a global understanding of the reduction modulo $p$, as the union of two copies of the modular curve with level structure prime to $p$, crossing transversally at the set of supersingular points). Are there other level structures such that the reduction modulo $p$ is good, resp. is semi-stable? 

This is the question addressed in the present paper, in the context of general Shimura varieties. The question can be interpreted in two different ways. One can ask whether there exists \emph{some} model over $\Spec \BZ_{(p)}$ which has good, resp. semi-stable reduction. In the case of the modular curve, one can prove that, indeed, the two examples above exhaust all possibilities (this statement has to be interpreted correctly, by considering the natural compactification of the modular curve). This comes down to a statement about the spectral decomposition under the action of the Hecke algebra of the $\ell$-adic cohomology of modular curves. Unfortunately, the generalization of this statement to other Shimura varieties seems out of reach at the moment.

The other possible interpretation of the question is to ask for good, resp. semi-stable, reduction of a specific class of $p$-integral models of Shimura varieties. Such a specific class has been established in recent years for Shimura varieties with level structure which is \emph{parahoric} at $p$, the most general result being due to M.~Kisin and the second author \cite{KP}. The main point of these models is that their  singularities are modeled by their associated \emph{local models}, cf. \cite{Pappasicm}. These are projective varieties which are defined in a certain sense by linear algebra, cf. \cite{H, R}. More precisely, for every closed point of the reduction modulo $p$ of the $p$-integral model of the Shimura variety, there is  an isomorphism between the strict henselization of its local ring and the strict henselization of the local ring of a corresponding closed point in the reduction modulo $p$ of the local model. Very often every closed point of the local model is attained in this way. In this case, the model of the Shimura variety has good, resp. semi-stable, reduction if and only if the local model has this property. Even when this attainment statement is not known, we deduce that if the local model has good, resp. semi-stable, reduction, then so does the model of the Shimura variety. Therefore, the emphasis of the present paper is on the structure of the singularities of the local models and our results  determine   local models which have good, resp. semi-stable reduction.

Let us state now the main results of the paper, as they pertain to local models. See Section \ref{s:shimvar} for corresponding results for Shimura varieties, and Section \ref{s:RZspaces} for results on Rapoport-Zink spaces.  Local models are associated to \emph{local model triples}. Here a LM triple over a finite extension $F$ of $\BQ_p$ is a triple $(G, \{\mu\}, K)$ consisting of a reductive group $G$ over $F$, a conjugacy class of cocharacters $\{\mu\}$ of $G$ over an algebraic closure of $F$, and a parahoric group $K$ of $G$. We sometimes write $\CG$ for the affine smooth group scheme over $O_F$ corresponding to $K$.  It is assumed that the cocharacter $\{\mu\}$ is minuscule (i.e., any root takes values in $\{0, \pm 1\}$ on $\{\mu\}$). The \emph{reflex field} of the LM triple $(G, \{\mu\}, K)$ is the field of definition of the conjugacy class $\{\mu\}$. One would like to associate to  $(G, \{\mu\}, K)$ a \emph{local model} $\Mloc$, a flat projective  scheme over the ring $O_E$ of integers in the corresponding reflex field $E$, with action of $\CG_{O_E}$. Also, one would like to characterize uniquely this local model.

At this point a restrictive hypothesis enters. Namely, we have to impose throughout most of the paper that the group $G$ splits over a tamely ramified extension. Indeed, only under this hypothesis, X.~Zhu and the second author define local models \cite{PZ}  which generalize the local models defined earlier in the concrete situations considered by Arzdorf, de~Jong,  G\"ortz, Pappas, Rapoport-Zink, Smithling, comp. \cite{PRS}. 
Our first main result is that the result of the construction in \cite{PZ} is unique, i.e., is independent of all auxiliary choices. This independence issue  was left 
unexamined in \emph{loc. cit.}. In fact, we slightly modify here the construction in \cite{PZ} and define  $\Mloc$ in such a way  that it always has reduced special fiber, a property that is stable under base change. In \cite{PZ} this reducedness property for the local models of \cite{PZ} was established  only when $\pi_1(G_{\rm der})$ has order prime to $p$ (in which case the local model of \cite{PZ} coincides with $\Mloc$). This then also implies that the definition of $\Mloc$ is well-posed. Here, we need uniqueness after base-changing to an unramified extension to even make unambiguous sense of our classification of local models which are smooth or semi-stable. We show:

\begin{theorem}\label{indepIntro}
Let $(G, \{\mu\}, K)$ be an LM triple such that $G$ splits over a tamely ramified extension of $F$. 
The local model $\Mloc$  is independent of all choices made in its construction. Its generic fiber is $G_E$-equivariantly isomorphic to the projective homogeneous space $X_{\{\mu\}}$, and its 
 geometric special fiber $\Mloc\otimes_{O_E}k $ is reduced  and is
$\CG\otimes_{O_F}k$-equivariantly  isomorphic to the \emph{$\{\mu\}$-admissible locus} 
$\frak A_K(  G, \{  \mu\})$ in an affine partial flag variety over $k$.
\end{theorem}

We refer to the body of the text for undefined items. We conjecture that the properties in Theorem \ref{indepIntro} uniquely characterize the local model $\BM^\loc_K(G, \{\mu\})$, cf.  Conjecture \ref{uniqLM}. 

Local models should exist even without the tameness hypothesis.  Levin \cite{Levin} has achieved some progress on this front by extending the Pappas-Zhu construction to some wild cases. Scholze \cite{Schber} considers  the general case and defines a \emph{diamond local model}  over $O_E$ attached to the LM triple $(G, \{\mu\}, K)$. Furthermore, he proves that there is at most one local model whose associated  ``$v$-sheaf''  is the diamond local model. Unfortunately, the existence question is still open. Hence in the general situation, Scholze does not have a construction of a local model but has a characterization; under our tameness hypothesis, we have a construction but no characterization. In the case of \emph{classical groups}, the situation is somewhat better: under some additional hypothesis, we then show that the local models of \cite{PZ} satisfy Scholze's characterizing property, cf. Corollary \ref{Schmodclass}. 

Our second main result gives a characterization of all cases when Pappas-Zhu  local models have good reduction. In its statement, $\breve F$ denotes the completion of the maximal unramified extension of $F$. 

\begin{theorem}\label{main1Intro}
Let $(G, \{\mu\}, K)$ be a LM triple over $F$ such that $G$ splits over a tamely ramified extension of $F$.  Assume that $p\neq 2$. 
Assume that $G_\ad$ is $F$-simple, $\mu_\ad$ is not the trivial cocharacter, and that in the product decomposition over $\breve F$, 
$
G_\ad\otimes_F\breve F= \breve G_{\ad, 1}\times\cdots \times \breve G_{\ad, m},
$
each factor $\breve G_{\ad, i}$ is  \emph{absolutely simple}. 
Then the  local model $\Mloc$ is smooth over $\Spec O_E$ if and only if $K$ is hyperspecial or $(G, \mu, K)$ is an LM triple of exotic good reduction type. 
\end{theorem}

Here the first alternative, that $K$ be hyperspecial, is the natural generalization of the case of the modular curve with level structure prime to $p$.  There are two cases of the second (``exotic'') alternative: The first is a striking discovery of T.~Richarz, cf.  \cite[Prop. 4.16]{Arz}. He proved that the local model associated to an even, resp. odd, ramified unitary group $G$, the cocharacter $\{\mu\}=(1, 0, \ldots, 0)$, and the parahoric subgroup which is the stabilizer of a $\pi$-modular, resp. almost $\pi$-modular, lattice has good reduction (the case of a $\pi$-modular lattice is much easier and was known earlier, cf. \cite[5.3]{PR}). The second case, which is a new observation of the current paper, is that of the local model associated to an even ramified quasi-split orthogonal group $G$, the cocharacter $\{\mu\}$ that corresponds to the orthogonal Grassmannian of isotropic subspaces of maximal dimension, and the parahoric $K$ given by the stabilizer of an almost  selfdual lattice. We therefore see that in the statement of the theorem both implications are interesting and non-trivial. 

Let us comment on the hypotheses in this theorem. The hypothesis that $G_\ad$ be $F$-simple is just for convenience. However, the hypothesis that each factor $ \breve G_{\ad, i}$ be absolutely simple is essential to our method. It implies that the translation element associated to $\{\mu\}$ in the extended affine Weyl group for $ \breve G_{\ad, i}$ is not too large and this limits drastically the number of possibilities of LM triples with associated local models of good reduction.  Note that the tameness assumption on $G$ is automatically satisfied for $p\geq 5$ under these hypotheses. We refer to the passage after the statement of Theorem \ref{maingoodred} for a description of the structure of the proof of Theorem \ref{main1Intro}. Roughly speaking, we eliminate most possibilities by various combinatorial considerations and calculations of Poincare polynomials.
Ultimately, we reduce to a few cases that can be examined explicitly, and a single exceptional case (for the quasi-split ramified triality ${}^3D_4$) which is handled by work of Haines-Richarz \cite{HR}.

Our third main result gives a characterization of all cases when Pappas-Zhu local models have semi-stable reduction. 

\begin{theorem}\label{main2Intro}
Let $(G, \{\mu\}, K)$ be a LM triple over $F$ such that $G$ splits over a tamely ramified extension of $F$.  Assume $p\neq 2$. 
Assume that $G_\ad$  is  absolutely simple. Then the  local model $\Mloc$ has semi-stable, but non-smooth, reduction  over $\Spec O_E$ if and only if its enhanced Tits datum appears in the table of Theorem \ref{mainssred}. 
\end{theorem}

Again, let us comment on the hypothesis in this theorem. We are limited in the hypotheses of this theorem by the same constraints as in the criterion for good reduction--but we have to avoid the product of semi-stable  varieties since these are no longer semi-stable: this explains why we make the assumption that $G_\ad$ be  absolutely simple. The enhanced Tits datum of an LM triple is defined in Definition \ref{def:enhanced}. In the situation of Theorem \ref{main2Intro}, the enhanced Tits datum determines the LM triple over $F$ up to central isogeny and up to a scalar extension to an unramified extension of $F$. 

Again, as with Theorem \ref{main1Intro}, both implications in Theorem \ref{main2Intro} are interesting and non-trivial. The semi-stability in the case of the LM triple $(\PGL_n, (1,0,\ldots,0), K)$, where $K$ is an arbitrary parahoric subgroup has been known for a long time, due to the work of Drinfeld \cite{Dr}. The case of the LM triple $(\PGL_n, (1^{(r)},0^{(n-r)}), K)$, where $r$ is arbitrary and where $K$ is the parahoric subgroup stabilizing two adjacent vertex lattices appears in the work of G\"ortz \cite{Goertz}, although the significance of this case went unnoticed.  Related calculations also appear in work of Harris and Taylor \cite{HaT}. Another interesting case 
is when $G$ is the adjoint group of a symplectic group with its natural Siegel cocharacter and $K$ is the simultaneous stabilizer of a selfdual vertex lattice and an adjacent almost selfdual vertex lattice. This subgroup $K$ is the so-called ``Klingen parahoric" and the semi-stability in this case has been shown by Genestier and Tilouine \cite[6.3]{GenTil}. The case that triggered our interest in the classification of semi-stable local models is the case recently discovered by Faltings \cite{F}. Here $G$ is the adjoint group of the split orthogonal group of even size $2n$, the minuscule coweight is the one which leads to the hermitian-symmetric space given by a quadric, and $K$ is the parahoric subgroup simultaneously stabilizing the selfdual and the selfdual up to a scalar vertex lattices. Faltings' language is different from ours, and it could take the reader some effort to make the connection between our result and his. However, our point of view allows us to view Faltings' result as a corollary of the general results of  \cite{KP}; see Example \ref{Faltings}.  The list of Theorem \ref{mainssred} contains two more cases of LM triples with semi-stable associated local models, both for orthogonal groups, which seem to be new. Let us note here that the corresponding integral models of Shimura varieties are ``canonical'' in the sense of \cite{PCan}. In almost all of these cases of smooth or semi-stable reduction, these integral models can also be uniquely characterized more directly using an idea of Milne \cite{Milne} and results of Vasiu-Zink \cite{VasiuZink}, see Theorem \ref{3.6}.

We refer to the end of Section \ref{s:mainresults}  for  a description of the proof of Theorem \ref{main2Intro}. As a consequence of the proof, we obtain the following remarkable fact.
\begin{corollary}\label{corIntro}
Let $(G, \{\mu\}, K)$ be a LM triple over $F$  such that $G$ splits over a tamely ramified extension of $F$.  Assume $p\neq 2$. 
Assume that  $G$ is adjoint and absolutely simple. Then the  local model $\Mloc$ has semi-stable reduction if and only if $\Mloc$ has strictly pseudo semi-stable reduction. 
\end{corollary}
 We refer to Definition \ref{def:pseudorat} for what it means that a scheme over the  spectrum of a discrete valuation ring has strictly pseudo semi-stable reduction. It is a condition that only involves the reduced special fiber; the above corollary shows that in the case at hand it implies that the total scheme $\Mloc$ is regular.

\smallskip

Let us now explain the lay-out of the paper. In Section \ref{s:LM}, we recall the local models constructed in \cite{PZ} and show that they are independent (in a sense to be made precise) of the auxiliary data used in their construction; we also introduce the modification of this construction that has reduced special fiber, and compare it with the hypothetical construction of Scholze \cite{Schber}. In Section \ref{s:shimvar} we explain the relation between Shimura varieties and local models. Section \ref{s:RZspaces} does the same for Rapoport-Zink spaces.  Section \ref{s:mainresults} contains the statements of the main results on local models. In Section \ref{s:pssredCCP} we introduce the concepts of (rationally) strictly pseudo semi-stable reduction and the component count property  (CCP condition), and prove that the former condition implies the latter. In Section \ref{s:CCP}, we give a complete list of all enhanced Coxeter data for which the CCP condition is satisfied. In Section \ref{s:ratpsss}, we exclude from this list the cases that do not have rationally strictly pseudo semi-stable reduction. At this point, we have all tools available to prove Theorem \ref{main1Intro}, and this is the content of Section \ref{s:maingood}. In Section \ref{s:psssred}, we use Kumar's criterion to eliminate all cases that do not have strictly pseudo semi-stable reduction. At this point, we have all tools available to prove one implication of Theorem \ref{main2Intro}, and this is the content of Section \ref{s:oneimpl}, where we also prove Corollary \ref{corIntro}. In the final long Section \ref{s:converse1}, we prove the other implication of Theorem \ref{main2Intro}.

\smallskip
\smallskip

\noindent \ \ {\bf Notation:} For a local field $F$, we denote by $\breve F$ the completion of its maximal unramified extension (in a fixed algebraic closure). We denote by $\kappa_F$ the residue field of $F$ and by $k$ the algebraic closure of $\kappa_F$ which is the residue field of $\breve F$. We always denote by $p$ the characteristic of $\kappa_F$.  

For a reductive group $G$, we denote by $G_{\rm der}$ its derived group, by $G_{\rm sc}$ the simply-connected covering of $G_{\rm der}$, and by $G_\ad$ its adjoint group. If $G$ is defined over the local field $F$, we denote by $\CB(G, F)$ the extended Bruhat-Tits building of $G(F)$; if $S\subset G$ is a maximal $F$-split torus of $G$, we denote by ${\bf A}(G,S, F)\subset \CB(G, F)$ the corresponding apartment. A parahoric subgroup $K$ of $G(F)$ is, by definition, the \emph{connected stabilizer} of a point $x\in \CB(G, F)$; by \cite{BTII}, there is a smooth affine group scheme $\CG_x$ over $O_F$ with generic fiber $G$ and connected special fiber such that $K=\CG_x(O_F)$.

We often write the base change $X\times_{\Spec R}\Spec R'$ as $X\otimes_RR'$, or simply as $X_{R'}$.

\smallskip
\smallskip

\noindent \ \ {\bf Acknowledgements:} We thank P.~Deligne, G.~Faltings, U.~G\"ortz, T.~Haines, V.~Pilloni, T.~Richarz, P.~Scholze and B.~Smithling
 for helpful discussions, and W.~M.~McGovern for interesting e-mail exchanges. We also thank the referee for his work. X. H. was partially supported by NSF grant DMS-1463852. G.P was partially supported by NSF grant DMS-1701619. M.R. was supported by the grant SFB/TR 45  from the Deutsche Forschungsgemeinschaft and by funds connected with the Brin E-Nnovate Chair  at the University of Maryland.
 
\smallskip

\vfill\eject

\section{Local models}\label{s:LM}

In this section, we discuss the theory of local models, as used in the paper. 

\subsection{Local  model triples}
 Let $F$ be a finite extension  of $\BQ_p$, with algebraic closure $\ov F$. A \emph{local model triple } (LM triple) over $F$ is a triple $(G, \{\mu\}, K)$ consisting of a connected reductive group $G$ over $F$, a conjugacy class $\{\mu\}$ of cocharacters of $G_{\ov F}$, and a parahoric subgroup $K$ of $G(F)$. It is assumed that $\{\mu\}$ is a minuscule cocharacter. 
 We denote by $\CG=\CG_K$ the extension of $G$ to a smooth group scheme  over $O_F$ corresponding to $K$. Then $\CG$  has connected fibers  and satisfies $K=\CG(O_F)$. We set
 $\br K=\CG(O_{\breve F})$. Sometimes we also write $(G, \{\mu\}, \CG)$ for  the LM triple.

 Two LM triples $(G, \{\mu\}, K)$ and $(G', \{\mu'\}, K')$ are isomorphic if there exists an isomorphism $G\to G'$ which takes $\{\mu\}$ to $\{\mu'\}$ and $\breve K$ to a conjugate of $\breve K'$. More generally, a morphism 
   \[
   \phi: (G, \{\mu\}, K)\to (G', \{\mu'\}, K')
   \] of LM triples is a group
 scheme homomorphism $\phi: G\to G'$ such that $\{\mu'\}=\{\phi\circ \mu\}$ and $\phi(\breve K)\subset g'  \breve K' g'^{-1}$, for some $g'\in G'(\breve F)$.

 Let $E$ be the field of definition of $\{\mu\}$ inside the fixed algebraic closure $\ov F$ of $F$, with its ring of integers $O_E$. We denote by $k$ the algebraic closure of its residue field $\kappa_E$. 
 Denote by $X_{\{\mu\}}$ the partial flag variety over $E$ of $G_E$ associated to $\{\mu\}$.
 
 \subsection{Group schemes} Let $G$ be a reductive group over $F$ that splits over a tame extension of $F$. Choose a uniformizer $\pi$ of $F$. The theory of \cite{PZ} starts with the construction of a reductive group scheme $\underline G $ over $O_F[u^\pm]:=O_F[u, u^{-1}]$ which induces by specialization  $(O_F[u^\pm]\to F, u\mapsto \pi)$ the group $G$ over $F$. Let $G'$ be the reductive group induced by $\underline G $ by specialization along $(O_F[u^\pm]\to \kappa_F((u)), \pi\mapsto 0)$.
 
 By \cite[Thm 4.1]{PZ}, there exists a smooth affine group scheme $\underline \CG $ over $O_F[u]$ with connected fibers which restricts to $\underline G $ over $O_F[u^\pm]$ and which induces the parahoric group scheme $\CG$ under the specialization $(O_F[u]\to O_F, u\mapsto \pi)$. It also induces a parahoric group scheme $\CG'$ under the specialization $(O_F[u]\to \kappa_F[[u]], \pi\mapsto 0)$, with an identification
 \begin{equation}\label{identpara}
 \CG\otimes_{O_F, \pi\mapsto 0} k=\CG'\otimes_{\kappa_F[[u]], u\mapsto 0}k .
 \end{equation}
  We denote by $\breve{ \underline G} $, resp.  $\breve{ \underline \CG} $, the group schemes over $O_{\breve F}[u^\pm]$, resp. $O_{\breve F}[u]$, obtained by base change $O_F\to O_{\breve F}$. 
  
 Let us recall some aspects of the construction of these group schemes. The reader is referred to \cite{PZ} for more details. For simplicity we abbreviate $O=O_F$, $\breve O=O_{\breve F}$.
 
 Denote by $H$ (resp. $G^*$) the corresponding split (resp. quasi-split) form of $G$ over $O$ (resp. $F$). These forms are each unique up to isomorphism. 

Fix, once and for all, a pinning  $(H, T_H, B_H, e_O)$ defined over $O$. As in \cite{PZ}, we denote by $\Xi_H$ the group of automorphisms of the based root datum corresponding to  $(H, T_H, B_H)$.

Pick a maximal $F$-split torus $A\subset G$. By \cite[5.1.12]{BTII}, we can choose an $F$-rational maximal $\breve F$-split torus $S$ in $G$ that contains $A$ and a minimal $F$-rational parabolic subgroup $P$ which contains $Z_G(A)$.
In \cite{PZ}, a triple $(A, S, P)$ as above, is called a \emph{rigidification} of $G$. Since by Steinberg's theorem, the group $\breve G=G\otimes_F \breve F$ is quasi-split, $T=Z_G(S)$ is a maximal torus of $G$ which is defined over $F$.  

As in \cite[2.4.2]{PZ}, the indexed root datum of the group $G$ over $F$ gives a $\Xi_H$-torsor $\tau$ over $\Spec(F)$.
 Then, by \cite[Prop. 2.3]{PZ}, we obtain a pinned quasi-split group $(G^*, T^*, B^*, e^*)$  over $F$ and, by the identification of tame finite extensions of $ F$ with \'etale finite covers of $ O[u^\pm]$ given by $u\mapsto \pi$, a pinned quasi-split group $(\underline  G^*, \underline T^*, \underline B^*, \underline e^*)$ over $O[u^\pm]$
 (see \emph{loc. cit.}, 3.3). As in \cite{PZ}, we denote by 
$\underline S^*$ the maximal split subtorus of $\underline T^*$. 
We have an identification 
\begin{equation}\label{ident1}
(\underline  G^*, \underline T^*, \underline B^*, \underline e^*)\otimes_{O[u^\pm],u\mapsto \pi} F=(G^*, T^*, B^*, e^*).
\end{equation}

\begin{remark}\label{indep} a) The base change 
 $(\underline  G^*, \underline T^*, \underline B^*, \underline e^*)\otimes_{O[u^\pm]} \breve O[u^\pm]$ is independent of the choice of uniformizer $\pi$ of $F$. This follows by the above, since the identification of the tame Galois  group of $\breve F$ with $\BZ'(1)=\prod_{l\neq p} \BZ_l(1)$, given by $\gamma\mapsto \gamma(\pi^{1/m})/\pi^{1/m}$, does not depend on the choice of the uniformizer $\pi$.

b) It is not hard to see, using \cite[3.3.2]{PZ}, that the Picard group of every finite \'etale cover of $ O[u^\pm]$ is trivial. The argument in the proof of \cite[Prop. 7.2.12]{Conrad} , then shows that,  up to isomorphism, a quasi-split reductive group scheme over $O[u^\pm]$ is uniquely determined by a corresponding $\Xi_H$-torsor over $O[u^\pm]$ and therefore obtained by the above construction.  In fact, any quasi-split reductive group scheme over $O[u^\pm]$ is determined, up to isomorphism, by its base change along $O[u^\pm]\to F$, given by $u\mapsto \pi$.
\end{remark}

As in \cite{PZ}, we obtain from (\ref{ident1}) identifications of apartments
\begin{equation}\label{identApt}
{\bf A}(G^*, S^*, \breve F)={\bf A}(\underline  G^*_{\kappa((u))}, \underline S^*_{\kappa((u))}, \kappa((u))),
\end{equation}
for both $\kappa=\breve F$, $k$. Given $x^*\in {\bf A}(G^*, S^*, \breve F)\subset \CB(G^*, \breve F)$, Theorem 4.1 of \cite{PZ}, produces a smooth connected affine group scheme 
\[
 \underline{\CG}^*:= {\underline\CG}^*_{x^*}
 \]
  over $\breve O[u]$ which extends $\underline G^*\otimes_{O[u^\pm]}\breve O[u^\pm]$. 
 Using Remark \ref{indep} we see that ${\underline\CG}^*_{x^*}$ does not depend on 
 the choice of the uniformizer. (Notice that ${\underline\CG}^*_{x^*}$ might not descend
 over $O[u]$ since $x^*$ is not necessarily $F$-rational.)

Now, given $x\in \CB(G, F)$ which corresponds to $K$, choose a rigidification $(A, S, P)$
of $G$ over $F$, such that $x\in {\bf A}(G, S, F)$.

Since $\breve G=G\otimes_F\breve F$ and $G^*\otimes_F{\breve F}$ are both quasi-split and inner forms of each other, we can choose an inner twist, i.e., a ${\rm Gal}(\breve F/F)$-stable $G^*_{\rm ad}(\breve F)$-conjugacy class of an isomorphism
\[
\psi: G\otimes_F \breve F\xrightarrow{\sim} G^*\otimes_F{\breve F}.
\]
Then the class $[g_\sigma]$ of the $1$-cocycle $\sigma\mapsto \Int(g_\sigma)=\psi\sigma\psi^{-1}\sigma^{-1}$ in ${\rm H}^1(\hat \BZ, G^*_{\rm ad}(\breve F))$ maps to the class in 
${\rm H}^1(\hat \BZ, {\rm Aut}(G^*)(\breve F))$ that gives the twist $G$ of $G^*$. The orbit of $[g_\sigma]$
under the natural action of ${\rm Out}(G^*)(F)$ on ${\rm H}^1(\hat \BZ, G^*_{\rm ad}(\breve F))$ only depends on the
isomorphism class of $G$. In \cite{PZ}, it shown that there is a choice of $\psi$ that depends on the rigidification $(A, S, P)$ such that the inclusion  
\[
 \CB(G, F)\subset \CB(G, \breve F)\xrightarrow{\psi_*} \CB(G^*, \breve F)
 \]
identifies ${\bf A}(G, S, \breve F)$ with ${\bf A}(G^*, S^*, \breve F)$; set $x^*:=\psi_*(x)$. In \emph{loc. cit.} the group scheme $\underline \CG_x$ over $O[u]$ is then constructed such that 
$\psi$ extends to   isomorphisms
\[
\underline \psi: \breve{\underline G}= \underline\CG_x\otimes_{O[u]}\breve O[u^\pm]\xrightarrow {\sim}  \breve{\underline G^*},
 \] 
\[
\underline \psi: \breve{\underline\CG_x}\xrightarrow {\sim} {\underline \CG}^*_{x^*}.
\]
A priori, the group scheme $\underline \CG_x$  depends on several choices, in particular of $\underline G$ and of the uniformizer $\pi$. However, we now show:

\begin{proposition}\label{ind2} a) Up to isomorphism,  the   group scheme $\breve{ \underline G}=\underline G\otimes_{O[u^\pm]}\breve O[u^\pm]$    depends only on $\breve G=G\otimes_F\breve F$.

b) Up to isomorphism,  the   group scheme $\breve{\underline\CG}_x=\underline\CG_x\otimes_{O[u]}\breve O[u]$ depends only on $G\otimes_F\breve F$ and the $G_{\rm ad}(\breve F)$-orbit of $x\in \CB(G, \breve F)$.

c) For any $a\in O^\times$, the group scheme $\underline\CG_{x}\otimes_{ O[u]}\breve O[u]$ supports an isomorphism
\[
R_a: a^*(\underline\CG _{x }\otimes_{ O[u]}\breve O[u])\xrightarrow{\sim} \underline\CG _{x }\otimes_{O[u]}\breve O[u].
\]
that lifts the isomorphism given by  $u\mapsto a\cdot u$. 
\end{proposition}

\begin{proof} 
By the construction, as briefly recalled above, there are   isomorphisms
\[
\underline \psi:  \breve{ \underline G}\xrightarrow {\sim} \breve {\underline G}^* ,\quad 
\underline \psi: \breve {\underline\CG}_x \xrightarrow {\sim} {\underline \CG}^*_{x^*}.
\]
Hence, it is enough to show corresponding independence statements for 
$\breve{\underline G}^*$ and $ {\underline \CG}^*_{x^*}$. First we notice that by Remark \ref{indep}, $ \breve {\underline G}^*$ only depends on $G\otimes_F \breve F$ and so part (a) follows.
Now using the argument in \cite[4.3.1]{PZ}, we see that changing the rigidification $(A, S, P)$ of $G$, changes the point $x^*$
to another point $x'^*$ of ${\bf A}(G^*, S^*, \breve F)$ in the same $G^*_{\rm ad}(\breve F)$-orbit, hence in the same orbit under the adjoint Iwahori-Weyl group $\tilde W_{G^*_{\rm ad}}$. However, each element $w$ of $\tilde W_{G^*_{\rm ad}}$ lifts to an element $\underline n$ of $\underline G^*_{\rm ad}(\breve O[u^{\pm}])$ that normalizes $\underline S^*$. Acting by $\Int(\underline n)$ gives an isomorphism between the group schemes $ {\underline \CG}^*_{x^*}$ and $ {\underline \CG}^*_{x'^*}$. This implies   statement (b). To see (c), we first observe that Remark \ref{indep} implies that there is an isomorphism  over $\br O[u^\pm ]$
\[
R_a: a^*(\br{\underline G}^*) \xrightarrow{\sim} \br{\underline G}^*
\]
that lifts $u\mapsto a\cdot u$. To check that this extends to an isomorphism
over $\br O[u]$ it is enough to check the statement for the corresponding parahoric group scheme over $\br F[[u]]$. This follows
by an argument as in the proof of \cite[Lemma 5.4]{ZhuPRConj}.
\end{proof}

\begin{remark}\label{depends}
Suppose that $G=G^*$ is quasi-split over $F$. Then, by Remark \ref{indep} (b), the extension $\underline G=\underline G^*$ over $O[u^{\pm }]$ is determined by $G$ as the unique, up to isomorphism, quasi-split group scheme that restricts to $G$ after $u\mapsto \pi$. However,
the restriction $\underline G^*\otimes_{O[u^\pm]}F$, by $u\mapsto \pi'$, where $\pi'=a\cdot \pi$ is another choice of uniformizer, is not necessarily isomorphic to $G$. 
For example, suppose $G=\Res_{L/\BQ_p}(\BG_m)$, with $L=\BQ_p(p^{1/2})$, $p$ odd. Suppose $\pi=p$. Then, 
\[
\underline G=\Res_{\BZ_p [u^\pm][X]/(X^2-u)/\BZ_p[u^\pm]}(\BG_m).
\]
Specializing this by $u\mapsto \pi'=-p$, gives $ \Res_{L'/\BQ_p}(\BG_m)$, with $L'=\BQ_p((-p)^{1/2})$
which is a different torus than $G$ if $p\equiv 3\, {\rm mod}\, 4$.

 Therefore, the extension $\underline G^*$ depends on both $G$ and $\pi$. When we need to be more precise, we will denote it by $\underline G^*_\pi$. By the above,  we have an isomorphism
\[
R_a^\natural: a^*(\underline G^*_\pi)\xrightarrow{ \sim } \underline G^*_{\pi'},
\]
where $a: \Spec(O[u^\pm])\to \Spec(O[u^\pm])$ is given by $u\mapsto a\cdot u$, which descends
$R_a$ above.
\end{remark}

 \subsection{Weyl groups and the admissible locus}\label{ss:WeylAdm}
 We continue with the set-up of the last subsection. 
  The group scheme $\breve{ \underline G}  $ admits a chain of tori by closed subgroup schemes  $\underline {\breve S}\subset \underline {\breve T}$ which extend $S$ and $T$ and correspond to $\breve {\underline S}^*$, $\breve{\underline T}^*$ via $\underline \psi$.
  These define maximal split, resp. maximal, tori in the fibers $\breve G=G\otimes_F\breve F$ and $\breve G'=G'\otimes_{\kappa_F((u))}k((u))$ of $\breve{ \underline G} $. 
 By the above constructions, we obtain  identifications of \emph{relative Weyl groups}, resp. \emph{Iwahori Weyl groups}, 
 \begin{equation}\label{indentweyl}
 W_0(\breve G, \breve T)=W_0(\breve G', \breve T'), \quad \tilde W(\breve G, \breve T)=\tilde W(\breve G', \breve T'), \end{equation}
\emph{cf.} \cite[\S 2]{R}.
Assume now that we have a  conjugacy class $\{\mu\}$ of a minuscule geometric cocharacter  of $G$, so
that $(G, \{\mu\}, K)$ is a local model triple over $F$. Then the above give identifications of $\{\mu\}$-\emph{admissible sets} in the Iwahori Weyl groups 
 \begin{equation}\label{indentweyl}
  \Adm(\{\mu\})=\Adm'(\{\mu\}) ,
 \end{equation}
\emph{cf.} \cite[\S 3]{R}.
 Denoting by $\breve K'$ the parahoric subgroup of $G'\big(k((u))\big)$ defined by $\CG'$, with corresponding group scheme $\breve \CG'$, we also obtain an identification of $\{\mu\}$-admissible subsets in the double coset spaces (\emph{cf.} \cite[\S 3]{R}), 
 \begin{equation}\label{indentadm}
 \Adm_{\breve K}(\{\mu\})=\Adm'_{\breve K'}(\{\mu\})\subset W_{\breve K}\backslash \tilde W/W_{\breve K}= W_{\breve K'}\backslash \tilde W'/W_{\breve K'} .
 \end{equation}
 We define a closed reduced subset inside the loop group flag variety $\CF'=L\breve G'/L^+\breve \CG'$ over $k$, as the reduced union 
 \begin{equation}\label{Aschubert}
 \CA_K(\underline G, \{\mu\} )=\bigcup _{w\in\Adm'_{\breve K'}(\{\mu\})} S_w . 
 \end{equation}
 Here $S_w$ denotes the $L^+\breve \CG'$-orbit corresponding to $w\in  W_{\breve K'}\backslash \tilde W'/W_{\breve K'}$. We note that, since $\{\mu\}$ is minuscule, the action of $L^+\breve \CG'$
 on $\CA_K(\underline G, \{\mu\} )$ factors through $\CG'\otimes_{\kappa_F[[u]]}k$. Via \eqref{identpara}, we obtain an action of $\CG\otimes_{O_F} k$ on $\CA_K(\underline G, \{\mu\} )$. 
 \begin{corollary}\label{unique}
  Up to isomorphism, the group $\breve G'$ over $k((u))$ and its parahoric subgroup $\breve K'$ are  independent of the choice of the uniformizer $\pi$ and of $\underline G$. The isomorphism can be chosen compatibly with the identification  \eqref{identpara}, and the identifications \eqref{indentweyl} of Weyl groups and \eqref{indentadm} of admissible sets.
 As a consequence, the affine partial flag variety $\CF'$ over $k$ and its subscheme $\CA_K(\underline G, \{\mu\} )$ with action of $\CG\otimes_{O_F} k$ is independent of the choice of the uniformizer $\pi$ and of $\underline G$.
 \end{corollary}
 \begin{proof}
 Follows from Proposition \ref{ind2}, its proof and the definition of the $\{\mu\}$-admissible set.
 \end{proof}

 \subsection{Descent}

We continue with the set-up of the previous subsection; we will   apply a form of  Weil-\'etale descent from $\breve O$ to $O$. The following result is not needed for the proof of Theorems \ref{main1Intro} and \ref{main2Intro} about local models with smooth or semi-stable reduction, see Remark \ref{simplerInd}. However, it is an important part of the argument for the independence result  of Theorem \ref{indepIntro}.

 \begin{proposition}\label{ind3}
 a) The group scheme $\underline \CG_x\otimes_{O[u]}O[[u]]$ depends, up to isomorphism, 
 only on $G$, the uniformizer $\pi$ and the $G_{\rm ad}(F)$-orbit of $x\in \CB(G, F)$. We denote it by $\underline\CG_{x, \pi}\otimes_{ O[u]}  O[[u]]$.
 
 b) If $\pi'=a\cdot \pi$ is another choice of a uniformizer with $a\in O^\times$, then there is an isomorphism
 of group schemes
 \[
R^\natural_a: a^*(\underline\CG_{x, \pi}\otimes_{ O[u]}  O[[u]])\xrightarrow{\sim} \underline\CG^*_{x, \pi'}\otimes_{ O[u]}  O[[u]]
 \]
 where $a: \Spec O[[u]]\to \Spec O[[u]]$ is given by $u\mapsto a\cdot u$.
 \end{proposition}
 
 \begin{proof}
 
 We first show (a). For this we fix the uniformizer $\pi$.
By Proposition \ref{ind2}, the base change $\underline \CG_{x}\otimes_{O[u]}\breve O[[u]]$  
depends only on $G$ and the $G_{\rm ad}(F)$-orbit of $x\in \CB(G, F)$. We will now use descent.
 By the construction, the group  $\underline \CG_x$ in \cite{PZ} is given by a ($\sigma$-semilinear) Weil descent datum
\[
\Int(\underline g)\cdot \sigma: \underline\CG^*_{x^*}\xrightarrow{\ \ } 
\underline\CG^*_{x^*}.
\]
Here $\underline  g\in \underline G^*_{\rm ad}(\breve O[u^\pm])$; this depends on various choices made
in \cite{PZ}. The action of $\sigma$ is with respect to the rational structure 
given by the $O[u^\pm]$-group $\underline G^*_\pi$; this depends on our fixed choice of $\pi$, see
Remark \ref{depends}. We start the proof by giving:

 \begin{lemma}\label{P}
 The  automorphism group  $\mathscr A^*={\rm Aut}(\underline \CG^*_{x^*}\otimes_{\breve O[u]}\breve O[[u]]) $ of the group scheme $\underline \CG^*_{x^*}\otimes_{\breve O[u]}\breve O[[u]]$  has the following properties:
 \smallskip
 
 i) It contains the normalizer $\mathscr N^*$ of $ \underline \CG^*_{x^*}(\breve O[[u]])$ in
 $\underline G^*_{\rm ad}(\breve O((u)))$.
 \smallskip
 
 ii) The homomorphism ${\mathscr A}^* \to {\rm Aut}( \CG^*_{x^*})$
given by $u\mapsto \pi$ is surjective. We have   
\[
\ker({\mathscr A}^* \to {\rm Aut}( \CG^*_{x^*}))=\ker(\underline \CG^*_{{\rm ad}, x^*}(\breve O[[u]])\xrightarrow{u\to\pi} \CG^*_{{\rm ad}, x^*}(\breve O))
\]
and this kernel is pro-unipotent.
 \end{lemma}

\begin{proof} Let us first study ${\rm Aut}( \CG^*_{x^*})  $: 
Passing to the generic fiber gives an injection
\[
{\rm Aut}( \CG^*_{x^*}) \subset {\rm Aut}(\breve G^*).
\]
There is also (\cite[Prop. 7.2.11]{Conrad}) a (split) exact sequence 
\[
1\to G^*_{\rm ad}(\breve F)\to {\rm Aut}(\breve G^*) \to {\rm Out}(\breve G^*)\to 1.
\]
This gives
\[
1\to G^*_{\rm ad}(\breve F)_{x^*}\to {\rm Aut}(\breve G^*)_{ x^*}={\rm Aut}( \CG^*_{x^*})\to {\rm Out}(\breve G^*)_{x^*}\to 1
\]
where the subscript $x^*$ denotes the subgroup that fixes $x^*\in \CB(G^*, \breve F)$. 

Notice here that $G^*_{\rm ad}(\breve F)_{x^*}$ is the normalizer  in $G^*_{\rm ad}(\breve F)$ of the parahoric subgroup  $\CG^*_{x^*}(\br O)=G^*(\breve F)^0_{x^*}$. (Indeed, by \cite[5.1.39]{BTII}, the normalizer of the stabilizer of any facet 
in the Bruhat-Tits building has to also stabilize the facet; this last statement easily follows from that.) 
We also have
\[
1\to G^{*}_{\rm ad}(\breve F)^0_{x^*}\to G^*_{\rm ad}(\breve F)_{x^*}\ \to \Delta_{x^*}\to 1
\]
where $\Delta_{x^*}$ is the finite abelian group given as the group of connected components of the ``stabilizer of $x^*$" Bruhat-Tits group scheme for $G^*_{\rm ad}$ over $\breve O$.

Similarly, we have an injection ${\mathscr A}^*\subset {\rm Aut}(\underline {\br G}^*)$. The quasi-split $\underline G^*$
carries the pinning 
$(\underline T^*, \underline B^*, \underline e^*)$ and we can use this to identify ${\rm Out}(\breve {\underline G}^*)=
{\rm Out}(\breve G^*)$ with a subgroup of the group $\Xi_H$ of ``graph" automorphisms. By \cite[Prop. 7.2.11]{Conrad}, we have
\begin{equation}\label{autUn}
{\rm Aut}(\underline {\br G}^*)=\underline {G}^*_{\rm ad}(\br O((u)))\rtimes {\rm Out}(\breve {\underline G}^*).
\end{equation}

 We first show (i), i.e., that every $g\in \mathscr N^*\subset \underline {G}^*_{\rm ad}(\br O((u)))$ naturally induces an automorphism $\Int(g)$ of $\underline \CG^*\otimes_{\breve O[u]}\breve O[[u]]$. (For simplicity, we omit the subscript $x^*$ below.)  The adjoint action of $g\in {\mathscr N}^*$ gives an ind-group scheme homomorphism $\Int(g): L\underline G^* \to L\underline G^*$ which preserves $L^+\underline \CG^*(\breve O)$.
Using the fact $L^+\underline \CG^*$ is pro-algebraic and formally smooth over $\breve O$, we can easily see that the set of points $L^+\underline \CG^*(\breve F)$ with $\breve F$ as residue 
field is dense in $L^+\underline \CG^*$. Since $L^+\underline \CG^*$ is a reduced closed subscheme 
of the ind-scheme $L\underline \CG^*=L\underline G^*$ over $\breve O$, it follows that $g$ induces a group scheme homomorphism
\[
\Int(g): L^+\underline \CG^*\to L^+\underline \CG^*.
\]
In particular,  $g$ also normalizes $L^+\underline \CG^*(\breve F)=\CG^*(\breve F[[u]])$.
Since $\CG^*\otimes_{\br O[u]}\breve F((u)) $ is quasi-split and residually split, the $\br F$-valued points are dense in the fiber $\CG^*\otimes_{\br O[u]}\breve F $
over $u=0$. Hence, we obtain by \cite[1.7.2]{BTII} that $\Int(g)$ induces an  automorphism of the group scheme 
$\CG^*\otimes_{\br O[u]}\breve F[[u]]$. Since $\CG^*$ is smooth over $\br O[[u]]$ and $\Int(g)$ gives an automorphism
of  $\CG^*\otimes_{\br O[u]}\breve O((u))$, we see that $\Int(g)$ extends to an automorphism of $\CG^*\otimes_{\br O[u]}\breve O[[u]]$ as desired. This proves (i).

Let us show that ${\mathscr A}^*$ satisfies (ii). Sending $u\mapsto \pi$ gives a homomorphism
\[
{\mathscr A}^*\to  {\rm Aut}( \CG^*_{x^*}).
\]
This restricts to ${\mathscr N}^*\to G^*_{\rm ad}(\breve F)_{x^*}$: To see   this   we
use that $L^+\underline \CG^*(\breve O)\to \CG^*(\breve O)=G^* (\breve F)^0_{x^*}$ given by $u\mapsto \pi$ is surjective (by smoothness and Hensel's lemma) and that $G^*_{\rm ad}(\breve F)_{x^*}$ is the normalizer of $G^*(\breve F)^0_{x^*}$ in $G^*_{\rm ad}(\breve F)$. 
We obtain a commutative diagram with exact rows
\begin{equation}
\begin{matrix}
 1& \to &{\mathscr N}^* \to &{\mathscr A}^*&\to  {\rm Out}(\underline {\br G}^*)_{x^*}\to 1\cr
&&\downarrow\ \ \  \ \ \  & \downarrow &\downarrow\cr
1& \to &G^*_{\rm ad}(\breve F)_{x^*}\to  &{\rm Aut}( \CG^*_{x^*})&\to {\rm Out}(\breve G^*)_{x^*}\to 1.
\end{matrix}
\end{equation}
We will show that the left vertical arrow is a surjection with kernel equal to ${\mathscr K}^*:=\ker(\underline \CG^*_{{\rm ad}, x^*}(\breve O[[u]])\xrightarrow{u\to\pi} \CG^*_{{\rm ad}, x^*}(\breve O))$ and that the right vertical arrow is an isomorphism.
This would imply part (ii).

The subgroup $\underline \CG^*_{{\rm ad}, x^*}(\breve O[[u]])\subset 
\underline G_{\rm ad}^*(\breve O((u)))$  is contained in ${\mathscr N}^*$.
Mapping $u\mapsto \pi$ followed by taking connected component gives a homomorphism 
\[
\delta: {\mathscr N}^*\to   G^*_{\rm ad}(\breve F)_{x^*}\xrightarrow{\ } \Delta_{  x^*}.
\]
 We will   show that the sequence
\begin{equation}\label{exact1}
1\to  \underline \CG^*_{{\rm ad}, x^*}(\breve O[[u]])\to {\mathscr N}^*\xrightarrow{\ \delta\ } \Delta_{x^*}\to 1
\end{equation}
is exact.  Since  $\underline \CG^*_{{\rm ad}, x^*}(\breve O[[u]])\xrightarrow{u\mapsto \pi} \CG^*_{{\rm ad}, x^*}(\breve O)=G^*_{\rm ad}(\breve F)^0_{x^*}$ is surjective (by smoothness and Hensel's lemma) this would show that $u\mapsto \pi$ gives a surjective
\[
{\mathscr N}^*\to G^*_{\rm ad}(\breve F)_{x^*}\to 1
\]
 with kernel equal to ${\mathscr K}^*:=\ker(\underline \CG^*_{{\rm ad}, x^*}(\breve O[[u]])\xrightarrow{u\to\pi} \CG^*_{{\rm ad}, x^*}(\breve O))$.

Let us show the exactness of (\ref{exact1}). The subgroup $ \underline \CG^*_{{\rm ad}, x^*}(\breve O[[u]])$  lies in the kernel of $\delta$ and we can see that it is actually equal to that kernel:  Let $g\in {\mathscr N}^*$ with $\delta(g)=1$. Since $g$ also normalizes $ \CG^*(\breve F[[u]])$, we see
  as above, that $g$ lies in $\underline G^*_{\rm ad}(\breve F ((u)))_{x^*}$.
Using the identification of apartments (\ref{ident1}) we now see 
that since $\delta(g)=1$, $g$ is actually in the connected stabilizer $\underline G^*_{\rm ad}(\breve F((u)))_{x^*}^0=\underline \CG^*_{{\rm ad}, x^*}(\breve F[[u]])$. Since $g$ is also in $\underline G^*_{\rm ad}(\breve O((u)))$, we have
\[
g\in \underline \CG^*_{{\rm ad}, x^*}(\breve F[[u]])\cap \underline \CG^*_{{\rm ad}, x^*}(\breve O((u)))=\underline \CG^*_{{\rm ad}, x^*}(\breve O[[u]]).
\]
Therefore, ${\rm ker}(\delta)=  \underline \CG^*_{{\rm ad}, x^*}(\breve O[[u]])$. 
It remains to show that $\delta$ is surjective. By \cite[Proposition 4.6.28 (ii)]{BTII}, for each $y\in \Delta_{x^*}$, there is an element $n\in N_{\rm ad}(\breve F)$ that fixes $x^*$ in the building so that
$\delta (n)=y$. By the identification of the apartments (\ref{identApt}), we can lift $n$ to $\underline n\in \underline N_{\rm ad}(\breve O((u)))$ which fixes the point $x^*$ considered in the building over $\breve F((u))$. Then $\underline n$ normalizes $L^+\underline \CG^*(\breve F)\cap \underline G^* (\breve O((u)))=L^+\underline \CG^*(\breve O)$ so $\underline n$ is in ${\mathscr N}^*$. 

It remains to show that ${\rm Out}(\underline {\br G}^*)_{x^*}\to {\rm Out}( {\br G}^*)_{x^*}$ given by $u\mapsto \pi$ is an isomorphism. The corresponding map
${\rm Out}(\underline {\br G}^*) \to {\rm Out}( {\br G}^*) $ is an isomorphism 
by the construction of $\underline {\br G}^*$ from ${\br G}^*$. Hence, it is enough to show that
${\rm Out}(\underline {\br G}^*)_{x^*}\to {\rm Out}( {\br G}^*)_{x^*}$ is surjective.
By definition, $\gamma\in {\rm Out}( {\br G}^*)_{x^*}$ is given by an automorphism of $\br G^*$ preserving the pinning $(\br T^*, \br B^*, \br e^*)$, 
such that $\gamma(x^*)=\Int(g)( x^*)$, for some $g\in G^*_{\rm ad}(\br F)$. 
Since $\gamma(x^*)$ and $x^*$ both lie in the apartment for $\br S^*\subset \br T^*$,
this implies that $\gamma(x^*)=\Int(n)( x^*)$, for some $N^*_{\rm ad}(\br F)$.
As above, we can lift $n$ to $\underline n\in N^*_{\rm ad}(\br O((u)))$. Using the identification of apartments (\ref{identApt}) we see that $\gamma$ is in ${\rm Out}(\underline {\br G}^*)_{x^*}$.
\end{proof}

We can now resume the proof of Proposition \ref{ind3}.
We will show that $\underline \CG_x\otimes_{O[u]}{O[[u]]}$ 
 is independent, up to isomorphism, of additional choices. Suppose as above that 
 $\underline  g'\in \underline G^*_{\rm ad}(\breve O[u^\pm])$ is a second cocycle giving a 
 group scheme $\underline \CG'_x$; then $\underline \CG'_x$ is a form of $\underline \CG_x$.
The twisting is obtained by the image of the cocycle given by 
\[
\underline c=\underline g'\cdot \underline g^{-1} \in \underline G^*_{\rm ad}(\breve O[u^\pm]).
\]
(This is a cocycle for the twisted $\sigma$-action on $\underline G^*_{\rm ad}(\breve O[u^\pm])$ given by $\Int(\underline g)$.) Notice that the restriction of  $\underline c$ along $u=\pi$ preserves $x^*$. Hence, $\underline c$ also preserves $x^*$ considered as a point in the building over $\br F((u))$.  It follows that $\underline c$ lies in the
 normalizer of the parahoric $\underline \CG^*_{x^* }(\br F[[u]])$. Using $\breve O((u))\cap \breve F[[u]]=\breve O[[u]]$,
 we see that $\underline c$ lies in the normalizer $\mathscr N^*$ of $  \underline \CG^*_{x^*}(\breve O[[u]])$ and it gives a cocycle for the twisted $\sigma$-action. 
 The isomorphism class of  the form $\underline \CG'_x\otimes_{O[u]}O[[u]]$ is determined by the   class $ [\underline c]$ in
$
  {\rm H}^1(\hat\BZ, 
{\mathscr A}).
 $
Here ${\mathscr A}={\rm Aut}(\underline \CG_{x}\otimes_{\breve O[u]}\breve O[[u]])$ which is ${\mathscr A}^*$ but with the twisted $\sigma$-action.
By  Lemma \ref{P} (b), ${\mathscr K}^*$ and therefore also the kernel ${\mathscr K}=\ker( {\mathscr A}\to
{\rm Aut}( \CG_{x} ))$ is pro-unipotent. Using this, 
 a standard argument as in the proof of Lemmas 1 and 2, p. 690, of \cite{BTIII},
gives that ${\rm H}^1( \hat\BZ, {\mathscr K})=0$. 
Since the specialization of the form  $\underline \CG'_x$ at $u=\pi$ is isomorphic to $\CG_x$, the image of the class
$\underline c$ in $ {\rm H}^1(\hat\BZ, 
{\rm Aut}( \CG_{x} ))$ is trivial. Hence, by the exact sequence for cohomology, the class $ [\underline c]$ in 
${\rm H}^1(\hat\BZ, {\mathscr A})$ is trivial.  Therefore, we obtain
$\underline \CG'_x\otimes_{ O[u]} O[[u]]  \simeq \underline \CG_x\otimes_{O[u]}O[[u]] $, where in both,
the choice of $\pi$ remains the same. This proves part (a).

To prove part (b), suppose that $\pi'=a\cdot\pi$, $a\in O^\times$, is another choice of uniformizer.
By Proposition \ref{ind2} (c), the group scheme $\underline\CG^*_{x^*}\otimes_{\breve O[u]}\breve O[[u]]$ supports an isomorphism
\[
R_a: a^*(\underline\CG^*_{x^*}\otimes_{\breve O[u]}\breve O[[u]])\xrightarrow{\sim} \underline\CG^*_{x^*}\otimes_{\breve O[u]}\breve O[[u]].
\]
We would like to show that $R_a$ descends to an isomorphism
$
R^\natural_a: a^*(\underline\CG_{x, \pi}\otimes_{ O[u]}O[[u]])\xrightarrow{\sim} \underline\CG_{x, \pi'}\otimes_{O[u]} O[[u]]$. Consider the descent datum $\Phi:=\Int(\underline g)\cdot\sigma$ for $\underline\CG_{x, \pi}$ and its ``rotation"
given as $R_a(\Phi):=R_a (a^*\Int(\underline g))\sigma(R_a)^{-1}\cdot \sigma$ for $\underline\CG_{x, \pi'}$. 
Consider also a descent datum $\Phi':=\Int(\underline g')\cdot\sigma$ for $\underline\CG_{x, \pi'}$. It is enough to show that $\Phi'$ and $R_a(\Phi)$ are
cohomologous, i.e., that there is an automorphism $h$ of  $\underline\CG^*_{x^*}\otimes_{\br O[u]}\br O[[u]]$
such that $h^{-1}R_a(\Phi) = \Phi'\cdot \sigma(h)^{-1} $. Then we can set $R_a^\natural=h^{-1}R_a$ which descends.
To show the existence of $h$, note that $R_a$ is the identity on the maximal reductive quotient 
of the fiber of $\underline\CG^*_{x^*}\otimes_{\breve O[u]}\breve O[[u]]$ over the point $(u,\pi)$.
We have $\underline\CG_{x, \pi}\simeq \underline\CG_{x, \pi'}$ modulo $(u,\pi)$ since they both
are isomorphic to $\CG_x$ modulo $\pi$.  Hence,
$\Phi'$ and $R_a(\Phi)$ are cohomologous when considered modulo a (connected) pro-unipotent group. An argument similar to the one in the proof of part (a) above then shows  the result.
 \end{proof}

 \subsection{Pappas-Zhu local models}
 
Let $(G, \{\mu\}, K)$ be a local model triple over $F$ such that $G$ splits over a tamely ramified extension of $F$. Again we set $O=O_F$.

In \cite{PZ}, there is a construction of a ``local model" $M_{\CG, \mu}$.  The Pappas-Zhu local model $M_{\CG, \mu}$ is a flat projective $O_E$-scheme equipped with an action of $\CG_{O_E}$ such that its generic fiber is $G_E$-equivariantly isomorphic to $X_{\{\mu\}}$. By definition, $M_{\CG, \mu}$ is
the Zariski closure of $X_{\{\mu\}}\subset  {\rm Gr}_{\underline \CG}\otimes_{O[u]}E$
in ${\rm Gr}_{\underline \CG, O}\otimes_{O}O_E$, where ${\rm Gr}_{\underline \CG}$
is the Drinfeld-Beilinson (global) Grassmannian over $ O[u]$ for $\underline \CG$ and  $
{\rm Gr}_{\underline \CG, O}={\rm Gr}_{\underline \CG}\otimes_{O[u]}O$ is its base change to $O$  by $u\mapsto \pi$. A priori, $M_{\CG, \mu}$ depends on the group scheme $\underline \CG$ over $O[u]$ and the choice of the uniformizer $\pi$.

\begin{theorem}\label{indLM}
The $\CG_{O_E}$-scheme $M_{\CG, \mu}$ over $O_E$, depends, up to equivariant isomorphism, only on the local model triple $(G,\{\mu\}, K)$.
\end{theorem}  

\begin{proof}
We first observe that $M_{\CG, \mu}$ can be constructed starting only with $\{\mu\}$, the base change $\underline \CG\otimes_{O[u]}O[[u]]$, and the ideal $(u-\pi)$ in $O[[u]]$. Indeed, we first see that ${\rm Gr}_{\underline \CG, O}$
only depends on $\underline \CG\otimes_{O[u]}O[[u]]$, and the ideal $(u-\pi)$ in $O[[u]]$.
Set $t=u-\pi$. The base change ${\rm Gr}_{\underline \CG, O}={\rm Gr}_{\underline \CG}\otimes_{O[u]}O$ by $u\mapsto \pi$  has $R$-valued points for an $O$-algebra $R$ given by the set of isomorphism classes of $\underline \CG$-torsors over $R[t]$ with a trivialization over $R[t, 1/t]$. By the Beauville-Laszlo lemma (in the more general form given for example in \cite[Lemma 6.1, Prop. 6.2]{PZ}), this set is in bijection with the set of isomorphism classes of
$\underline \CG\otimes_{O[u]}{R[[t]]}$-torsors over $R[[t]]=R[[u]]$ together with a trivialization over $R((t))=R[[u]][(u-\pi)^{-1}]$.  To complete the proof we use Proposition \ref{ind3}. It gives that $\underline \CG\otimes_{O[u]}{O[[u]]}$ only depends on the local model triple and $\pi$, hence ${\rm Gr}_{\underline \CG, O}$ only depends on the local model triple and $\pi$; for clarity, denote it by ${\rm Gr}_{\underline \CG, O, \pi}$. Part (b) of Proposition \ref{ind3} with the above then gives
that pulling back of torsors along $a: \Spec R[[u]]\to \Spec R[[u]]$, given by $u\mapsto a\cdot u$, gives an isomorphism
\[
{\rm Gr}_{\underline \CG, O, \pi}\xrightarrow{\ \sim\ } {\rm Gr}_{\underline \CG, O, \pi'}.
\]
Hence, by the above  ${\rm Gr}_{\underline \CG, O}$ depends, up to equivariant isomorphism,  only on $G$ and $K$. The result then follows from the definition of $M_{\CG,\mu}$.
\end{proof}

\begin{remark}\label{simplerInd}  We can obtain directly the independence of the base change $M_{\CG, \mu}\otimes_{O_E}\br O_E$  
via the same argument as above, by using the simpler  Proposition \ref{ind2} in place of Proposition \ref{ind3}.
\end{remark}

 \subsection{Local models: A variant of the Pappas-Zhu local models}\label{ss:PZmodif}
 It appears that the Pappas-Zhu local models $M_{\CG, \mu}$ are  not  well behaved when the characteristic $p$ divides the order of $\pi_1(G_{\rm der})$. For example, in this case, their special fiber is 
sometimes not reduced (see \cite{HR}, \cite{HR2}). Motivated by an insight of Scholze, we employ $z$-extensions to slightly modify the definition of  \cite{PZ}. Suppose that $(G,\{\mu\}, K)$ is an LM triple over $F$ such that $G$ splits over a tame extension of $F$. Choose a $z$-extension over $F$
\begin{equation}\label{zExt}
1\to T\to \tilde G\to G_{\rm ad}\to 1.
\end{equation}
In other words, $\tilde  G$ is a central extension of $G_{\rm ad}$ by a strictly  induced torus $T$ and the reductive group $\tilde G$ has simply connected derived group,
$\tilde G_{\rm der}=G_{\rm sc}$
(see for example, \cite[Prop. 3.1]{MS}). (Here, we say that the torus $T$ over $F$ is strictly induced 
if it splits over a finite Galois extension $F'/F$ and the cocharacter group $X_*(T)$ is a free $\BZ[{\rm Gal}(F'/F)]$-module.) We can assume that $\tilde  G$, and then also $T$, split over 
a tamely ramified extension of $F$.
By  \cite[Applic. 3.4]{MS}, we can choose a cocharacter $\tilde\mu$ of $\tilde G$ which lifts $\mu_{\rm ad}$ and which is such that the reflex field $\tilde  E$ of $\{\tilde \mu\}$ 
is equal to the reflex field $E_{\rm ad}$ of $\{\mu_{\rm ad}\}$. Let $\tilde K$ be the unique parahoric subgroup  of $\tilde G$ which lifts $K_{\rm ad}$. Then the corresponding group scheme
$\tilde \CG$ fits in a fppf exact sequence   of group schemes
over $O_F$,
\[
1\to \CT\to \tilde\CG\to\CG_{\rm ad}\to 1,
 \]
  which extends the $z$-extension above, comp. \cite[Prop. 1.1.4]{KP}. We  set
\[
\Mloc:=M_{\tilde\CG, \tilde\mu}\otimes_{O_{E_{\rm ad}}}O_E
\]
which is, again,  a
  flat projective $O_E$-scheme equipped with an action of $\CG_{O_E}$ (factoring through $\CG_{{\rm ad}, {O_E}}$)  with generic fiber  $G_E$-equivariantly isomorphic to $X_{\{\mu\}}$. Indeed, the action of $\tilde\CG_{O_E}$ on $M_{\tilde\CG, \tilde\mu}$ factors through the quotient $\CG_{{\rm ad}, O_E}=\tilde\CG_{O_E}/\CT_{O_E}$ (because it does so on the generic fiber). Since $G\to G_{\rm ad}$ extends to a group scheme homomorphism $\CG\to \CG_{\rm ad}$, we also obtain an action of $\CG_{O_E}$ on $\Mloc$.

\begin{remark}\label{compPZ}
1) By \cite[Thm. 9.1]{PZ}, $M_{\tilde\CG, \tilde\mu}$  has reduced special fiber. Therefore, the same is true 
for the base change $\Mloc=M_{\tilde\CG, \tilde\mu}\otimes_{O_{E_{\rm ad}}}O_E$. 
By \cite[Prop. 9.2]{PZ}, it follows that $\Mloc$ is a normal scheme.

 2) If $p$ does not divide the order of $\pi_1(G_{\rm der})$ then we have an equivariant isomorphism $M_{\tilde \CG,\tilde \mu}\otimes_{O_{E_{\rm ad}}}O_E\simeq M_{\CG, \mu}$, cf. \cite[Prop. 2.2.7]{KP}\footnote{In loc.~cit.  $F=\BQ_p$, but the result holds for general $F$.}.
Therefore, in this case 
\[
\Mloc\simeq M_{ \CG,  \mu}.
\]
 
3) Suppose that $\tilde   G'\to G_{\rm ad}$ is another choice of a $z$-extension as in (\ref{zExt}) and let  $\tilde \mu'$ be a cocharacter that also lifts $\mu_{\rm ad}$ with reflex field $E=E_{\rm ad}$. Then the fibered product $H=\tilde  G\times_{G_{\rm ad}} \tilde  G'\to G$ is also a similar $z$-extension
with kernel the direct product $T\times T'$ of the kernels of $\tilde   G\to G_{\rm ad}$  and $\tilde   G'\to G_{\rm ad}$.  We have a cocharacter $\mu_H=(\tilde  \mu, \tilde \mu')$ which also has reflex field 
$E$. The  parahoric group scheme for $ H$ corresponding to $\CG$ is  ${\CH}={\tilde  \CG}\times_{\CG_{\rm ad}}  \tilde  \CG'$. We obtain $M_{\CH, \{\mu_H \}}$ as in \cite{PZ}.  By  construction, we obtain
\[
M_{\CH, \{ \mu_H\} }\xrightarrow{\sim} M_{\tilde \CG, \{\tilde \mu\}}, \quad M_{\CH, \{\mu_H\}}\xrightarrow{\sim} M_{\tilde \CG', \{\tilde \mu\}},
 \] 
  both $\CH_{O_E}$-equivariant isomorphisms. Hence, we obtain an isomorphism $M_{\tilde \CG, \{\tilde \mu\}}\xrightarrow{\sim} M_{\tilde \CG', \tilde \{\mu\}}$
which is $\CG_{{\rm ad}, O_E}$-equivariant. As a result, $\Mloc$ is independent of the choice of the $z$-extension.  We can now easily deduce from Theorem \ref{ind3}, that  $\Mloc$ also only depends on 
the local model triple $(G, \{\mu\}, K)$.

 4) (Suggested by the referee) In fact, one can give an alternative proof that $\Mloc$ is independent 
(up to isomorphism) of the choice of $z$-extension, by noting 
that it can be identified with the normalization of $M_{\CG_{\rm ad},  \{\mu_{\rm ad}\}}\otimes_{O_{E_{\rm ad}}}O_E$. Indeed, by (1) above, $\Mloc$ is normal and, by its construction, it affords a map to $M_{\CG_{\rm ad},  \{\mu_{\rm ad}\}}\otimes_{O_{E_{\rm ad}}}O_E$ which is finite and birational. 
\end{remark}

\begin{definition}\label{deftrueLM}
The projective flat $O_E$-scheme  $\Mloc$ with its  $\CG_{O_E}$-action  is called the \emph{local model} of the LM triple $(G, \{\mu\}, K)$. 
\end{definition}
\begin{theorem}\label{2.9}
The geometric special fiber $\Mloc\otimes_{O_E}k $ is reduced  and is
$\CG\otimes_{O_F}k$-equivariantly  isomorphic to 
$\CA_{\tilde K}(\tilde G, \{\tilde \mu\} )$.
\end{theorem}

\begin{proof}
This follows from the construction and \cite[Thm. 9.1, Thm. 9.3]{PZ}.  
\end{proof}

Note that this implies that the reduced $k$-scheme $\CA_{\tilde K}(\tilde G, \{\tilde\mu\})$ is independent of the choice of $z$-extension and only depends on $(G,\{\mu\}, K)$. 
(This fact can be also seen more directly using Corollary \ref{unique}
and \cite[\S 6]{PRTwisted}.) We call this \emph{  the $\mu$-admissible locus} of the local model triple $(G, \{\mu\}, K)$ and denote it by ${\mathfrak A}_K(G,\{\mu\})$. 

\begin{remark} It follows from \cite[6.a, 6.b]{PRTwisted}  that $\breve {G}'\to \breve G'_{\rm ad}$ and $\breve {\tilde G}'\to \breve G'_{\rm ad}$ induce equivariant
morphisms
\[
\CA_K(G, \{ \mu\} )\to \CA_{K_{\rm ad}}(G_{\rm ad}, \{\mu_{\rm ad}\} ),\quad \CA_{\tilde K}(\tilde G, \{\tilde \mu\} )\to \CA_{K_{\rm ad}}(G_{\rm ad}, \{\mu_{\rm ad}\})
\]
which both induce bijections on $k$-points. As a result, we have  equivariant bijections
\[
{\mathfrak A}_K(G,\{\mu\})(k)=\CA_K(G, \{ \mu\} )(k)=\CA_{K_{\rm ad}}(G_{\rm ad}, \{\mu_{\rm ad}\} )(k).
\]
\end{remark}
 
The following conjecture would characterize the local model $\Mloc$ uniquely. 

\begin{conjecture}\label{uniqLM}
Up to equivariant isomorphism, there exists a unique flat projective $O_E$-scheme $\BM$ equipped with an action of $\CG_{O_E}$ and the following properties. 

\begin{altenumerate}
\item[(a)]\ Its generic fiber is $G_E$-equivariantly isomorphic to $X_{\{\mu\}}$.

\item[(b)]\ Its special fiber is reduced and there is a $\CG\otimes_{O_F}k$-equivariant isomorphism of $k$-schemes
\begin{equation*}
\BM\otimes_{O_E}k\simeq  {\mathfrak A}_K( G, \{\mu\}).
\end{equation*}
\end{altenumerate}

\end{conjecture}

The local models constructed above have the following properties. 

 \begin{proposition}\label{proofLM}  The following  hold.
\begin{altenumerate}
\item If $K$ is hyperspecial, then $\Mloc$ is smooth over $O_E$. 
\smallskip

\item If $F'/F$ is a finite unramified extension, then 
\begin{equation}
\Mloc\otimes_{O_E}O_{E'}\xrightarrow{\sim }\BM^\loc_{K'}(G\otimes_FF', \{\mu\otimes_FF'\}) .
\end{equation}
Note that  here the reflex field $E'$ of $(G\otimes_FF', \{\mu\otimes_FF'\})$ is the join of $E$ and $F'$. 
\smallskip

\item If $(G, \{\mu\}, K)=(G_1, \{\mu_1\}, K_1)\times (G_2, \{\mu_2\}, K_2)$, then
\begin{equation}
\Mloc=\big(\BM^\loc_{K_1}(G_1, \{\mu_1\})\otimes_{O_{E_1}}O_E\big)\times\big(\BM^\loc_{K_2}(G_2, \{\mu_2\})\otimes_{O_{E_2}}O_E\big) .
\end{equation}
Note that  here the reflex field $E$ of $(G, \{\mu\})$ is the join of the reflex fields $E_1$ and $E_2$.
\smallskip

\item If  $\phi: (G, \{\mu\}, K)\to (G', \{\mu'\}, K')$ is a morphism of local model triples such that $\phi: G\to G'$ gives a central extension of $G'$, there is a $\CG_{O_E}$-equivariant isomorphism
\begin{equation}
\Mloc\xrightarrow{\sim }\BM^\loc_{K'}(G', \{\mu'\})\otimes_{O_{E'}}O_{E} .
\end{equation}
 
\end{altenumerate}

\end{proposition}

\begin{proof}
 When $K$ is hyperspecial, we can choose the extension $\underline{\tilde \CG}$ over $O_F[u]$ to be reductive; then $\Mloc$ is smooth as required in property (i). By choosing the extension $\underline{\tilde \CG'}=\underline{\tilde \CG}\otimes_{O_F[u]}O_{F'}[u]$, we easily obtain (ii). For (iii), we choose the extension $\underline{\tilde\CG}=\underline{\tilde\CG}_1\times \underline{\tilde\CG}_2$ over $O_F[u]$. Finally,  (iv) follows by the construction since $G_{\rm ad}=G'_{\rm ad}$.
\end{proof}
 
\subsection{Scholze local models}  
Under special circumstances, we can relate the local models above to Scholze local models and give in this way a characterization of them different from Conjecture \ref{uniqLM}. In particular, this gives a different way of proving the independence of all choices in the construction of local models. Recall Scholze's conjecture \cite[Conj. 21.4.1]{Schber} that there exists a flat projective $O_E$-scheme $\BM^{\rm loc, flat}_{\CG, \mu}$ with generic fiber $X_{\{\mu\}}$ and reduced special fiber and with an equivariant closed immersion of the associated diamond, $\BM^{{\rm loc, flat}, \diamond}_{\CG, \mu}\hookrightarrow {\rm Gr}_{\CG, {\rm Spd} O_E}$. Scholze proves that  $\BM^{\rm loc, flat}_{\CG, \mu}$ is unique if it exists, cf. \cite[Prop. 18.3.1]{Schber}. Note that Scholze does not make the hypothesis that $G$ split over a tame extension. We are going to exhibit a class of LM triples $(G, \{\mu\}, K)$ (with $G$ split over a tame extension) such that the  local models $\Mloc$ defined above  satisfy Scholze's conjecture. 

We will say that a pair  $(G,\{\mu\})$, consisting  of a reductive group over $F$ and a geometric conjugacy class of  minuscule coweights  is of \emph{abelian type} when there is a similar pair $(G_1, \{\mu_1\})$ with $E_1\subset E\breve F$ and with a central isogeny $\phi: G_{1,\rm der}\to G_{\rm der}$ which induces an isomorphism $(G_{1, \rm ad}, \{\mu_{1, \rm ad}\})\simeq (G_{\rm ad}, \{\mu_{\rm ad}\})$ and is such that there exists a faithful minuscule representation $\rho_1: G_1\hookrightarrow {\rm GL}_n$ over $F$ such that $\rho_1\circ \mu_1$ is a minuscule cocharacter $\mu_d$ of ${\rm GL}_n$. 
Here by a minuscule representation we mean a direct sum of irreducible minuscules (i.e., with all weights
conjugate by the Weyl group.) In this case, we call such a pair $(G_1, \{\mu_1\})$ a \emph{realization} of the pair $(G, \{\mu\})$ of abelian type. 

\begin{theorem}\label{SchLM}
Let $(G,\{\mu\}, K)$ be a LM triple over $F$ such that $G$ splits over a tame extension of $F$,  for which there is an unramified finite extension $F'/F$
such that the base change $(G, \{\mu\})\otimes_F F'$ is of abelian type, with realization $(G_1, \{\mu_1\})$ such that $p\nmid  |\pi_1(G_{1,\rm der})|$. Then the local
model $\Mloc$ defined above satisfies Scholze's conjecture
\cite[Conj. 21.4.1]{Schber}. \end{theorem}

\begin{proof}
We already checked that the flat projective scheme $\Mloc$ has reduced special fiber. To show the conjecture it remains to show that the associated diamond $\Mloc^\diamond$ over ${\rm Spd}(O_E)$ embeds via an equivariant closed immersion in ${\mathrm {Gr}}_{\CG, {\rm Spd}(O_E)}$ such that its generic fiber identifies with $X_{\{\mu\}}^\diamond$.

 Using \'etale descent along $F'/F$ 
and property (ii) of Proposition \ref{proofLM},
we see that it is enough to show the conjecture for $(G, \{\mu\})\otimes_F F'$; so, we can assume that
$(G,\{\mu\})$ is of abelian type to begin with. Let $(G_1, \{\mu_1\})$ be a realization. In fact, we can also assume that $E_1\subset E$. Observe that by using $\phi$, we obtain a parahoric subgroup $K_1$
of $G_1$ which corresponds to $K$.
By \cite[Prop. 1.3.3]{KP}, $\rho_1: G_1\hookrightarrow {\rm GL}_n$ extends to a closed immersion
\[
\rho_1: \CG'_1\hookrightarrow {\mathcal {GL}} ,
\]
where $\CG'_1$ is the stabilizer (possibly non connected) of a point in the building of $G_1(F)$
that corresponds to $K_1$ and where ${\mathcal {GL}}$ is a certain parahoric group scheme for 
${\rm GL}_n$. In fact, by replacing $\rho_1$ by a direct sum $\rho_1^{\oplus m}$ we can assume that
${\mathcal {GL}}$ is ${\rm GL}_n$ over $O_F$; we will do this in the rest of the proof.  By \cite[Prop. 21.4.3]{Schber}, ${\mathrm {Gr}}_{\CG'_1, {\rm Spd}(O_E), \mu_1}={\mathrm {Gr}}_{\CG_1, {\rm Spd}(O_E), \mu_1}$, where $\CG_1=(\CG'_1)^\circ$.
This gives a closed immersion ${\mathrm {Gr}}_{\CG_1, {\rm Spd}(O_E), \mu_1}\hookrightarrow  {\mathrm {Gr}}_{{\mathcal {GL}}, {\rm Spd}(O_E)}$. By \cite[Prop. 2.3.7]{KP}, $\rho_1: \CG'_1\hookrightarrow {\mathcal {GL}}$ induces
\[
M_{\CG_1, \{\mu_1\}}\hookrightarrow (M_{{\rm GL}_n, \{\mu_d\}})_{O_{E_1}}={\rm Gr}(d, n)_{O_{E_1}} ,
\]
which is also an equivariant closed immersion. (Here the local model $M_{{\rm GL}_n, \{\mu_d\}}$ is the Grassmannian 
${\rm Gr}(d, n)$ over $O_F$.) By the assumption $p\nmid |\pi_1(G_{1,\rm der})|$, Remark \ref{compPZ} above gives that $\Mloc\simeq M_{ \CG_1,  \mu_1}\otimes_{O_{E_1}}O_E$. This allows us to reduce the result to the case of $\GL_n$ which is dealt with by \cite[Cor. 21.5.10]{Schber}.
\end{proof}
We view  Theorem \ref{SchLM} as evidence for the following conjecture.

\begin{conjecture}\label{conjsch}  For all local model triples $(G,\{\mu\}, K)$ with $G$ split over a tame extension, the local
model $\Mloc$ defined in the last subsection satisfies Scholze's Conjecture
\cite[Conj. 21.4.1]{Schber}.
\end{conjecture}

It has in any case the following concrete consequence.\footnote{We were recently informed that a similar result, which also covers cases of wildly ramified groups, was obtained   by J. Lourenco  (forthcoming Bonn thesis).} 

\begin{corollary}\label{Schmodclass}
Suppose that $(G,\{\mu\}, K)$ is an LM triple with $G$ adjoint and classical
such that $G$ splits over a tame extension of $F$. Assume that there exists a product decomposition over $\breve F$, 
$
G\otimes_F\breve F=\breve G_{1}\times\cdots \times \breve G_{m},
$
where each factor $\br G_i$ is \emph{absolutely simple}. If there is a factor for which $(\breve G_i,\{\mu_i\})\otimes_{\br F}\bar F$ is of type $(D_n, \omega^\vee_n)$ with $n\geq 4$ $\rm ($i.e., of type  $D^{\mathbb H}_n$ in Deligne's notation
\cite[Tables 1.3.9, 2.3.8]{D}$\rm )$, also assume that $p$ is odd. 
Then the local
model $\Mloc$ defined above satisfies Scholze's conjecture
\cite[Conj. 21.4.1]{Schber}.
 \end{corollary}

\begin{proof} We will show that such a LM triple $(G,\{\mu\}, K)$ is, after an unramified extension, of abelian type.
Using our assumption, we can easily reduce to the case that $G$ is absolutely simple,  quasi-split and residually split. 
The possible pairs $(G, \{\mu\})$ 
with such $G$  and $\{\mu\}$ minuscule,
are listed in the first two tables in \S 4.
A case--by--case check gives that, when $G$ is a classical group, we can find a realization $(G_1, \{\mu_1\})$ of $(G, \{\mu\})$ as a pair of abelian type such that $G_{1,\rm der}$ is  simply connected---except when the type of $G_{\bar F}$ is $D_{n}$. (See \cite[Rem. 3.10]{D}.) In the latter case we can find a realization with 
 $G_{1,\rm der}$ simply connected in the case $(D_n, \omega^\vee_1)$ (i.e., of type  $D^{\mathbb R}_n$ in Deligne's notation), and  a realization where $\pi_1(G_{1,\rm der})$ has order $2$ in the case
 $(D_n, \omega^\vee_n)$ (i.e., of type  $D^{\mathbb H}_n$ in Deligne's notation).
(For types $A_n$, $C_n$ and $D_{n}^{\mathbb H}$, the minuscule representation $\rho_1$ is given over $\bar F$ by a sum of corresponding standard representations, for types $B_n$ and   
$D_{n}^{\mathbb R}$, is given by a sum of spin representations.) In all cases, we can pick $\mu_1$ so that $E_1=E$.  The result follows from Theorem \ref{SchLM}.
\end{proof}
 
\section{Shimura varieties}\label{s:shimvar}

\subsection{Consequences for Shimura varieties}
 Let $({\bf G}, {\bf X})$ be a Shimura datum. We fix a prime $p>2$ such that $G:={\bf G}\otimes_\BQ \BQ_p$ splits over a tamely ramified extension of $\BQ_p$. We consider open compact subgroups ${\bf K}$ of ${\bf G}(\BA_f)$ of the form ${\bf K}={\bf K}^p\cdot K_p\subset {\bf G}(\BA_f^p)\times G(\BQ_p)$, where $K=K_p$ is a parahoric subgroup of $G(\BQ_p)$
 and ${\bf K}^p$ is sufficiently small. Let ${\bf E}$ be the reflex field of $({\bf G}, {\bf X})$. Fixing an embedding $\ov\BQ\to \ov\BQ_p$ determines a place ${\bf p}$ of ${\bf E}$ over $p$. Let $E={\bf E}_{\bf p}$. Then $E$ is the reflex field of $(G, \{\mu\})$, where $\{\mu\}$ is the conjugacy class of cocharacters over $\ov\BQ_p$ associated to ${\bf X}$. We denote by  the same symbol ${\rm Sh}_{\bf K}({\bf G}, {\bf X}) $ the canonical model of the Shimura variety over ${\bf E}$ and its base change over $E$. 
 \begin{theorem}\label{KPtheorem}
 
 a)  {\rm (\cite{KP})} Assume that $({\bf G}, {\bf X})$ is of abelian type. Then  there exists a scheme  $\CS_{ K}({\bf G}, {\bf X}) $ over $O_E$ with right ${\bf G}({\BA}^p_f)$-action such that: 
 \smallskip
 
 1) Any sufficiently small open compact   ${\bf K}^p\subset {\bf G}(\BA_f^p)$ acts freely on 
$\CS_{ K}({\bf G}, {\bf X}) $, and the quotient $\CS_{\bf K}({\bf G}, {\bf X}) :=\CS_{ K}({\bf G}, {\bf X})/{\bf K}^p$ is a scheme of finite type over $O_E$ which extends  ${\rm Sh}_{\bf K}({\bf G}, {\bf X}) $. Furthermore 
\[
\CS_{K}({\bf G}, {\bf X})=\varprojlim\nolimits_{{\bf K}^p} \CS_{{\bf K}^pK}({\bf G}, {\bf X}) ,
\]
where the limit is over all such ${\bf K}^p\subset {\bf G}(\BA_f^p)$.
 \smallskip
 
 2) For every closed point $x$ of $\CS_{K}({\bf G}, {\bf X})$, there is a  closed point $y$ of  $\Mloc$ such that the strict henselizations of $\CS_{K}({\bf G}, {\bf X})$ at $x$ and of $\Mloc$ at $y$ are isomorphic.  
  \smallskip
 
 3) The scheme $\CS_K({\bf G}, {\bf X})$  has the  extension property: 
 For every discrete valuation ring $R\supset O_E$ of
 characteristic $(0, p)$ the   map 
 \[
 \CS_K({\bf G}, {\bf X})(R)\to \CS_K({\bf G}, {\bf X})(R[1/p])
 \] is a bijection.
 \smallskip
 
b) {\rm (\cite{KP})} Assume that $({\bf G}, {\bf X})$ is of Hodge type, that $K$ is the stabilizer of a point in the Bruhat-Tits building of $G$, and that $p$ does not divide $|\pi_1(G_{\rm der})|$. Then the  model $\CS_{K}({\bf G}, {\bf X})$   of (a) above admits a ${\bf G}({\BA}^p_f)$-equivariant \emph{local model diagram} over $O_E$, 
 \begin{equation}\label{hecke unram semi-global G}
\begin{gathered}
   \xymatrix{
	     &\tilde{\CS}_{ K} ({\bf G}, {\bf X})  \ar[dl]_-{\text{$\pi$}} \ar[dr]^-{\text{$\tilde\varphi$}}\\
	   \CS_{ K} ({\bf G}, {\bf X}) & &  \Mloc,
	}
\end{gathered}
\end{equation}
 in which $\pi$ is a torsor under the group scheme $\CG_{O_E}$, and $\tilde\varphi$ is a $\CG_{O_E}$-equivariant and smooth morphism of relative dimension $\dim G$.  
 
 \smallskip

 c) {\rm (\cite[Thm. 8.2]{Zhou}, \cite[Thm. 4.1]{HeR})} Under the assumptions of (b) above, the morphism $\tilde \varphi$ in the local model diagram (\ref {hecke unram semi-global G}) is surjective. \qed
 
  \end{theorem}
  
\begin{remark}
 Part (a2) appears as \cite[Thm.~0.2]{KP}, but is stated there for the original   local models of \cite{PZ},
and under the assumption $p\nmid  |\pi_1(G_{ \rm der})|$. The statement above is for the modified local models of this paper and can be deduced by the results in \cite{KP}. Part (b) follows from \cite[Thm.~4.2.7]{KP}
and Remark \ref{compPZ} (2).
\end{remark}

\begin{definition}\label{def:goodss}
Let $O$ be a discrete valuation ring and suppose that $X$ is a locally noetherian scheme over $O$.  
\begin{altenumerate}
\item $X$ is said to have \emph{good reduction} over $O$  if $X$ is smooth over $O$.
\item $X$ is said to have \emph{semi-stable reduction} over $O$   if the special fiber is a normal crossings divisor in the sense of \cite[Def. 40.21.4]{Stacks}.
\end{altenumerate}

Both properties are local for the \'etale topology around each closed point of $X$ and imply that $X$ is a regular scheme with reduced special fiber. 
\end{definition}

\begin{corollary}\label{CorAb} Assume that $({\bf G}, {\bf X})$ is a Shimura datum of abelian type.
If the local model $\Mloc$ has good, resp. semi-stable, reduction over $O_E$, then so does $\CS_{\bf K}({\bf G}, {\bf X})$. If $({\bf G}, {\bf X})$ is  of Hodge type and satisfies the assumptions of Theorem \ref{KPtheorem} (b), then the converse also holds. 
\end{corollary}
\begin{proof}
  The first assertion follows from  Theorem \ref{KPtheorem} (a). The second assertion follows from (b) and (c). 
  \end{proof}

\subsection{Canonical nature of integral models}  By the main result of \cite{PCan},  the integral models $\CS_{K}({\bf G}, {\bf X})$ constructed in \cite{KP} are, under the assumptions of Theorem \ref{KPtheorem} (b), independent of the choices in their construction.  In fact, they are
``canonical'' in the sense that they satisfy the characterization given in \cite{PCan}. In this paper, we are dealing with models that have smooth or semi-stable reduction. Then, and under some additional assumptions, we can give a simpler characterization of the integral models using an idea of Milne \cite{Milne} and results of Vasiu and Zink (\cite{VasiuZink}). More precisely, we have:

\begin{corollary}\label{canCor}
 Assume that $({\bf G}, {\bf X})$ is a Shimura datum of abelian type. Suppose that $\Mloc$ has good or semi-stable reduction over $O_E$, that $E/\BQ_p$ is unramified, and that the geometric special fiber 
$\Mloc\otimes_{O_E}k$ has no more than $2p-3$ irreducible components. Then $\CS_{ K}({\bf G}, {\bf X})$ is, up to isomorphism, the \emph{unique} $O_E$-faithfully flat ${\bf G}({\BA}^p_f)$-equivariant integral model of ${\rm Sh}_{K}({\bf G}, {\bf X})$ that satisfies (1), (2) and the following stronger version of (3): The bijection 
\[
\CS_K({\bf G}, {\bf X})(R)\xrightarrow{\sim} \CS_K({\bf G}, {\bf X})(R[1/p])
\]
holds for $R$ any $O_E$-faithfully flat algebra which is either a dvr, or a regular ring which is \emph{healthy} in the sense of \cite{VasiuZink}.

\end{corollary}
\begin{proof}
  Note that under our assumption, 
by  \cite[Thm.~3, Cor. 5]{VasiuZink} (see also loc.~cit., p. 594), the scheme
$\Mloc$  is \emph{regular healthy},  when the 
  maximum number of transversely intersecting smooth components of its special fiber is $\leq 2p-3$. Then, by Theorem \ref{KPtheorem} (a), the same is true for $\CS_{ K}({\bf G}, {\bf X})$. By the construction of $\CS_{ K}({\bf G}, {\bf X})$ in \cite{KP} and \cite{VasiuZink}, it then follows that  the limit $\CS_{K}({\bf G}, {\bf X})$ also satisfies the extension property not just for dvr's but for all regular healthy schemes. The uniqueness part of the statement then also follows (see also \cite{Milne}, \cite{KP}).
\end{proof}
  
Consider the cases of smooth or semi-stable reduction covered by the results in this paper, see Theorems \ref{main1Intro} and \ref{main2Intro}, for $F=\BQ_p$: it turns out that 
the  number $r$ of geometric irreducible components of the special fiber of $\Mloc$ is $\leq 2$ in all cases, except in the first case of Theorem \ref{mainssred} (the Drinfeld case). In the latter case,  this number $r$ is equal to the number of lattices in the primitive part of the periodic  lattice chain.  Since we assume that $p$ is odd to begin with, we obtain:

\begin{theorem}\label{3.6}
Assume that $({\bf G}, {\bf X})$ is a Shimura datum of abelian type such that the corresponding LM triple $(G,\{\mu\}, K)$ satisfies the hypothesis of either Theorem \ref{main1Intro} or Theorem \ref{main2Intro}, with $F=\BQ_p$. Then, unless $(G,\{\mu\}, K)$ corresponds to the ``Drinfeld case'' of Theorem \ref{mainssred}, the model $\CS_{ K}({\bf G}, {\bf X})$ is  canonical, i.e., it satisfies the conclusion of Corollary \ref{canCor}.
If $(G, \{\mu\}, K)$ corresponds to the Drinfeld case of Theorem \ref{mainssred}, then $\CS_{ K}({\bf G}, {\bf X})$ is  canonical, provided that $K$ is the connected stabilizer of a facet  in the building of $  {\rm PGL}_n $ that is of dimension   $\leq 2p-4$. \hfill $\square$
\end{theorem}

\begin{example}\label{Faltings}
Consider the group  ${\bf G}={\mathrm  {GSpin} }({\bf V})$, where ${\bf V}$ is a (non-degenerate) orthogonal space
of dimension $2n\geq 8$ over $\BQ$ of signature $(2n-2, 2)$ over  ${\BR}$. Take  
\[
{\bf X}=\{v\in {\bf V}\otimes_{\BQ}{\BC}\ |\ \langle v, v\rangle=0, \langle v, \bar v\rangle<0\}/{\BC}^*.
\]
(Here  $ \langle\ ,\ \rangle$ is the corresponding symmetric bilinear form.)
The group ${\bf G}({\BR})$ acts on ${\bf X}$ via ${\bf G}\to {\rm SO}({\bf V})$ and $({\bf G}, {\bf X})$
is a Shimura datum of Hodge type.

Suppose that there exists a pair $(\Lambda_0, \Lambda_n)$ of $\BZ_p$-lattices in $ {\bf V}\otimes_{\BQ}{\BQ_p}$, with $\Lambda^\vee_0=\Lambda_0$, $\Lambda^\vee_n=p\Lambda_n$, and $p\Lambda_n\subset \Lambda_0\subset \Lambda_n$. 
Let $K_p\subset {\bf G}(\BQ_p)$ be the parahoric subgroup which corresponds to 
the connected stabilizer of this lattice chain. By combining Theorem \ref{mainssred} and the above, we obtain that, for small enough ${\bf K}^p$, the Shimura variety ${\rm Sh}_{\bf K}({\bf G}, {\bf X}) $ has a canonical $\BZ_p$-integral model with semi-stable reduction. In fact, we can see, using the calculations in Subsection \ref{ss: spliteven}, that the integral model is locally smoothly equivalent to $\BZ_p[x,y]/(xy-p)$. This integral model was  found by Faltings \cite{F} as an application of his   theory of ${\mathcal {MF}}$-objects over semi-stable bases.

\end{example}

 \section{Rapoport-Zink spaces}\label{s:RZspaces}
We consider RZ-spaces of EL-type or PEL-type, cf. \cite{R-Z}. We place ourselves in the situation described in \cite[\S 4]{RV}. 

\subsection{The formal schemes}
In the EL-case, we start with \emph{rational RZ data} of EL-type  
\[
\mathcal D=(F, B, V, G, \{\mu\}, [b]).
\]
 Here $F$ is a finite extension of $\BQ_p$, $B$ is a central division algebra over $F$, $V$ is a finite-dimensional $B$-module, $G=\GL_B(V)$ as algebraic group over $\BQ_p$, $\{\mu\}$ is a conjugacy class of minuscule cocharacters of $G$, and $[b]\in A(G, \{\mu\})$ is an \emph{acceptable} $\sigma$-conjugacy class in $G(\breve \BQ_p)$. Let $E=E_{\{\mu\}}$ be the corresponding reflex field inside $\ov\BQ_p$. In addition, we fix \emph{integral RZ data} ${\mathcal D}_{\BZ_p}$, i.e., a periodic lattice chain of $O_B$-modules $\Lambda_\bullet$ in $V$. This lattice chain defines a parahoric group scheme $\CG$ over $\BZ_p$ with generic fiber $G$. 

In the PEL-case, we start with \emph{rational RZ data} of PEL-type 
\[
{\mathcal D}=(F, B, V, (\, , \, ), *, G, \{\mu\}, [b]).
\] Here $F$,  $B$ and $V$ are as in the EL-case, $(\, , \,)$ is a non-degenerate alternating $\BQ_p$-bilinear form on $V$, $*$ is an involution on $B$,  $G={\rm GSp}_B(V)$ as algebraic group over $\BQ_p$, and $\{\mu\}$ and $[b]$ are as before. We refer to \cite{RV} for the precise conditions these data have to satisfy. In addition, we fix \emph{integral RZ data} ${\mathcal D}_{\BZ_p}$, i.e., a periodic self-dual lattice chain of $O_B$-modules $\Lambda_\bullet$ in $V$. In the PEL case we make the following assumptions.
\begin{altitemize}
 \item $p\neq 2$.
 \item $G$ is connected. 
 \item The stabilizer group scheme $\CG$ is a parahoric group scheme over $\BZ_p$.
\end{altitemize}  
Then in all cases $(G, \{\mu\}, \CG)$ is a LM triple over $\BQ_p$. As in Section \ref{s:LM}, we sometimes write the LM triple as $(G, \{\mu\}, K)$ with $K=\CG(\BZ_p)$. 

Let $O_{\breve E}$ be the ring of integers in $\breve E$ (the completion of the maximal unramified extension of $E$). In either EL or PEL case, after fixing a \emph{framing object} $\BX$ over $k$ (the residue field of $O_{\breve E}$), we obtain a formal scheme locally formally of finite type over $\Spf O_{\breve E}$ which represents a certain moduli problem of $p$-divisible groups on the category ${\rm Nilp}_{O_{\breve E}}$. We denote this formal scheme by $\CM^\naive_{\mathcal D_{\BZ_p}}$. The reason for the upper index is that we impose only the \emph{Kottwitz condition} on the $p$-divisible groups appearing in the formulation of the moduli problem. In particular, $\CM^\naive_{\mathcal D_{\BZ_p}}$ need not be flat over $\Spf O_{\breve E}$. 

Analogously, associated to $\mathcal D_{\BZ_p}$, there is the local model $\BM^\naive_{\mathcal D_{\BZ_p}}$, a projective scheme over $O_E$ equipped with an action of $\CG_{O_E}=\CG\otimes_{\BZ_p} O_E$. Furthermore, there is a local model diagram of morphisms of formal schemes over $\Spf O_{\breve E}$, 

 \begin{equation}\label{LMDforRZ}
\begin{gathered}
   \xymatrix{
	     &\wt{\CM}^\naive_{\mathcal D_{\BZ_p}}  \ar[dl]_-{\text{$\pi$}} \ar[dr]^-{\text{$\tilde\varphi$}}\\
	   \CM^\naive_{\mathcal D_{\BZ_p}} & &  (\BM^\naive_{\mathcal D_{\BZ_p}})^\wedge,
	}
\end{gathered}
\end{equation}
 in which $\pi$ is a torsor under the group scheme $\CG_{O_E}$, and $\tilde\varphi$ is a $\CG_{O_E}$-equivariant and formally smooth morphism of relative dimension $\dim G$.  Here $(\BM^\naive_{\mathcal D_{\BZ_p}})^\wedge$ denotes the completion of $\BM^\naive_{\mathcal D_{\BZ_p}}\otimes_{O_E} O_{\breve E}$ along its special fiber. 
 
\begin{lemma}
Assume that the group $G$ attached to the rational RZ-data $\mathcal D$ splits over a tame extension of $\BQ_p$.  Then the modified PZ-local model $\Mloc$ of Subsection \ref{ss:PZmodif} attached to the LM triple $(G, \{\mu\}, \CG)$ is a closed subscheme of $\BM^\naive_{\mathcal D_{\BZ_p}}$, with identical generic fiber. 
\end{lemma}

\begin{proof} Notice that under our assumptions, since this is always true in the EL case, $G$ is connected.
We can see that, under our assumptions, $p$ does not divide $|\pi_1(G_{\rm der})|$. Indeed, this is clear in the EL case since then $(G_{\rm der})_{\bar\BQ_p}$ is a product of special linear groups ${\mathrm {SL}}$. In the PEL case, $(G_{\rm der})_{\bar\BQ_p}$ is the product of groups of types ${\mathrm {SL}}$, ${\mathrm {Sp}}$, ${\mathrm {SO}}$, and our assumptions include that $p$ is odd. It follows from \ref{compPZ} (2) that $\BM^{\rm loc}_K(\CG, \{\mu\})\simeq M_{\CG,\mu}$. By \cite[(8.3)]{PZ},
under the above assumptions again (in particular, the fact that  $G$ is connected is used), the local model $M_{\CG,\mu}$ agrees with the flat closure of the generic fiber of the naive local model $\BM^\naive_{\mathcal D_{\BZ_p}}$. The result follows.
\end{proof} 

We now use the local model diagram \eqref{LMDforRZ} to define a closed formal subscheme $\CM_{\mathcal D_{\BZ_p}}$ of $\CM^\naive_{\mathcal D_{\BZ_p}}$, defined by an ideal sheaf killed by a power of the uniformizer of $O_E$. Indeed, consider the ideal sheaf on $\BM^\naive_{\mathcal D_{\BZ_p}}$ defining $\Mloc$. It defines, after completion and pullback under $\tilde\varphi$ an ideal sheaf on  $\wt{\CM}^\naive_{\mathcal D_{\BZ_p}}$ which descends along $\pi$ to ${\CM}^\naive_{\mathcal D_{\BZ_p}}$. We therefore obtain a local model diagram
\begin{equation}\label{LMDforRZ2}
\begin{gathered}
   \xymatrix{
	     &\wt{\CM}_{\mathcal D_{\BZ_p}}  \ar[dl]_-{\text{$\pi$}} \ar[dr]^-{\text{$\tilde\varphi$}}\\
	   \CM_{\mathcal D_{\BZ_p}} & &  (\Mloc)^\wedge,
	}
\end{gathered}
\end{equation}
 where we have recycled the notation from \eqref{LMDforRZ}. Again, the left oblique arrow is a torsor under $\CG_{O_E}$, and the right oblique arrow  is  $\CG_{O_E}$-equivariant and formally smooth of relative dimension $\dim G$. 
\begin{corollary}\label{CorRZ} 
Assume that the group $G$ attached to the rational RZ-data $\mathcal D$ splits over a tame extension of $\BQ_p$.
If the local model $\Mloc$ has good, resp. semi-stable, reduction over $O_E$, then so does $\CM_{\mathcal D_{\BZ_p}}$.
\begin{proof}
This follows by descent from the local model diagram. 
\end{proof}
\end{corollary}
\begin{remark}
In contrast to Corollary \ref{CorAb}, the converse does not hold in general because the morphism $\tilde\varphi$ is not always surjective. However, the converse holds if the RZ data $\mathcal D$ are \emph{basic}, i.e.,  $[b]$ is basic. 
\end{remark}
 \begin{proposition}
Assume that $\mathcal D$ is basic and that the group $G$ attached to  $\mathcal D$ splits over a tame extension of $\BQ_p$.
If the RZ space $\CM_{\mathcal D_{\BZ_p}}$ has good, resp. semi-stable, reduction over $O_E$, then so does the local model $\Mloc$.
\end{proposition}
\begin{proof}
Indeed, $\CM_{\mathcal D_{\BZ_p}}$ can be identified with the formal completion of an open and closed subset of a Shimura variety of Hodge type along its basic stratum. But this closed stratum is contained in the closed subset of non-smooth, resp. non-semi-stable points (if these are non-empty). Therefore the assertion follows from Corollary \ref{CorAb}. 
\end{proof}
\begin{proposition}
Assume that the group $G$ attached to the rational RZ-data $\mathcal D$ splits over a tame extension of $\BQ_p$. Then the formal scheme $\CM_{\mathcal D_{\BZ_p}}$ is flat over $\Spf O_{\breve E}$ and normal. Furthermore, it only depends on ${\mathcal D_{\BZ_p}}$ through the quadruple $(G, \{\mu\}, \CG, [b])$. Finally,
\begin{equation}\label{ADLV}
\CM_{\mathcal D_{\BZ_p}}(k)=\bigcup _{w\in\Adm_{\breve K}(\{\mu\})}  X_w(b) .
\end{equation}
\end{proposition}
\noindent Here $\Adm_{\breve K}(\{\mu\})\subset W_{\breve K}\backslash \wt W/W_{\breve K}$ denotes the admissible set. Also  $X_w(b)$ denotes  for $w\in W_{\breve K}\backslash \wt W/W_{\breve K}$ the affine Deligne-Lusztig set 
$$
X_w(b)=\{g\in G(\breve\BQ_p)/\breve K\mid g^{-1}b\sigma(g)\in \breve K w \breve K \} 
$$
where $b$ is a fixed representative of $[b]$.

\begin{proof}
Flatness and normality follows via the local model diagram  from the corresponding properties of $\Mloc$, cf. Remark \ref{compPZ} (1). The uniqueness statement follows from \cite[Cor. 25.1.3]{Schber}. The final statement follows from
\cite[Cor. 25.1.3]{Schber} and Theorem \ref{2.9} together with (\ref{indentadm}) and the definition 
\cite[Def. 25.1.1]{Schber} of the $v$-sheaf ${\mathcal M}^{\rm int}_{(\CG, \mu, b)}$ by observing the following: In the
 definition   of  ${\mathcal M}^{\rm int}_{(\CG, \mu, b)}$ we can take, by Corollary \ref{Schmodclass}, the local model $\Mloc$ to give the ``diamond" local model $v$-sheaf ${\mathcal M}^{\rm loc}_{\CG,\mu}$ used there.
 \end{proof}

\subsection{The RZ tower}
We now pass to the RZ-tower of rigid-analytic spaces $\big(M_K, K\subset G(\BQ_p)\big)$, cf. \cite[\S 4.15]{RV}. For its formation, we can start with $\BM^\naive_{\mathcal D_{\BZ_p}}$ for an arbitrary integral RZ datum $\mathcal D_{\BZ_p}$ for $\mathcal D$; in particular, we need not assume that $G$ is tamely ramified. 
\begin{proposition}
The RZ-tower $(M_K)$ depends only on the rational RZ datum $\mathcal D$ through the triple $(G, \{\mu\}, [b])$. Furthermore, if it is non-empty, then $[b]\in B(G, \{\mu\})$. The converse holds if $G$ splits over a tamely ramified extension of $\BQ_p$.   
\end{proposition}
\begin{proof}
The first assertion follows from \cite[Cor. 24.3.5]{Schber}. The second assertion is \cite[Prop. 4.19]{RV}. To prove the converse, using flatness of $\CM_{\mathcal D_{\BZ_p}}$, it suffices to prove $\CM_{\mathcal D_{\BZ_p}}(k)\neq\emptyset$. Via the identification \eqref{ADLV}, this follows from \cite{He}.
\end{proof}
\begin{remark}
The uniqueness statement  is conjectured in \cite[Conj. 4.16]{RV} without the tameness assumption. The converse statement is conjectured in \cite[Conj. 4.21]{RV}, again without the tameness assumption.
\end{remark}

\section{Statement of the main results}\label{s:mainresults}

\subsection{Good reduction}\label{ss:goodred}
In the following, we call the LM triple $(G, \{\mu\}, K)$ of \emph{exotic good reduction type} if $p\neq 2$ and if the corresponding adjoint LM triple $(G_\ad, \{\mu_\ad\}, K_\ad)$ is isomorphic to the adjoint LM triple associated to one of the following two LM triples.
 
\medskip

1) ({\sl Unitary exotic reduction})
\medskip
 
 \begin{altitemize}\label{exotgood}
\item $G=\Res_{F'/F} G'.$ Here $F'/F$  is an unramified extension, and  $G'=U(V)$, with $V$  a $\tilde F'/F'$-hermitian vector space of dimension $\geq 3$, where $\tilde F'/F'$ is a \emph{ramified} quadratic extension. 

 \item $\{\mu\}=\{\mu_\varphi \}_{\varphi\colon F'\to \ov F}$,  with $\{\mu_\varphi\}=(1,0,\ldots,0)$ or $\{\mu_\varphi\}=(0,0,\ldots,0),$ for any $\varphi$. 

\item $K= \Res_{O_{F'}/O_F}(K')$ ,  with $K'={\rm Stab} (\Lambda)$, where $\Lambda$ is a $\pi$-modular or almost $\pi$-modular vertex lattice in $V$, i.e., $\Lambda^\vee=\pi_{\tilde F'}^{-1} \Lambda$ if $\dim V$ is even, resp.  $\Lambda\subset \Lambda^\vee\subset^1\pi_{\tilde F'}^{-1} \Lambda$ if $\dim V$ is odd. 
\end{altitemize}
\medskip

2) ({\sl Orthogonal exotic reduction})
\medskip

 \begin{altitemize}\label{exotgood2}
 \item $G=\Res_{F'/F} G'.$ Here $F'/F$  is an unramified extension, and  $G'= {\rm GO}(V)$, with $V$  an orthogonal $F'$-vector space of even dimension $2n\geq 6$.
 
  \item  $\{\mu\}=\{\mu_\varphi \}_{\varphi\colon F'\to \ov F}$,  with $\{\mu_{\rm ad, \varphi}\}=(1^{(n)}, 0^{(n)})_{\rm ad}$ or $\{\mu_{\rm ad, \varphi}\}=(0,0,\ldots,0),$ for any $\varphi$.

 \item $K=\Res_{O_{F'}/O_F}(K')$ ,  with $K'={\rm Stab} (\Lambda)$, where $\Lambda$ is an almost selfdual vertex lattice in $V$, i.e., $\Lambda\subset^1 \Lambda^\vee\subset\pi_{F'}^{-1} \Lambda$.

 \end{altitemize}

\begin{theorem}\label{maingoodred} 
Let $(G, \{\mu\}, K)$ be a triple over $F$ such that $G$ splits over a tame extension of $F$.  Assume $p\neq 2$. 
Assume that $G_\ad$ is $F$-simple, that in the product decomposition over $\breve F$, 
$$
G_\ad\otimes_F\breve F=\prod\nolimits_i \breve G_{\ad, i}
$$
each factor is  \emph{absolutely simple}, and that $\mu_\ad$ is not trivial.
Then the  local model $\Mloc$ is smooth over $\Spec O_E$ if and only if $K$ is hyperspecial or $(G, \mu, K)$ is a triple of exotic good reduction type\footnote{Haines-Richarz \cite{HR} gives an alternative explanation for the smoothness of $\BM^\loc_K(G, \{\mu\})$ in the case of exotic good reduction type for the \emph{even unitary} case and the \emph{orthogonal} case: in these cases, the special fiber of $\BM^\loc_K(G, \{\mu\})$ can be identified with a Schubert variety attached to a  \emph{minuscule} cocharacter in the twisted affine Grassmannian corresponding to  the special maximal parahoric $K$.}.
\end{theorem}
We are going to use the following d\'evissage lemma.
\begin{lemma}\label{devislem}
a) Let $F'/F$ be a finite unramified extension contained in $\breve F$. Let
$$
(G, \{\mu\}, K)\otimes_F F'=\prod\nolimits_i (G_i, \{\mu_i\}, K_i) ,
$$
where $(G_i, \{\mu_i\}, K_i)$ are LM triples over $F'$. 
Then  $\Mloc$ is smooth over $\Spec O_E$ if and only if $\BM^\loc_{K_i}(G_i, \{\mu_i\} )$ is smooth over $\Spec O_{E_i}$ for all $i$. 

\smallskip

\noindent b) Let $(G', \{\mu'\}, K')\to (G, \{\mu\}, K)$ be a morphism of triples such that $G'\to G$ gives a central extension. Then 
 $\Mloc$ is smooth over $\Spec O_E$ if and only $\BM^\loc_{K'}(G', \{\mu'\})$ 
is smooth over $\Spec O_{E'}$.
\end{lemma}
\begin{proof}
This follows from properties (ii)--(iv) of Proposition \ref{proofLM}. 
\end{proof}
The lemma implies that, in order to prove Theorem \ref{maingoodred}, we may assume that $G_\ad$ is absolutely simple
and that $\mu_\ad$ is not trivial. That $\Mloc$ is smooth over $\Spec O_E$ when $K$ is hyperspecial is property (i) of Proposition \ref{proofLM}. The case of unitary exotic good reduction is treated in \cite[Prop. 4.16]{Arz}, comp. \cite[Thm. 2.27, (iii)]{PRS}.  The case of orthogonal exotic good reduction is   discussed in Subsection  \ref{exoticOrtho}.

The proof of the converse proceeds in three steps. In a first step,
we establish a list of all cases in which the special fiber of $\Mloc$ is irreducible, i.e., $\frak A_K(G, \{\mu\})$ is a single Schubert variety in the corresponding affine partial flag variety. This is done in Section \ref{s:CCP}. In a second step, we go through this list and eliminate the cases when $K$ is not a special maximal parahoric by showing that   in those cases the special fiber is not  smooth (in fact, not even \emph{rationally smooth}, in the sense explained in Section \ref{s:pssredCCP}). This is done in Section \ref{s:ratpsss}.  Finally, we deal with the cases when $K$ is a special maximal parahoric; most of these   can be also dealt with by the same methods. In a few cases, we need to refer to certain explicit calculations of the special fibers 
given in \cite{PR}, \cite{Arz}, and, in one exceptional type, appeal to the result of Haines-Richarz \cite{HR}.

\subsection{Weyl group notation}\label{Titsdata} Recall that simple adjoint groups $\breve G$ over $\breve F$ are classified up to isomorphism by their associated {\it local Dynkin diagram}\footnote{Note that only the first batch of cases on Tits' list is relevant since $\breve G$ is automatically residually split.}, cf. \cite[\S 4]{T}. Recall that to a local Dynkin diagram $\tilde \Delta$ there is associated its Coxeter system, cf. \cite{Bou}, which is of affine type. The associated Coxeter group is the affine Weyl group $W_a$. We denote by   $\tilde W$ its extended affine Weyl group. Both $W_a$ and $\tilde W$ are extensions of the finite Weyl group $W_0$ by  translation subgroups,  i.e.,   finitely generated free $\BZ$-modules. We denote by $X_*$ the translation subgroup of $\tilde W$.
\begin{definition}\label{def:enhanced}
\begin{altenumerate}
	\item An \emph{enhanced Tits datum} is a triple $(\tilde\Delta, \{\lambda\}, \tilde K)$ consisting of a local Dynkin diagram $\tilde \Delta$, a $W_0$-conjugacy class $\{\lambda\}$ of elements in $X_*$, and a non-empty subset $\tilde K$ of the set $\tilde S$ of vertices of $\tilde \Delta$. 
	\item An \emph{enhanced Coxeter datum} is a triple $\big((W_a, \tilde S), \{\lambda\}, \tilde K\big)$ consisting of a Coxeter system $(W_a, \tilde S)$ of affine type, a $W_0$-conjugacy class  $\{\lambda\}$ of  elements in $X_*$, and a non-empty subset $\tilde K$ of  $\tilde S$. 
\end{altenumerate}
\end{definition}

   Note that the Coxeter system $(W_a, \tilde S)$ is given by its associated Coxeter diagram, cf. \cite[Chapter VI, \S 4, Thm.~4]{Bou}. The Coxeter diagram associated to a local Dynkin diagram is obtained by disregarding the arrows in the local Dynkin diagram. An enhanced Tits datum determines an enhanced Coxeter datum. The natural map from the set of enhanced Tits data to the set of enhanced Coxeter data is not injective. 

Let $(G, \{\mu\}, K)$ be a LM triple over $F$ such that $G$ is adjoint and absolutely simple. We associate as follows an enhanced Tits datum to $(G, \{\mu\}, K)$. The local Dynkin diagram $\tilde \Delta$  is that associated to $\breve G=G\otimes_F\breve F$. Let $\breve T$ be a maximal torus of $\breve G$ contained in a Borel subgroup $\breve B$ containing $\breve T$. We may choose a representative $\mu$ of $\{\mu\}$ in $X_*(\breve T)$ which is dominant for $\breve B$. There is a canonical identification of $X_*$ with $X_*(\breve T)_{\Gamma_0}$ (co-invariants under the inertia group). The second component of the enhanced Tits datum is the image $\lambda$ of $\mu$ in $X_*$. It is well-defined up to the action of $W_0$ (this follows, since $W_0$ is identified with the relative Weyl group of $\breve G$ and any two choices of $\breve B$ are conjugate under the relative Weyl group). The third component of the enhanced Tits datum is the subset $\tilde K$ of vertices of $\tilde \Delta$ which describes  the conjugacy class  under $\breve G(\breve F)$ of the parahoric subgroup $\breve K$ of $\breve G(\breve F)$ determined by $K$.  

Given a LM triple, one may compute its associated enhanced Tits datum as follows. First, if $G$ is a split group, with associated Dynkin diagram $\Delta$, then the local Dynkin diagram $\tilde\Delta$ is simply the associated affine Dynkin diagram, cf. \cite[VI,\S2]{B}. 
\medskip

\begin{footnotesize}

\begin{center}
	\begin{tabular}{|c|c|c|}
		\hline
		\text{Name (Index)} & \text{Local Dynkin diagram} &{\text{Minuscule coweights}} \\
		\hline
		$A_n$ (${}^1 A^{(1)}_{n, n}$) for $n \ge 2$ & 	\begin{tikzpicture}[baseline=0]
		\node[]  at (0,0) {$\circ$};
		\node[below] at (0, 0) {$1$};
		\node[] at (4,0) {$\circ$};
		\node[below] at (4.15, 0) {$n$};
		\node at (1, 0) {$\circ$};
		\node[below] at (1, 0) {$2$};
		\node at (3, 0) {$\circ$};
		\node[below] at (2.85, 0) {$n-1$};
		\node at (2, 0.5) {$\circ$};
		\node[above] at (2.2, 0.5) {$0$};
		\draw[-] (0.1, 0) to (0.9,0);
		\draw[-] (0.1, 0.1) to (1.9, 0.4);
		\draw[-] (2.1, 0.4) to (3.9, 0.1); 
		\draw[dashed] (1.1,0) to (2.9,0);
		\draw[-] (3.1,0)  to (3.9,0);
		\end{tikzpicture} & $\{\omega^\vee_i\}$, $1 \le i \le n$ \\
		\hline 
		$A_1$ (${}^1 A^{(1)}_{1, 1}$) &
		\begin{tikzpicture}[baseline=0]
		\node[]  at (0,0) {$\circ$};
		\node[below] at (0, 0) {$1$};
		\node[above] at (0,0) {\phantom{A}};
		\node at (1, 0) {$\circ$};
		\node[below] at (1, 0) {$0$};
		\draw[-, line width=2] (0.1, 0) to (0.9, 0);
		\end{tikzpicture} & $\{\omega^\vee_1\}$ \\
		\hline 
		$B_n$ ($B_{n, n}$) for $n \ge 3$ &
		\begin{tikzpicture}[baseline=0]
		\node at (0,0) {$\circ$};
		\node[below] at (0, 0) {$n$};
		\draw[implies-, double equal sign distance] (0.1, 0) to (0.9,0);
		\node  at (1,0) {$\circ$};
		\node[below] at (1, 0) {$n-1$};
		\draw (1.1, 0) to (1.9,0);
		\node at (2,0) {$\circ$};
		\node[below] at (2, 0) {$n-2$};	
		\draw[dashed] (2.1, 0) to (2.9, 0);
		\node at (3.0, 0) {$\circ$};
		\node[below] at (3, 0) {$2$};
		\draw[-] (3.1,0.1) to (3.9, 0.35);
		\node at (4, 0.4) {$\circ$};
		\node [right] at (4, 0.4) {$1$};
		\draw[-] (3.1, - 0.1) to (3.9, -0.35);
		\node at (4,-0.4) {$\circ$};
		\node [right] at (4, -0.4) {$0$};
		\end{tikzpicture} & $\{\omega^\vee_1\}$ \\
		\hline
		$C_n$ ($C^{(1)}_{n, n}$) for $n \ge 2$ &
		\begin{tikzpicture}[baseline=0]
		\node at (0,0) {$\circ$};
		\node[below] at (0, 0) {$0$};
		\node[above] at (0,0) {\phantom{A}};
		\draw[-implies, double equal sign distance] (0.1, 0) to (0.9,0);
		\node  at (1,0) {$\circ$};
		\node[below] at (1, 0) {$1$};
		\draw (1.1, 0) to (1.9,0);
		\node at (2,0) {$\circ$};
		\node[below] at (2, 0) {$2$};
		\draw[dashed] (2.1, 0) to (2.9, 0);
		\node at (3, 0) {$\circ$};
		\node[below] at (3, 0) {$n-1$};
		\node at (4, 0) {$\circ$};
		\node[below] at (4, 0) {$n$};
		\draw[implies-, double equal sign distance] (3.1, 0) to (3.9,0);
		\end{tikzpicture} & $\{\omega^\vee_n\}$ \\
		\hline
		$D_n$ (${}^1 D^{(1)}_{n, n}$) for $n \ge 4$ &
		\begin{tikzpicture}[baseline=0]
		\node at (0,0.4) {$\circ$};
		\node [left] at (0, 0.4) {1};
		\node [left] at (0, -0.4) {0};
		\node at (0, -0.4) {$\circ$};
		\node at (1,0) {$\circ$};
		\node[below] at (1, 0) {$2$};
		\draw[-] (0.1, 0.35) to (0.9, 0.1);
		\draw[-] (0.1, -0.35) to (0.9, -0.1);
		\draw (1.1, 0) to (1.9,0);
		\node at (2,0) {$\circ$};
		\node[below] at (2, 0) {$3$};	
		\draw[dashed] (2.1, 0) to (2.9, 0);
		\node at (3.0, 0) {$\circ$};        \node[below] at (2.8, 0) {$n-2$};
		\draw[-] (3.1,0.1) to (3.9, 0.35);
		\node at (4, 0.4) {$\circ$};
		\node [right] at (4, 0.4) {$n-1$};
		\draw[-] (3.1, - 0.1) to (3.9, -0.35);
		\node at (4,-0.4) {$\circ$};
		\node [right] at (4, -0.4) {$n$};
		\end{tikzpicture} & 
		$\{\omega^\vee_1, \omega^\vee_{n-1}, \omega^\vee_{n}\}$ \\
		\hline 
		$E_6$ (${}^1 E^0_{6, 6}$) &
		\begin{tikzpicture}[baseline=0]
		\node at (0,0) {$\circ$};
		\node[below] at (0, 0) {$0$};
		\node at (1,0) {$\circ$};
		\node[below] at (1, 0) {$2$};	
		\draw (0.1, 0) to (0.9,0);		
		\draw (1.1, 0) to (1.9,0);		
		\node at (2.0, 0) {$\circ$};        \node[below] at (2, 0) {$4$};
		\draw[-] (2.1,0.1) to (2.9, 0.35);
		\node at (3, 0.4) {$\circ$};
		\node [below] at (3, 0.4) {$3$};
		\draw[-] (2.1, - 0.1) to (2.9, -0.35);
		\node at (3,-0.4) {$\circ$};
		\node [below] at (3, -0.4) {$5$};		
		\node at (4, 0.8) {$\circ$};
		\node [right] at (4, 0.8) {$1$};
		\draw[-] (3.1, -0.45) to (3.9, -0.75);		\draw[-] (3.1, 0.45) to (3.9, 0.75);
		\node at (4,-0.8) {$\circ$};
		\node [right] at (4, -0.8) {$6$};
		\end{tikzpicture} & 
		$\{\omega^\vee_1, \omega^\vee_6\}$ \\
		\hline 
		$E_7$ ($E^0_{7, 7}$) & 
		\begin{tikzpicture}[baseline=0]
		\node at (0,0) {$\circ$};
		\node[below] at (0, 0) {$0$};
		\draw[] (0.1, 0) to (0.9,0);
		\node  at (1,0) {$\circ$};
		\node[below] at (1, 0) {$1$};
		\draw (1.1, 0) to (1.9,0);
		\node at (2,0) {$\circ$};
		\draw[] (2.1, 0) to (2.9, 0);
		\node[below] at (2, 0) {$3$};
		\draw[] (3.1, 0) to (3.9, 0);
		\node at (3, 0) {$\circ$};
		\node[below] at (3, 0) {$4$};
		\node at (4, 0) {$\circ$};
		\node[below] at (4, 0) {$5$};
		\draw[] (4.1, 0) to (4.9,0);
		\node[below] at (5, 0) {$6$};
		\node at (5, 0) {$\circ$};
		\draw[] (5.1, 0) to (5.9, 0);
		\node at (6, 0) {$\circ$};
		\node[below] at (6, 0) {$7$};
		\node at (3, 0.5) {$\circ$};
		\node[right] at (3, 0.5) {$2$};
		\draw[] (3, 0.1) to (3, 0.4);
		\end{tikzpicture} & $\{\omega^\vee_7\}$
		\\
		
		\hline
	\end{tabular} 
\end{center}

\end{footnotesize}
\medskip
\smallskip

 Now let  $G$ be quasi-split and residually split. Then the affine root system is calculated 
 following the recipe in \cite[\S 2.3]{PR}. This gives the list below.  In the column ``Local Dynkin diagram'', there are two rows associated to each group: the first row gives the local Dynkin diagram of the group $G$ over a (ramified) field extension $\breve F'$ of $\breve F$ such that $G$ splits over $\breve F'$; the second row gives the local Dynkin diagram of the group $G$ over $\breve F$. In the column ``Coweights'', there are two rows: the first row for the minuscule coweight $\mu$; the second row for the corresponding $\l$ realized as a translation element of the associated extended affine Weyl group. Here we put minuscule coweights between braces if they determine the same $\lambda$ which appears directly below. We follow the notation in \cite{T}. 
 \smallskip

\begin{footnotesize}
\begin{center}
	\begin{tabular}{|p{2.2cm}|p{5.0cm}|c|p{1.45cm}|}
		\hline
		\text{Name  (Index)} & \text{Local Dynkin diagram} &\multicolumn{2}{|c|}{\text{Coweights }}\\
		\hline
		$B$-$C_n$ (${}^2 A^{(1)}_{2n-1, n}$) & 	\begin{tikzpicture}[baseline=0]
		\node[]  at (0,0) {$\circ$};
		\node[below] at (0, 0) {$1$};
		\node[] at (4,0) {$\circ$};
		\node[below] at (4.15, 0) {$2n-1$};
		\node at (1, 0) {$\circ$};
		\node[below] at (1, 0) {$2$};
		\node at (3, 0) {$\circ$};
		\node[below] at (2.85, 0) {$2n-2$};
		\node at (2, 0.5) {$\circ$};
		\node[right] at (2.2, 0.5) {$0$};
		\draw[-] (0.1, 0) to (0.9,0);
		\draw[-] (0.1, 0.1) to (1.9, 0.4);
		\draw[-] (2.1, 0.4) to (3.9, 0.1); 
		\draw[dashed] (1.1,0) to (2.9,0);
		\draw[-] (3.1,0)  to (3.9,0);
		\end{tikzpicture} & \multicolumn{2}{|c|}{$\{\omega^\vee_i, \omega^\vee_{2n-i}\}$, $1 \le i \le n$} \\ \cline{2-4}  for $n \ge 3$  & 
		\begin{tikzpicture}[baseline=-2]
		\node at (0,0) {$\circ$};
		\node[below] at (0, 0) {$n$};
		\draw[-implies, double equal sign distance] (0.1, 0) to (0.9,0);
		\node  at (1,0) {$\circ$};
		\node[below] at (1, 0) {$n-1$};
		\draw (1.1, 0) to (1.9,0);
		\node at (2,0) {$\circ$};
		\node[below] at (2, 0) {$n-2$};	
		\draw[dashed] (2.1, 0) to (2.9, 0);
		\node at (3.0, 0) {$\circ$};
		\node[below] at (3, 0) {$2$};
		\draw[-] (3.1,0.1) to (3.9, 0.35);
		\node at (4, 0.4) {$\circ$};
		\node [right] at (4, 0.4) {$1$};
		\draw[-] (3.1, - 0.1) to (3.9, -0.35);
		\node at (4,-0.4) {$\circ$};
		\node [right] at (4, -0.4) {$0$};
		\end{tikzpicture}
		& \multicolumn{2}{|c|}{$\o^\vee_i$} \\
		\hline
		$C$-$BC_n$ (${}^2 A^{(1)}_{2n, n}$) & 	\begin{tikzpicture}[baseline=0]
		\node[]  at (0,0) {$\circ$};
		\node[below] at (0, 0) {$1$};
		\node[] at (4,0) {$\circ$};
		\node[below] at (4.15, 0) {$2n$};
		\node at (1, 0) {$\circ$};
		\node[below] at (1, 0) {$2$};
		\node at (3, 0) {$\circ$};
		\node[below] at (2.85, 0) {$2n-1$};
		\node at (2, 0.5) {$\circ$};
		\node[right] at (2.2, 0.5) {$0$};
		\draw[-] (0.1, 0) to (0.9,0);
		\draw[-] (0.1, 0.1) to (1.9, 0.4);
		\draw[-] (2.1, 0.4) to (3.9, 0.1); 
		\draw[dashed] (1.1,0) to (2.9,0);
		\draw[-] (3.1,0)  to (3.9,0);
		\end{tikzpicture} & $\{\omega^\vee_i, \omega^\vee_{2n+1-i}\}$, $1 \le i < n$ & $\{\omega^\vee_n, \omega^\vee_{n+1}\}$ \\ \cline{2-4} for $n \ge 2$ & 
		\begin{tikzpicture}[baseline=0]
		\node at (0,0) {$\circ$};
		\node[below] at (0, 0) {$0$};
		\node[above] at (0,0) {\phantom{A}};
		\draw[implies-, double equal sign distance] (0.1, 0) to (0.9,0);
		\node  at (1,0) {$\circ$};
		\node[below] at (1, 0) {$1$};
		\draw (1.1, 0) to (1.9,0);
		\node at (2,0) {$\circ$};
		\node[below] at (2, 0) {$2$};
		\draw[dashed] (2.1, 0) to (2.9, 0);
		\node at (3, 0) {$\circ$};
		\node[below] at (3, 0) {$n-1$};
		\node at (4, 0) {$\circ$};
		\node[below] at (4, 0) {$n$};
		\draw[implies-, double equal sign distance] (3.1, 0) to (3.9,0);
		\end{tikzpicture}
		&  $\omega^\vee_i$  & $2\omega^\vee_n$ \\
		\hline
        \multirow{2}{*}{$C$-$BC_1$ (${}^2 A_{2, 1}^{(1)}$)} & \begin{tikzpicture}[baseline=0]
        \node[]  at (0,0) {$\circ$};
        \node[below] at (0, 0) {$1$};
        \node at (2, 0) {$\circ$};
        \node[below] at (2, 0) {$2$};
        \node at (1, 0.5) {$\circ$};
        \node[right] at (1.2, 0.5) {$0$};
        \draw[-] (0.1, 0) to (1.9,0);        
        \draw[-] (0.1, 0.1) to (0.9, 0.4);
        \draw[-] (1.1, 0.4) to (1.9, 0.1); 
        \end{tikzpicture} & 
        \multicolumn{2}{|c|}{$\{\o^\vee_1, \o^\vee_2\}$}
        \\ \cline{2-4}
        & \begin{tikzpicture}[baseline=0]
        \node[]  at (0,0) {$\circ$};
        \node[below] at (0, 0) {$1$};
        \node[above] at (0,0) {\phantom{a} };
        \node at (1, 0) {$\circ$};
        \node[below] at (1, 0) {$0$};
        \draw[implies-, line width=3] (0.1, 0) to (0.9, 0);
        \end{tikzpicture} & \multicolumn{2}{|c|}{$2 \omega^\vee_1$}
        \\
        \hline 
        $C$-$B_n$ (${}^2 D_{n+1, n}^{(1)}$) & \begin{tikzpicture}[baseline=0]
        \node at (0,0.4) {$\circ$};
        \node [left] at (0, 0.4) {1};
        \node [left] at (0, -0.4) {0};
        \node at (0, -0.4) {$\circ$};
        \node at (1,0) {$\circ$};
        \node[below] at (1, 0) {$2$};
        \draw[-] (0.1, 0.35) to (0.9, 0.1);
        \draw[-] (0.1, -0.35) to (0.9, -0.1);
        \draw (1.1, 0) to (1.9,0);
        \node at (2,0) {$\circ$};
        \node[below] at (2, 0) {$3$};	
        \draw[dashed] (2.1, 0) to (2.9, 0);
        \node at (3.0, 0) {$\circ$};        \node[below] at (2.8, 0) {$n-1$};
        \draw[-] (3.1,0.1) to (3.9, 0.35);
        \node at (4, 0.4) {$\circ$};
        \node [right] at (4, 0.4) {$n$};
        \draw[-] (3.1, - 0.1) to (3.9, -0.35);
        \node at (4,-0.4) {$\circ$};
        \node [right] at (4, -0.4) {$n+1$};
        \end{tikzpicture} & 
        $\omega^\vee_1$ & $\{\omega^\vee_n, \omega^\vee_{n+1}\}$
        \\ \cline{2-4} for $n \ge 2$
        &
		\begin{tikzpicture}[baseline=0]
        \node at (0,0) {$\circ$};
        \node[below] at (0, 0) {$0$};
        \node[above] at (0,0) {\phantom{A} };
        \draw[implies-, double equal sign distance] (0.1, 0) to (0.9,0);
        \node  at (1,0) {$\circ$};
        \node[below] at (1, 0) {$1$};
        \draw (1.1, 0) to (1.9,0);
        \node at (2,0) {$\circ$};
        \node[below] at (2, 0) {$2$};
        \draw[dashed] (2.1, 0) to (2.9, 0);
        \node at (3, 0) {$\circ$};
        \node[below] at (3, 0) {$n-1$};
        \node at (4, 0) {$\circ$};
        \node[below] at (4, 0) {$n$};
        \draw[-implies, double equal sign distance] (3.1, 0) to (3.9,0);
        \end{tikzpicture}
		& $\omega^\vee_1$ & $\omega^\vee_n$ \\      
		\hline 
		\multirow{2}{*}{$F^1_4$ (${}^2 E_{6, 4}^{2}$)} & \begin{tikzpicture}[baseline=0]
		\node at (0,0) {$\circ$};
		\node[below] at (0, 0) {$0$};
		\node at (1,0) {$\circ$};
		\node[below] at (1, 0) {$2$};	
		\draw (0.1, 0) to (0.9,0);		
		\draw (1.1, 0) to (1.9,0);		
		\node at (2.0, 0) {$\circ$};        \node[below] at (2, 0) {$4$};
		\draw[-] (2.1,0.1) to (2.9, 0.35);
		\node at (3, 0.4) {$\circ$};
		\node [below] at (3, 0.4) {$3$};
		\draw[-] (2.1, - 0.1) to (2.9, -0.35);
		\node at (3,-0.4) {$\circ$};
		\node [below] at (3, -0.4) {$5$};		
		\node at (4, 0.8) {$\circ$};
		\node [right] at (4, 0.8) {$1$};
		\draw[-] (3.1, -0.45) to (3.9, -0.75);		\draw[-] (3.1, 0.45) to (3.9, 0.75);
		\node at (4,-0.8) {$\circ$};
		\node [right] at (4, -0.8) {$6$};
		\end{tikzpicture} & 
		\multicolumn{2}{|c|}{$\{\omega^\vee_1, \omega^\vee_6\}$}
		\\ \cline{2-4}
		&
		\begin{tikzpicture}[baseline=0]
		\node at (0,0) {$\circ$};
		\node[below] at (0, 0) {$0$};
		\node[above] at (0,0) {\phantom{A}};
		\draw[] (0.1, 0) to (0.9,0);
		\node  at (1,0) {$\circ$};
		\node[below] at (1, 0) {$1$};
		\draw (1.1, 0) to (1.9,0);
		\node at (2,0) {$\circ$};
		\node[below] at (2, 0) {$2$};
		\draw[] (3.1, 0) to (3.9, 0);
		\node at (3, 0) {$\circ$};
		\node[below] at (3, 0) {$3$};
		\node at (4, 0) {$\circ$};
		\node[below] at (4, 0) {$4$};
		\draw[implies-, double equal sign distance] (2.1, 0) to (2.9,0);
		\end{tikzpicture} &
		\multicolumn{2}{|c|}{$\omega^\vee_1$} \\
		\hline 
		\multirow{2}{*}{$G^1_2$} & \begin{tikzpicture}[baseline=0]
		\node at (0,0.4) {$\circ$};
		\node [left] at (0, 0.4) {1};
		\node [left] at (0, -0.4) {0};
		\node at (0, -0.4) {$\circ$};
		\node at (1,0) {$\circ$};
		\node[below] at (1, 0) {$2$};
		\draw[-] (0.1, 0.35) to (0.9, 0.1);
		\draw[-] (0.1, -0.35) to (0.9, -0.1);
		\draw[-] (1.1,0.1) to (1.9, 0.35);
		\node at (2, 0.4) {$\circ$};
		\node [right] at (2, 0.4) {3};
		\draw[-] (1.1, - 0.1) to (1.9, -0.35);
		\node at (2,-0.4) {$\circ$};
		\node [right] at (2, -0.4) {$4$};
		\end{tikzpicture} & 
		\multicolumn{2}{|c|}{$\{\omega^\vee_1, \omega^\vee_3, \omega^\vee_4\}$}
		\\ \cline{2-4} (${}^3 D_{4, 2}$ or ${}^6 D_{4, 2}$)
		&
		\begin{tikzpicture}[baseline=0]
		\node at (0,0) {$\circ$};
		\node[below] at (0, 0) {$0$};
		\node[above] at (0,0) {\phantom{A}};
		\draw[] (0.1, 0) to (0.9,0);
		\node  at (1,0) {$\circ$};
		\node[below] at (1, 0) {$2$};
	    \draw[implies-, double equal sign distance]  (1.1, 0) to (1.9,0);
	    \draw[implies-] (1.1, 0) to (1.9, 0);
		\node at (2,0) {$\circ$};
		\node[below] at (2, 0) {$1$};
		\end{tikzpicture}
		& 		
		\multicolumn{2}{|c|}{$\omega^\vee_2$} \\
		\hline
	\end{tabular}
\end{center}

\end{footnotesize}

\smallskip
\medskip

From this list we deduce the following statement. 

\begin{lemma}\label{redtoenhT}
	Two LM triples $(G, \{\mu\}, K)$ and $(G', \{\mu'\}, K')$ over $F$, with $G$ and $G'$ absolutely simple adjoint, define the same enhanced Tits datum if and only if they become isomorphic after scalar extension to an unramified extension of $F$.\qed
\end{lemma}

Suppose that $G$ and $G'$ are absolutely simple adjoint such that $G\otimes_F\breve F\simeq G\otimes_F\breve F$.
The isomorphism classes of $G$ and $G'$ are distinguished by considering the corresponding action of the automorphism $F$ of the local Dynkin diagram $\tilde \Delta$ of $\breve G\simeq \breve G'\simeq G^*\otimes_F\breve F$ given by Frobenius (see \cite{T}, \cite{Gross}).
In \cite{Gross} one can find a very useful list of all possible such actions and of the corresponding forms of the group. The 
parahoric subgroups $K$, $K'$ correspond to non-empty $F$-stable subsets $\tilde K$ of the vertices of $\tilde \Delta$.

\begin{example}\label{Titsexotic}
Consider the enhanced Tits data defined by LM triples of exotic good reduction type, cf. beginning of Subsection \ref{ss:goodred}. Assume that $G\otimes_F\breve F$ is absolutely simple and adjoint. There are two cases: 

1) $G$ is the adjoint group of $U(V)$, where $V$ is the $\tilde F/F$-hermitian vector space for a (tamely) ramified quadratic extension $\tilde F$ of $F$. If $\dim V=2m\geq 4$ is even, then the corresponding  enhanced Tits datum is ($B$-$C_m$, $\omega_1^\vee, \{0\})$ for $m\geq 3$ and ($C$-$B_2$, $\omega_2^\vee, \{0\})$ for $m=2$. If $\dim V=2m+1\geq 3$ is odd, then the corresponding  enhanced Tits datum is ($C$-$BC_m$, $\omega_1^\vee, \{0\})$ for $m\geq 2$ and ($C$-$BC_1$, $2\omega_1^\vee, \{0\})$ for $m=1$.  

2) $G$ is the adjoint group of ${\mathrm {SO}}(V)$
where $V$ is an orthogonal $F$-vector space of dimension $2m+2\geq 6$. Then $V$ has Witt index $m$ and non-square discriminant.  The corresponding  enhanced Tits datum is ($C$-$B_m$, $\omega_m^\vee, \{0\})$.
\end{example}

\subsection{Semi-stable reduction}\label{ss:ssred}
In the classification problem of all triples $(G, \{\mu\}, K)$ such that $\Mloc$ has semi-stable reduction, Lemma \ref{devislem} points to two problems. First,  the product of  semi-stable schemes is semi-stable only when all factors except at most one are smooth. And we can consider the problem of classifying the good reduction cases as solved by Theorem \ref{maingoodred}. Second, the extension of scalars of a semi-stable scheme is again semi-stable only if the base extension is unramified. Therefore, we will consider in the classification problem of semi-stable reduction only  triples $(G, \{\mu\}, K)$ such that   $G$  is an  absolutely simple adjoint group.

Lemma \ref{redtoenhT} justifies  classifying local models $\Mloc$ with semi-stable reduction by the enhanced Tits datum  associated to $(G, \{\mu\}, K)$. Indeed,  for $F'/F$ unramified, $\Mloc\otimes_{O_E} O_{E'}\simeq \BM^\loc_{K'}(G', \{\mu'\})$, where $G'=G\otimes_F F'$ and  $\{\mu'\}$ and $K'$ are induced from $\{\mu\}$ and $K$, cf. Proposition \ref{proofLM}, (ii).  Furthermore, $\Mloc\otimes_{O_E}O_{ E'}$ has semi-stable reduction if and only if $\Mloc$ has semi-stable reduction (this follows because the reflex field $E'$ is an unramified extension of $E$).

Now we can state the classification of local models with semi-stable reduction. 
\begin{theorem}\label{mainssred}
Let $(G, \{\mu\}, K)$ be a LM triple over $F$ such that $G$ splits over a tame extension of $F$. Assume $p\neq 2$.  Assume also that the group $G$ is adjoint and absolutely simple. 
The local model $\Mloc$ has  semi-stable but not smooth reduction over $\Spec(O_E)$ if and
only if the enhanced Tits datum corresponding to $(G,\{\mu\}, K)$ appears in the first column of the table below. 
\medskip
\begin{footnotesize}
\begin{center}
\begin{tabular}{|c|p{5.8cm}|c|}
\hline
\text{Enhanced Tits datum} & \text{Linear algebra datum} &\text{Discoverer}\\
\hline
	\begin{tikzpicture}[baseline=0]
		\node[]  at (0,0) {$\circ$};
		\node[above] at (0, 0) {$\times$};
		\node[] at (2,0) {$\circ$};
		\node at (0.5, 0) {$\circ$};
		\node at (1.5, 0) {$\circ$};
		\node at (1, 0.5) {$\square$};
		\draw[-] (0.1, 0) to (0.4,0);
		\draw[-] (0.1, 0.1) to (0.9, 0.4);
		\draw[-] (1.1, 0.4) to (1.9, 0.1); 
		\draw[dashed] (0.6,0) to (1.4,0);
		\draw[-] (1.6,0)  to (1.9,0);
	         \node[align=center, below] at (1, 0)%
		     {All vertices are hyperspecial \\ $\# \tilde K \ge 2$};
	\end{tikzpicture}
 &  \text{Split $\SL_n$, $r=1$} \newline \text{arbitrary chain of lattices of length $\ge 2$} & \text{Drinfeld} \\
\hline
	\begin{tikzpicture}[baseline=0]
		\node[]  at (0,0) {$\bullet$};
		\node[] at (2,0) {$\circ$};
		\node at (0.5, 0) {$\bullet$};
		\node at (1.5, 0) {$\circ$};
		\node at (1, 0.5) {$\square$};
		\draw[-] (0.1, 0) to (0.4,0);
		\draw[-] (0.1, 0.1) to (0.9, 0.4);
		\draw[-] (1.1, 0.4) to (1.9, 0.1); 
		\draw[dashed] (0.6,0) to (1.4,0);
		\draw[-] (1.6,0)  to (1.9,0);
		\node[align=center, below] at (1, 0)%
		     {All vertices are hyperspecial \\ $\mu$ is any minuscule coweight};
	\end{tikzpicture}
 & \text{Split $\SL_n$ with $n \ge 4$} \newline \text{$r$ arbitrary, $(\Lambda_0,  \Lambda_1)$}& \text{G\"ortz} \\
\hline	
\begin{tikzpicture}[baseline=0]
		\node at (0,0) {$\bullet$};
		\draw[implies-, double equal sign distance] (0.1, 0) to (0.4,0);
		\node  at (0.5,0) {$\circ$};
		\draw (0.6, 0) to (0.9,0);
		\node at (1,0) {$\circ$};	
		\draw[dashed] (1.1, 0) to (1.9, 0);
		\node at (2.0, 0) {$\circ$};
		\draw[-] (2.1,0.1) to (2.4, 0.35);
		\node at (2.5, 0.4) {$\blacksquare$};
		\node [right] at (2.5, 0.4) {hs};
		\draw[-] (2.1, - 0.1) to (2.4, -0.35);
		\node at (2.5,-0.4) {$\circ$};
		\node [right] at (2.5, -0.4) {hs};
		\node [above] at (2.5, -0.4) {$\times$};
	\end{tikzpicture}
	& \text{Split $\SO_{2n+1}$ with $n \ge 3,r=1, (\Lambda_0,  \Lambda_n)$} & \text{new}\\
\hline
\begin{tikzpicture}[baseline=0]
		\node at (0,0) {$\blacksquare$};
		\draw[-implies, double equal sign distance] (0.1, 0) to (0.4,0);
		\node  at (0.5,0) {$\bullet$};
		\draw (0.6, 0) to (0.9,0);
		\node at (1,0) {$\circ$};
		\draw[dashed] (1.1, 0) to (1.9, 0);
		\node at (2, 0) {$\circ$};
		\node at (2.5, 0) {$\circ$};
		\draw[implies-, double equal sign distance] (2.1, 0) to (2.4,0);
		\node [right] at (2.5, 0) {hs};
		\node [left] at (-0.1, 0) {hs};
		\node [above] at (2.5, 0) {$\times$};
	\end{tikzpicture}
	& \text{Split $\Sp_{2n}$ with $n \ge 2$, $r=n, (\Lambda_0, \Lambda_1)$}& \text{Genestier-Tilouine}\\
\hline
\begin{tikzpicture}[baseline=0]
		\node at (0,0.4) {$\blacksquare$};
		\node [left] at (-0.1, 0.4) {hs};
		\node [left] at (0, -0.4) {hs};
		\node at (0, -0.4) {$\circ$};
		\node  at (0.5,0) {$\circ$};
		\draw[-] (0.1, 0.35) to (0.4, 0.1);
		\draw[-] (0.1, -0.35) to (0.4, -0.1);
		\draw (0.6, 0) to (0.9,0);
		\node at (1,0) {$\circ$};	
		\draw[dashed] (1.1, 0) to (1.9, 0);
		\node at (2.0, 0) {$\circ$};
		\draw[-] (2.1,0.1) to (2.4, 0.35);
		\node at (2.5, 0.4) {$\bullet$};
		\node [right] at (2.5, 0.4) {hs};
		\draw[-] (2.1, - 0.1) to (2.4, -0.35);
		\node at (2.5,-0.4) {$\circ$};
		\node [right] at (2.5, -0.4) {hs};
		\node [above] at (0, -0.4) {$\times$};
	\end{tikzpicture}
 & \text{Split $\SO_{2n}$ with $n \ge 4$, $r=1, (\Lambda_0, \Lambda_n)$}& \text{Faltings} \\
\hline
\begin{tikzpicture}[baseline=0]
		\node at (0,0.4) {$\blacksquare$};
		\node [left] at (-0.1, 0.4) {hs};
		\node [left] at (0, -0.4) {hs};
		\node at (0, -0.4) {$\bullet$};
		\node  at (0.5,0) {$\circ$};
		\draw[-] (0.1, 0.35) to (0.4, 0.1);
		\draw[-] (0.1, -0.35) to (0.4, -0.1);
		\draw (0.6, 0) to (0.9,0);
		\node at (1,0) {$\circ$};	
		\draw[dashed] (1.1, 0) to (1.9, 0);
		\node at (2.0, 0) {$\circ$};
		\draw[-] (2.1,0.1) to (2.4, 0.35);
		\node at (2.5, 0.4) {$\circ$};
		\node [right] at (2.5, 0.4) {\,hs};
		\draw[-] (2.1, - 0.1) to (2.4, -0.35);
		\node at (2.5,-0.4) {$\circ$};
		\node [right] at (2.5, -0.4) {hs};
		\node [above] at (2.5, -0.4) {$\times$};
	\end{tikzpicture}
 & \text{Split $\SO_{2n}$ with $n \ge 5$, $r=n, \Lambda_1$}& \text{new} \\
\hline
\end{tabular}
\end{center}
\end{footnotesize}
\bigskip

 In the second column, we list the linear algebra data that correspond\footnote {By definition, this means that the corresponding parahoric subgroup is the \emph{connected} stabilizer of the listed lattices. Note that in the last row, the connected stabilizer of the lattice $\Lambda_1$ also stabilizes $\Lambda_0$.} to the LM triple $(G,\{\mu\}, K)\otimes_F\breve F$.
 
 In the diagrams above, if not specified, hyperspecial vertices are marked with an hs. In order to also show
the coweight $\{\lambda\}$, a special vertex is specified (marked by a square)\footnote{Note that the local Dynkin type $C$-$BC_n$ does not occur here so that all special vertices are conjugate; hence this specification plays no role.} so that the extended affine Weyl group appears as a semi-direct product of $W_0$ and $X_*$. Then $\{\lambda\}$ is equal to the fundamental coweight of the vertex marked with $\times$. The number $r$ is the labeling of this special vertex. Finally,  the  subset $\tilde K$ is the set of vertices filled with black color.

 Note that there are some obvious overlaps between the first two rows. 
\end{theorem}

\begin{remark}\label{Frobeniusaction} 
Starting with the table in Theorem \ref{mainssred} above, one can also easily  list all LM triples $(G, \{\mu\}, K)$ over $F$,  with $G$ adjoint and absolutely simple such that $G$ splits over a tame extension of $F$ and with $\Mloc$ having semi-stable reduction over $O_E$ (provided $p\neq 2$). These are given by listing the possible conjugacy classes of Frobenius automorphisms in the group ${\rm Aut}(\tilde \Delta, \tilde K)$ of automorphisms of the corresponding local Dynkin diagram $\tilde \Delta$ that preserve the  black subset $\tilde K$.  
B.~Gross \cite{Gross} gives a convenient enumeration of possible Frobenius conjugacy classes in ${\rm Aut}(\tilde \Delta)$.

For example, in the first case of our list, there could be several possible Frobenius actions on the $n$-gon
that stabilize $\tilde K$ depending on that set; the corresponding groups are the adjoints of either unitary groups or  of ${\rm SL}_m(D)$, where $D$ are division algebras and $m|n$ (see \cite[p. 15-16]{Gross}). 

In the second case, there is only one possibility of a non-trivial Frobenius action on the $n$-gon that stabilizes the set of two adjacent vertices: A reflection ($F$ of order $2$). Then $G$ is the adjoint group
of $U(V)$ where $V$ is a non-degenerate Hermitian space for an unramified quadratic extension of $F$. Furthermore, when $n=2m$ is even, $F$ cannot fix a vertex so $V$ \emph{does not} contain an isotropic subspace of dimension $m$ (\cite[p. 16]{Gross}). 

In the third and fourth cases, there are no non-trivial automorphisms $F$ that preserve the subset $\tilde K$ and so $G$ is split. 

In the fifth case, there is also only one possible non-trivial Frobenius action that stabilizes $\tilde K$, up to conjugacy in the group ${\rm Aut}(\tilde \Delta, \tilde K)$. The corresponding group is the adjoint group of ${ {U}}(W)$ where $W$ is a non-degenerate anti-Hermitian space over the quaternion division algebra over $F$; the center of the Clifford algebra is $F\times F$ if $n$ is even and the quadratic unramified extension $L/F$ if $n$ is odd
(\cite[p. 18-20]{Gross}).

In the sixth case, there are three possibilities of a non-trivial Frobenius action that stabilizes $\tilde K$, up to conjugacy in the group ${\rm Aut}(\tilde \Delta, \tilde K)$. In the one case, the group is the adjoint group of ${\mathrm {SO}}(V)$ where $V$ is a non-degenerate  orthogonal space of dimension $2n$, discriminant $1$ and Witt index $n-2$. In the other two, the group is the adjoint 
group of the unramified quasi-split but not split ${\mathrm {SO}}(V)$ (\cite[p. 18-20]{Gross}).

In all these cases, we can realize $K$ as the parahoric stabilizer of a suitable lattice chain.

 \end{remark}

\begin{remark}\label{strictversusnonstrict}
We note that $\Mloc$ has semi-stable reduction if and only if the base change $\Mloc\otimes_{O_E}\br O_E$ has strictly semi-stable reduction, i.e., the geometric special fiber is a strict normal crossings divisor, in the sense of \cite[Def. 40.21.1]{Stacks}:   Indeed, both $\Mloc\otimes_{O_E}\br O_E$ and all the irreducible components of its special fiber are normal \cite{PZ}, hence unibranch at each closed point $x$.
From this we deduce that each
intersection of a subset of irreducible components of the geometric special fiber in the strict 
henselization of  $\Mloc$ at   $x$ (i.e., of ``branches"), is isomorphic to the strict henselization of the   intersection of a corresponding subset of global irreducible components at $x$. Therefore, if the geometric special fiber is (\'etale locally) a normal crossings divisor, it is in fact (globally) a strict normal crossings divisor.
\end{remark}

\begin{remarks}\label{examples}
Let us compare this list with the local models investigated in earlier papers. We always assume $p\neq 2$. We use the terminology \emph{rationally smooth, strictly pseudo semi-stable reduction, rationally strictly pseudo semi-stable reduction} introduced in the next section. 

\smallskip

\noindent (i) Let us consider the LM triples whose first two components are $G=\GU(V)$ where $V$ is a split $F'/F$-hermitian space of dimension $3$ relative to a \emph{ramified} quadratic extension $F'/F$, and where $\{\mu\}=(1, 0, 0)$. We identify $E$ with $F'$.  We use the notation for the parahoric subgroups as in \cite{PRS}. Since $G$ is not unramified, there are no hyperspecial maximal parahoric subgroups. If $K$ is the stabilizer of the self-dual vertex lattice $\Lambda_0$, then $K$ is a special maximal parahoric and the special fiber is irreducible, normal with an isolated singularity which is a rational singularity, comp. \cite[Thm. 2.24]{PRS}. The special fiber occurs in the list in \cite{HR} of rationally smooth Schubert varieties in twisted affine Grassmannians. The blow-up of $\Mloc$ in the unique singular point of the special fiber has semi-stable reduction, cf. \cite[Thm. 4.5]{Pappas}, \cite{Kr}. This is an example of a local model which does not have semi-stable reduction but where the generic fiber has a different model which has  semi-stable reduction.

If $K$ is the stabilizer of the non-selfdual vertex lattice $\Lambda_1$, then $\Mloc$ is smooth over $\Spec O_{F'}$: this case is of exotic good reduction type. 

Finally, if $K$ is an Iwahori subgroup, then the local model does not have rationally strictly pseudo semi-stable reduction, comp. \cite[Thm. 2.24, (iii)]{PRS}. And, indeed, this case is eliminated in Subsection \ref{7.9}.

\smallskip

\noindent (ii) Let us consider $G=\GU(V)$, where $V$ is a split $F'/F$-hermitian space of arbitrary dimension $n\geq 2$ relative to a ramified quadratic extension $F'/F$. Let us consider the LM triple $(G, \{\mu\}, K)$,  where $\{\mu\}=(1, 0, \ldots, 0)$, and where $K$ is the parahoric stabilizer of a self-dual lattice $\Lambda$ (except when $n=2$, $K$ is the full stabilizer of $\Lambda$, cf. \cite[1.2.3]{PR}). If $n=2$, then 
$\Mloc$ has semi-stable reduction, cf. \cite[Rm. 2.35]{PRS}. If $n\geq 3$, the special fiber of $\Mloc$ is irreducible and has a unique isolated singular point, cf. \cite[Thm. 4.5]{Pappas}. Generalizing the previous example, the blow-up of this singular point has semi-stable reduction, cf. \cite{Pappas, Kr}. 

For $n>3$ with $n=2m+1$ odd, the associated local Dynkin diagram is of type $C$-$BC_{m}$ and the parahoric subgroup $K$ corresponds to the \emph{special} vertex $m$ in the local Dynkin diagram. The special fiber of the local model is a Schubert variety  that  occurs in the list in \cite{HR} of rationally smooth Schubert varieties in twisted affine Grassmannians.  Remarkably,  Zhu \cite[Cor.~7.6]{ZhugeoSat} has shown in this case that the Weil sheaf defined by the complex of nearby cycles is the constant sheaf $\BQ_\ell$, even though the special fiber is singular. In particular, as shown previously by Kr\"amer \cite[Thm.~5.4]{Kr},  the \emph{semi-simple Frobenius trace} function is constant equal to $1$ on the special fiber. 

For  $n=2m\geq 4$ even, the associated local Dynkin diagram is of type $B$-$C_{m}$   and the parahoric subgroup $K$ corresponds to the \emph{non-special} vertex $m$ in the local Dynkin diagram if $m\geq 3$, or $C$-$B_2$ and the \emph{non-special} vertex $1$, if $m=2$. By \S \ref{7.5.2}, resp. Subsection \ref{ss:Cnfirst}, the associated Poincar\'e polynomial is not symmetric and hence the special fiber is not rationally smooth, cf. Lemma \ref{rpssPoinc}. In this case, Kr\"amer \cite[Thm. 5.4]{Kr} has shown that the {semi-simple Frobenius trace function} is not constant equal to $1$ on the special fiber, but rather has a jump at the singular point.  

\smallskip

\noindent (iii) Let us consider $G=\Res_{F'/F}(\GL_n)$, where $F'/F$ is a totally ramified (possibly wildly) extension. This  is excluded from the above considerations (both for the classification of good reduction and of semi-stable reduction); still, it is interesting to compare this case with the above lists. Let $K=\GL_n(O_{F'})$ and
\[
 \{\mu\}=\big( (1^{(r_\varphi)}, 0^{(n-r_\varphi)})_{\varphi\colon F'\to \ov F}\big).
 \]    The singularities of the special fiber are analyzed in \cite{PR1} by relating the special fiber $\Mloc\otimes_{O_E}\kappa_E$  with a Schubert variety in the affine Grassmannian for $\GL_n$. More precisely, the special fiber is irreducible and reduced and there is an isomorphism of closed reduced subschemes 
$$
\Mloc\otimes_{O_E}\kappa_E\simeq \ov\CO_{{\bf t}} .
$$
Here $ \ov\CO_{{\bf t}}$ is the Schubert variety associated to the dominant coweight $ {\bf t}={\bf r}^\vee$ dual to ${\bf r}= (r_\varphi)_\varphi$, i.e.,
$$
t_1=\#\{\varphi\mid r_\varphi\geq 1 \},\quad  t_2=\#\{\varphi\mid r_\varphi\geq 2 \}, \ \ldots . 
$$ 
By \cite{HR} (cf. also \cite{MOV} for the analogue over a ground field of characteristic zero, and \cite{EM}, \cite{Zhufixed} for the analogue over $\BC$), $\ov\CO_{\bf t}$ is smooth if and only if ${\bf t}$ is minuscule, i.e., $t_1-t_n\leq 1$. This holds if and only if there is at most one $\varphi$ such that  $r_\varphi\notin \{0, n\}$. We conclude that $\Mloc$ is smooth only in the trivial case when at most one $r_\varphi$ is not $0$ or $n$.

\smallskip

\noindent (iv) Very similarly to the case above, we can also consider $G={\rm Res}_{F'/F}(H)$, where   $F'$ is a totally ramified (possibly wildly) extension, and $H$ is unramified over $F'$
(i.e., quasi-split and split over an unramified extension of $F'$). Then $H$ extends to a reductive group scheme over $O_{F'}$ which is unique up to isomorphism and which we will also denote by $H$. Take 
 $K=H(O_{F'})$, let $\{\mu\}=\big(( \mu_\varphi)_{\varphi\colon F'\to \ov F}\big)$, and consider 
 the LM triple $(G,\{\mu\}, K)$.
 
 When $F'/F$ is wildly ramified, the theory of \cite{PZ} does not apply to $(G,\{\mu\}, K)$. However, Levin \cite{Levin} has extended the construction of \cite{PZ} to such groups obtained by restriction of scalars and has defined local models $\Mloc$ for such triples. Assume that $p$ does not divide $|\pi_1(H_{\rm der})|$. Then, by \cite[Thm. 2.3.5]{Levin}, the
geometric special fiber $\Mloc\otimes_{O_E}k$ is reduced and can be identified with a Schubert variety ${\rm Gr}_{H, \lambda}$ of the affine Grassmannian  for $H$ over $k$. Here, $\lambda$ is given by the sum $\sum_{\varphi} \mu_{\varphi}$ of the minuscule coweights $\mu_{\varphi}$. By \cite{HR}, (or \cite{MOV} for the analogue over a ground field of characteristic zero), ${\rm Gr}_{H, \lambda}$ is smooth if and only if $\lambda$ is minuscule. Therefore, $\Mloc$ is smooth over $O_E$ if and only if at most one of the coweights $\mu_{\varphi, \rm ad}$ is not trivial.
 
\end{remarks}

The proof of Theorem \ref{mainssred} proceeds in four steps. In a first step, we establish a list of all cases which satisfy the \emph{component count property} condition  (CCP), cf. Section \ref{s:CCP}. This condition is implied by \emph{strictly pseudo semi-stable reduction}. This last condition, concerns only the special fiber and entails in particular that all irreducible components are smooth, with  their intersections smooth of the correct dimension, cf. Section \ref{s:pssredCCP}. By weakening the condition of smoothness to \emph{rational smoothness}, we arrive at the notion of  \emph{rationally strictly pseudo semi-stable reduction}, cf. Section \ref{s:pssredCCP}. The second step consists in eliminating from the CCP-list all cases which do not have rationally strictly pseudo semi-stable reduction, cf. Section \ref{s:ratpsss}. In a third step, we eliminate all cases which have  rationally strictly pseudo semi-stable reduction but not strictly pseudo semi-stable reduction, cf. Section \ref{s:psssred}. In the final step we prove that in all the remaining cases strictly pseudo semi-stable reduction implies semi-stable reduction. This last step is a lengthy case-by-case analysis through linear algebra and occupies  Section \ref{s:converse1}. 

\section{Strictly pseudo semi-stable reduction and the CCP condition}\label{s:pssredCCP}

\begin{definition}\label{def:pseudorat}
a) A scheme over the spectrum of a discrete valuation ring is said to have {\it strictly pseudo semi-stable reduction} (abbreviated to \emph{SPSS reduction}) if all irreducible components of the reduced geometric special fiber are smooth and of the same dimension, and the reduced intersection of any $i$ irreducible components is smooth and irreducible and of codimension $i-1$.

b) A scheme over the spectrum of a discrete valuation ring is said to have {\it rationally strictly pseudo semi-stable reduction} if all irreducible components of the reduced geometric special fiber are rationally smooth and of the same dimension, and the reduced intersection of any $i$ irreducible components is rationally smooth and irreducible and of codimension $i-1$.
 \end{definition}
 Here we recall that an irreducible  variety $Y$ of dimension $d$ over an algebraically closed field $k$ is said to be \emph{rationally smooth}\footnote{A priori, this definition depends on $\ell$. However, as we will see from the proof, the schemes we consider in this paper will be either rationally smooth for all $\ell$ or non rationally smooth for any $\ell$. We will simply use the terminology ``rationally smooth'' instead of  ``$\ell$-rationally smooth''.}
 if for all closed points $y$ of $Y$ 
 the relative $\ell$-adic cohomology (for some $\ell\neq {\rm char}\, k$) satisfies
 $$
 \dim_{\BQ_\ell} H^i(Y, Y\setminus \{y\}, \BQ_\ell)=\begin{cases}
 0& i\neq 2d\\
 1&i=2d .
 \end{cases}
 $$
When $k=\BC$, this definition (for singular cohomology with coefficients in $\BQ$) appears in \cite{KL}, cf. also \cite{B, Ku, BL}. 

We note that both notions, that of SPSS reduction and that  of rationally SPSS reduction, only depend on the geometric special fiber. For instance, they do not imply that the scheme  is regular. 

\begin{lemma}\label{rpssPoinc}
	Let $Y$ be a proper irreducible variety  of dimension $d$ over an algebraically closed field. If $Y$ is rationally smooth, then the Poincar\'e polynomial 
	$$
	P(t)=\sum\nolimits_{i=0}^{2d} a_i t^i,
	$$ 
	of cohomology with $\BQ_\ell$-coefficients ($\ell\neq {\rm char}\, k$) is symmetric, i.e., $a_i=a_{2d-i} $, for all $i$. \qed
\end{lemma} 

\begin{remark} By \cite[Prop.~2.1]{HR}, if the irreducible variety $Y$ is rationally smooth, then the intersection complex $IC_Y$ is isomorphic to $\BQ_{\ell}[d]$.  Thus the cohomology groups with $\BQ_\ell$-coefficients satisfy Poincar\'e duality. Also, in the applications in this paper, the varieties involved are unions of affine spaces and thus the polynomials $P(t)$ can be computed by counting rational points on the varieties.
\end{remark}

\begin{remark}\label{carrell}
It is proved in \cite{Car} that for Schubert varieties in the finite and affine flag varieties for split groups, the converse is true. Namely,  in this context, a Schubert variety  is rationally smooth if and only if its Poincar\'e polynomial formed with $\BQ_\ell$-coefficients is symmetric. Something analogous holds in the Kac-Moody context, cf. \cite[12.2 E(2)]{Ku2}.
\end{remark}

\begin{Notation}\label{clarify} In the rest of this section and also in Sections \ref{s:CCP}, \ref{s:ratpsss} and \ref{s:psssred} we consider the enhanced Tits datum $(\tilde\Delta, \{\lambda\}, \tilde K)$ obtained, as in \S \ref{Titsdata}, from a local model triple 
$(G, \{\mu\}, K)$ with $G$ adjoint and absolutely simple. 

On the other hand, the enhanced Tits datum $(\tilde\Delta, \{\lambda\}, \tilde K)$ also corresponds to an adjoint, absolutely simple   group $G^\flat$ over $k((u))$, a $G^\flat(k((u))^{\rm sep})$-conjugacy class of a minuscule cocharacter,
and a conjugacy class of a parahoric subgroup $K^\flat=\CG^\flat(k[[u]])$. In terms of the identifications of \S \ref{ss:WeylAdm}, we have $G^\flat=\breve G'$, $K^\flat=\breve K'$, $\CG^\flat=\underline\CG\otimes_{O[u]}k[[u]]$, and the class of the cocharacter is the one that  corresponds to $\{\mu\}$. 

By Theorem \ref{2.9}, and the above discussion,  the geometric special fiber of $\Mloc$ can be identified (up to a radicial morphism) with the union,  over the set ${\Adm}_{\tilde K}(\{\lambda\})$, of Schubert varieties in the partial flag variety $LG^\flat/L^+\CG^\flat$. 
In what follows, to ease the notation, we will denote this partial flag variety by $G^\flat/K^\flat$ and its Schubert varieties as $\overline{K^\flat w K^\flat/K^\flat}$.

Below, and also in Sections  \ref{s:CCP}, \ref{s:ratpsss} and \ref{s:psssred}, we will employ various combinatorial arguments in the extended Weyl group $\tilde W$ which
only involve $(\tilde\Delta, \{\lambda\}, \tilde K)$; for example, which use cosets for the subgroup  $W_{\tilde K}$. For these arguments, we will often omit the tilde from the notation. For example, we will simply write $W_K$ instead of $W_{\tilde K}$; in any case, this subgroup  ultimately only depends on the conjugacy class of the parahoric subgroup $\breve K\subset G(\br F)$. 
 
\end{Notation}

Let us first make Lemma \ref{rpssPoinc} explicit in the case of interest for us, namely for  affine Schubert varieties in the partial flag variety $  G^\fl/  K^\fl$. Note that for any $w \in \tW$, we have the projection map $$\overline{I^\fl v I^\fl/I^\fl} \to \overline{ K^\fl w K^\fl/K^\fl},$$ where $v=\max(W_K w W_K)$. This map is a locally trivial fiber bundle (for the \'etale toplogy) with fibers isomorphic to the smooth projective variety $K^\fl/I^\fl$. Hence $\overline{K^\fl w K^\fl/K^\fl}$ is rationally smooth if and only if $\overline{I^\fl v I^\fl/I^\fl}$ is rationally smooth. Thus we may use the Poincar\'e polynomial of $\overline{  I^\fl v  I^\fl/ I^\fl}$ to determine if $\overline{  K^\fl w  K^\fl/ K^\fl}$ is rationally smooth.

We denote by $\tW^K$ the set of elements $w \in \tW$ that are of minimal length in their coset $w W_K$. For any translation element $\l$ in $\tW$, we set \begin{equation}
W_{\le \l, K}=\{v \in \tW^K\mid  v \le \max\{W_K t^{\l} W_K\}\}.
\end{equation}
The set $W_{\le \l, K}$ contains a unique maximal element, which we denote by $w_{\l, K}$. For any $w \in W_{\le \l, K}$, we define the \emph{colength} of $w$ to be $\ell(w_{\l, K})-\ell(w)$, where   $\ell(w)$ denotes the length of $w$. 

We have $\overline{ K^\flat \l   K^\flat/ K^\flat}=\sqcup_{v \in W_{\le \l, K}}  I^\flat v K^\flat/ K^\flat$.   The associated Poincar\'e polynomial $P(t)$ for $\overline{ K^\flat \l   K^\flat/ K^\flat}$ is obtained from counting the rational points on  $\overline{ K^\flat \l   K^\flat/ K^\flat}$. Set $q=t^2$.  Then $P(t)$ equals to
\begin{equation}\label{poincW}
P_{\le \l, K}(q)=\sum\nolimits_{v \in W_{\le \l, K}} q^{\ell(v)}.
\end{equation}

On the other hand, set $v_1=\max(W_K t^\l W_K)$. Then $\overline{ I^\flat v_1  I^\flat/I^\flat}=\sqcup_{v \le W_{\le \l, K}} \sqcup_{x \in W_K}  I^\flat v x  I^\flat/ I^\flat$. The associated Poincar\'e polynomial is 
$$
\sum_{v \in W_{\le \l, K}, x \in W_K} q^{\ell(v x)}=P_{\le \l, K} (q) \sum_{x \in W_K} q^{\ell(x)}.
$$
 As $\sum_{x \in W_K} q^{\ell(x)}$ is symmetric, we deduce that $P_{\le \l, K} (q) \sum_{x \in W_K} q^{\ell(x)}$ is symmetric if and only if $P_{\le \l, K} (q)$ is symmetric. By  Lemma \ref{rpssPoinc}, we have 

\begin{proposition}\label{rs-criterion}
	If the Schubert variety $\overline{ K^\fl \l   K^\fl/  K^\fl}$ is rationally smooth, then
	$P_{\le \l, K}(q)$ is symmetric. 
\end{proposition}

\begin{definition}
The LM triple $(G, \{\mu\}, K)$ has the \emph{component count property }  (CCP condition) if the following inequality is satisfied,
$$
\#\{\text{extreme elements of $\Adm_{\tilde K}(\{\lambda\})$}\} \le \# \tilde K.
$$ 
\end{definition}
\begin{proposition}If the local model $
\Mloc$ has 
rationally SPSS reduction over $O_E$, then the CCP condition holds for the triple $(G, \{\mu\}, K)$.  
\end{proposition}

\begin{proof} Let $(\tilde\Delta, \{\lambda\}, \tilde K)$ be the associated enhanced Tits datum. 
As $\l$ is not central, there exists $\l' \in W_0 \cdot \l$ such that $\<\l', \a\> \neq 0$ for some root $\a$ of $K$. 

By Theorem \ref{2.9}, $\overline{  K^\fl \l'  K^\fl/ K^\fl}$ is an irreducible component of the geometric special fiber of $\Mloc$. Thus if $\Mloc$ has rationally SPSS reduction, then $\overline{  K^\fl \l'  K^\fl/ K^\fl}$ is rationally smooth. Therefore by Lemma \ref{rpssPoinc}, the Poincar\'e polynomial of $\overline{  K^\fl \l'  K^\fl/ K^\fl}$ is symmetric. But this coincides with the Poincar\'e polynomial of $W_{\le \l', K}$, cf. \eqref{poincW}, which is therefore symmetric.

Any length one element in $W_{\le \l', K} \subset \tW^K$ is of the form $\tau s$ for some $s \in \tilde K$, where $\tau$ is the unique length-zero element in $\tilde W$ with $t^{\l'} \in W_a \tau$. Thus there are at most $\# \tilde K$ length one elements in $W_{\le \l', K}$.  Hence there are also at most $\# \tilde K$ colength one elements of $W_{\le \l', K}$. 

Now  list the irreducible components of the geometric special fiber as $X_1= \overline{  K^\fl \l'  K^\fl/ K^\fl} $, $X_2, \ldots, X_l$. By the definition of rationally SPSS reduction, for any $i$ with $2 \le i \le l$, the intersection $X_1 \cap X_i$ is of the form $\overline{  K^\fl w_i   K^\fl/ K^\fl}$, where $w_i \in W_{\le \l', K}$ with $\dim( K^\fl w_i  K^\fl/ K^\fl)=\dim( K^\fl \l'   K^\fl/ K^\fl)-1$. In particular, $w_i$ is a colength one element in $W_{\le \l', K}$. As the intersection of any three irreducible components of the geometric special fiber is of codimension $2$, we have $w_i \neq w_j$ for $i \neq j$. In particular, $\{w_2, w_3, \ldots, w_l\} \subset W_{\le \l', K}$ is a subset of colength one elements. 

Next we construct another colength one element of $W_{\le \l', K}$. Recall that $w_{\l', K}$ is the unique maximal element of $W_{\le \l', K}$. Let $s \notin \tilde K$. Then $s w_{\l', K} \in W_K t^{\l'} W_K$. Therefore, we have either $s w_{\l', K}<w_{\l', K}$, or $s w_{\l', K}>w_{\l', K}$ and $s w_{\l', K}=w_{\l', K} s'$ for some $s' \notin \tilde K$. 

If $s w_{\l', K}>w_{\l', K}$ for all $s \notin K$, then  $W_K w_{\l', K}=w_{\l', K} W_K$. Since $w_{\l', K} \in W_K t^{\l'} W_K$, we get $W_K t^{\l'}=t^{\l'} W_K$. This contradicts  the assumption that $\<\l', \a\> \neq 0$ for some $\a \in \Phi_K$. 

Therefore there exists $s \notin K$ such that $s w_{\l', K}<w_{\l', K}$. Since $w_{\l', K} \in \tW^K$, we have $s w_{\l', K} \in \tW^K$. Hence $s w_{\l', K} \in W_{\le \l', K}$ is a colength one element. As $ K^\fl (s w_{\l', K})  K^\fl= K^\fl w_{\l', K}   K^\fl= K^\fl \l'  K^\fl$, we have $s w_{\l', K} \neq w_i$ for any $i$. 

We now have found at least $l$ distinct colength one elements in $W_{\le \l', K}$, namely $s w_{\l', K}$ and  $w_2, \ldots, w_l$. Thus we have $l \le \# \tilde K$. The proposition is proved. 
\end{proof}

\section{Analysis of the CCP condition}\label{s:CCP}

\subsection{Statement of the result} The purpose of this section is to determine for which enhanced Tits data the CCP condition is satisfied. Note that the CCP condition only depends on the associated enhanced Coxeter datum.

\begin{theorem}\label{CCP}
	Assume that $G$ is adjoint and absolutely simple. The enhanced Coxeter data satisfying the CCP condition are the following (up to isomorphism):
	
	\begin{enumerate}
		\item Irreducible cases:
		\begin{enumerate}
			\item The parahoric subgroup corresponding to $\tilde K$ is maximal special;
			
			\item The triple $(\tilde B_n, \o^\vee_r, \{n\})$ with $n \ge 3$ and $1 \le r \le n-1$; 
			
			\item The triple $(\tilde C_n, l \o^\vee_n, \{i\})$ with $n \ge 2$, $l=1$ or $2$ and $1 \le i \le n-1$;
			
			\item The triple $(\tilde F_4, \o^\vee_1, \{4\})$.
			
			\item The triple $(\tilde G_2, \o^\vee_2, \{1\})$. 
		\end{enumerate}
		
		\item Reducible cases:
		\begin{enumerate}
		\item The triple $(\tilde A_1, 2 \o^\vee_1, \{0, 1\})$;
		
		\item The triple $(\tilde A_{n-1}, \o^\vee_1, \tilde K)$ with arbitrary $\tilde K$ of cardinality $\ge 2$;
		
		\item The triple $(\tilde A_{n-1}, \o^\vee_i, \{0, 1\})$ with $n \ge 4$ and $2 \le i \le n-2$;
		
		\item The triple $(\tilde B_n, \o^\vee_1, \{0, n\})$ with $n \ge 3$;
		
		\item The triple $(\tilde B_n, \o^\vee_n, \{0, 1\})$ with $n \ge 3$;
		
		\item The triple $(\tilde C_n, \o^\vee_1, \{0, n\})$ with $n \ge 2$;
		
		\item The triple $(\tilde C_n, l \o^\vee_n, \{i, i+1\})$ with $n \ge 2$, $l=1$ or $2$ and $0 \le i \le \frac{n}{2}-1$;
		
		\item The triple $(\tilde D_n, \o^\vee_1, \{0, n\})$ with $n \ge 4$;
		
		\item The triple $(\tilde D_n, \o^\vee_n, \{0, 1\})$ with $n \ge 5$.
		\end{enumerate}
	\end{enumerate}

Here ``irreducible'' and ``reducible'' refer to the components in the special fiber.
\end{theorem}

\subsection{Classical types} We first study the classical types. Let $E=\BR^n$ with the canonical basis $(\e_1, \ldots, \e_n)$. We equip $E$ with the scalar product such that this basis is orthonormal and we identify $E$ with $E^*$. 

We regard the Weyl group $W(B_n)$ of type $B_n$ (and also $C_n$) as the group of permutations $\sigma$ on $\{\pm 1, \ldots, \pm n\}$ such that $\sigma(-i)=-\sigma(i)$ for $1 \le i \le n$. The Weyl group $W(A_{n-1})$ of type $A_{n-1}$ is the subgroup of $W(B_n)$ consisting of permutations $\sigma$ with $\sigma(i)>0$ for all $1 \le i \le n$. We have $W(A_{n-1}) \cong S_n$, the group of permutations on $\{1, 2, \ldots, n\}$. The Weyl group $W(D_n)$ of type $D_n$ is the subgroup of $W(B_n)$ consisting of permutations $\sigma$ such that  $\#\{i; 1 \le i \le n, \sigma(i)<0\}$ is an even number. 

\subsection{Type $\tilde A_{n-1}$} One may consider the extended affine Weyl group $\BZ^n \rtimes S_n$ instead. In this case, one may use the coweight $(1^r, 0^{n-r})$ instead of the fundamental coweight $\o^\vee_r$. 

It is easy to see that the triple $(\tilde A_1, 2 \o^\vee_1, \tilde K)$ with $\tilde K$ arbitrary satisfies the CCP condition. Now we assume that $\lambda=\omega^\vee_r$ for some $r$. 

By applying an automorphism, we may assume that $0 \in \tilde K$. It is easy to see that the case $\tilde K=\{0\}$ satisfies the CCP condition. Now we assume that $\# \tilde K \ge 2$. Then $\tilde K=\{0, i_1, \ldots, i_{l-1}\}$ with $l \ge 2$ and $i_1<\cdots<i_l$. Then the action of $W_K$ on $\{1, 2, \ldots, n\}$ stabilizes the subsets $\{1, \ldots, i_1\}$, $\{i_1+1, \ldots, i_2\}$, $\ldots$, $\{i_{l-1}+1, \ldots, n\}$. So for $\l=(1^r, 0^{n-r})$, the number of extreme elements equals the number of partitions $r=j_1+\cdots+j_{l}$, where $0 \le j_m \le i_m-i_{m-1}$ for any $m$. Here by convention, we set $i_0=0$ and $i_{l}=n$. Now the statement of Theorem \ref{CCP} for type $\tilde A$ follows from the following result. 

\begin{proposition}\label{par-count}
	Let $l \ge 2$ and $n, r, n_1, \ldots, n_l$, be positive integers with $n=n_1+\cdots+n_l$ and $r<n$. Set 
	$$
	X=\{(j_1, \ldots, j_l)\, \mid\,  r=j_1+\cdots+j_l,\ 0 \le j_i \le n_i \text{ for all i}\}.
	$$ Then $\# X \ge l$ and  equality holds if and only if $r=1$, or $l=2$ and $n_1$ or $n_2$ equals $1$. 
\end{proposition}

\begin{proof}
	Without loss of generality, we may assume that $r \le n/2$ and $n_1 \ge n_2 \ge \cdots \ge n_l$. Let $t \in \BZ_{>0}$ such that $n_1+\cdots+n_{t-1}<r \le n_1+\cdots+n_t$. Note that if $t=l$, then $n-r<n_l \le n_1 \le r$, which contradicts our assumption that $r \le n/2$. Therefore $t<l$. 
	
	We have $x_0=(n_1, \ldots, n_{t-1}, r-n_1-\cdots-n_{t-1}, 0, \ldots, 0) \in X$. For any $1 \le i_1 \le t, t+1 \le i_2 \le l$, we obtain a new element in $X$ from the element $x_0$ by subtracting $1$ in the $i_1$-th factor and adding $1$ in the $i_2$-th factor. In this way, we construct $1+t(l-t)$ elements in $X$. Note that $t(l-t) \ge l-1$ and the equality holds if and only if $t=1$ or $t=l-1$. Therefore, $\# X \ge l$. 
	
	Moreover, if $\# X=l$, then $t=1$ or $t=l-1$, and the elements we constructed above are all the elements in $X$. 
	
	{\it Case} (i): $t=1$. In this case, $x_0=(r, 0, \ldots, 0)$. By our construction, there is no element of the form $(r-2, j_2, \ldots, j_l)$ in $X$. This happens only when $r=1$ or $n_2+\cdots+n_l=1$. In the latter case $l=2$ and $n_2=1$. 
	
	{\it Case} (ii): $t=l-1$. In this case, $x_0=(n_1, \ldots, n_{l-2}, r-n_1-\cdots-n_{l-1}, 0)$. By our construction,
	\begin{enumerate}
		\item there is no element in $X$ with $2$ in the last factor; 
		\item there is no element in $X$ with $r-n_1-\cdots-n_{l-1}+1$ in the $l-1$ factor.
	\end{enumerate}
	Note that (1) happens only when $r=1$ or $n_l=1$ and (2) happens only when $l=2$ or $r=n_1+\cdots+n_{l-1}$. However, if $r=n_1+\cdots+n_{l-1}$ and $n_l=1$, then, since $r \le n/2$, we must have $n=2$ and $r=1$. Hence both (1) and (2) happens only when $r=1$ or $l=2$ and $n_2=1$. 
\end{proof}

\subsection{Type $\tilde B_n$}\label{type-B} By applying a suitable automorphism, we may assume that if $1 \in \tilde K$, then $0 \in \tilde K$. Let $\epsilon=\#(\{0, n\} \cap \tilde K)$.

We have $\{i_1, \ldots, i_l\} \subset \tilde K \subset \{0, i_1, \ldots, i_l, n\}$, where $1 \le i_1<\cdots<i_l \le n-1$. Then $$W_K \cong W_1 \times S_{i_2-i_1} \times \cdots \times S_{i_l-i_{l-1}} \times W_2,$$ where $W_1=\begin{cases} W(D_{i_1}), & \text{ if } 0 \notin \tilde K \\ S_{i_1}, & \text{ if } 0 \in \tilde K\end{cases}$ and  $W_2=\begin{cases} W(B_{n-i_l}), & \text{ if } n \notin \tilde K \\ S_{n-i_l}, & \text{ if } n \in \tilde K.\end{cases}$

{\it Case }(i): $l=0$. 

In this case $\tilde K \subset \{0, n\}$. 

If $\tilde K=\{0\}$, then $K$ is maximal special and there is a unique extreme element. 

If $\tilde K=\{n\}$, then $p(W_K)$ is of type $D_n$ and $p(W_K) \backslash W_0$ has cardinality $2$. Thus $\l$ is the only extreme element if and only if $p(W_K) W_\l=W_0$, i.e., $W_\l \nsubseteq p(W_K)$. This happens exactly when $\l=\o^\vee_r$ with $r<n$.

If $\tilde K=\{0, n\}$, then $W_K \cong S_n$. In this case, the number of extreme elements equals  $2^r$, where $\l=\o^\vee_r$. Thus the CCP condition is satisfied exactly when $r=1$.  

{\it Case} (ii): $l\ge 1$. 

{\it Case} (ii)(a): $l \ge 1$ and $r<n$. 

By Proposition \ref{par-count}, the number of extreme elements with nonnegative entries is at least $l+1$. 

Note that for any $m$ with $2 \le m \le l$, if $\l'=(c'_1, \ldots, c'_n)$ is an extreme element, then $\l''=(c''_1, \ldots, c''_n)$ is another extreme element, where $c''_k=-c'_{i_m+i_{m-1}+1-k}$ for any $k$ with $i_{m-1}+1 \le k \le i_m$, and $c''_k=c'_k$ for all other $k$'s.

In particular, there exists an extreme element with some negative entries among the $\{i_{m-1}+1, \ldots, i_m\}$-th entries, and with all the other entries nonnegative. If $0 \in \tilde K$ (resp., $n \in \tilde K$), then there exists an extreme element with some negative entries among the $\{1, \ldots, i_1\}$-th entries (resp., $\{i_l+1, \ldots, n\}$-th entries), and with all the other entries nonnegative. In this way, we construct $l-1+\e$ extreme elements. 

Therefore the number of extreme elements is at least $l+1+l-1+\e>l+\e$. The CCP condition does not hold in this case. 

{\it Case} (ii)(b): $l \ge 1$ and $r=n$. 

Note that if $\e=2$, then $\{0, n\} \subset \tilde K$ and there are $2^n \ge n>l$ extreme elements and the CCP condition does not hold in this case. 

Now assume that $\e \le 1$. It is easy to see there are at least $2^l$ extreme elements, whose entries are $\pm 1$, and there are at most one $-1$ entry in the $\{i_{m-1}+1, \ldots, i_m\}$-th entries for $1 \le m \le l$ (if $0 \in \tilde K$), or $2 \le m \le l+1$ (if $n \in \tilde K$). Here we set $i_0=0$ and $i_{l+1}=n$. Thus if the CCP condition is satisfied, then $2^l \le l+\e \le l+1$. Therefore $l=1$ and $\e=1$. Hence $\tilde K=\{0, i\}$ or $\tilde K=\{i, n\}$ for some $1 \le i \le n-1$. 

For $\tilde K=\{0, i\}$, there are $2^i$ extreme elements. Thus the CCP condition is satisfied if and only if $i=1$. For $\tilde K=\{i, n\}$, there are $2^{n-i+1} \ge 4$ extreme elements and the CCP condition does not hold in this case. 

\subsection{Type $\tilde C_n$} By applying a suitable automorphism, we may assume that if $n \in \tilde K$, then $0 \in \tilde K$. We have $\{i_1, \ldots, i_l\} \subset \tilde K \subset \{0, i_1, \ldots, i_l, n\}$, where $1 \le i_1<\cdots<i_l \le n-1$. Then $$W_K \cong W_1 \times S_{i_2-i_1} \times \cdots \times S_{i_l-i_{l-1}} \times W_2,$$ where $W_1=\begin{cases} W(B_{i_1}), & \text{ if } 0 \notin \tilde K \\ S_{i_1}, & \text{ if } 0 \in \tilde K\end{cases}$ and  $W_2=\begin{cases} W(B_{n-i_l}), & \text{ if } n \notin \tilde K \\ S_{n-i_l}, & \text{ if } n \in \tilde K.\end{cases}$

{\it Case} (i): $l=0$. 

If $\tilde K=\{0\}$, then $K$ is maximal special and the number of extreme elements is $1$. 

If $\tilde K=\{0, n\}$, then $W_K \cong S_n$. In this case, the number of extreme elements equals to $2^r$, where $\l=\o^\vee_r$. Thus the CCP condition is satisfied exactly when $r=1$.  

{\it Case} (ii): $l \ge 1$. 

Let $\epsilon=\#(\{0, n\} \cap \tilde K)$. By the same argument as in Subsection \ref{type-B}, if the CCP condition is satisfied, then we must have $\l=\o^\vee_n$ or $2 \o^\vee_n$, and $\e\le 1$. 

{\it Case} (ii)(a): $l \ge 1$ and $\e=0$. 

Similarly to the argument in  Subsection \ref{type-B}, there are at least $2^{l-1}$ extreme elements. If the CCP condition is satisfied, then $2^{l-1} \le l$ and hence $l \le 2$. If $l=1$, then $\tilde K=\{i\}$ for some $1 \le i \le n-1$ and there is only one extreme element, i.e., the element $\l$. If $l=2$, then $\tilde K=\{i_1, i_2\}$ for some $1 \le i_1<i_2 \le n$. The number of extreme elements is $2^{i_2-i_1}$. In this case, the CCP condition is satisfied if and only if $i_2=i_1+1$. 

{\it Case} (ii)(b): $l \ge 1$ and $\e=1$. 

By our assumption, $0 \in \tilde K$. Similarly to the argument in  Subsection \ref{type-B}, there are at least $2^l$ extreme elements. If the CCP condition is satisfied, then $2^l \le l+1$ and hence $l=1$. Thus $\tilde K=\{0, i\}$ for some $1 \le i \le n-1$. In this case, there are $2^i$ extreme elements and the CCP condition is satisfied exactly when $i=1$. 

\subsection{Type $\tilde D_n$} By applying a suitable automorphism, we may assume that if $1 \in \tilde K$, then $0 \in \tilde K$, and  if $n-1 \in \tilde K$, then $n \in \tilde K$, and if $n \in \tilde K$, then $0 \in \tilde K$. 

We have $\{i_1, \ldots, i_l\} \subset \tilde K \subset \{0, i_1, \ldots, i_l, n\}$, where $1 \le i_1<\cdots<i_l \le n-1$. Then $$W_K \cong W_1 \times S_{i_2-i_1} \times \cdots \times S_{i_l-i_{l-1}} \times W_2,$$ where $W_1=\begin{cases} W(D_{i_1}), & \text{ if } 0 \notin \tilde K \\ S_{i_1}, & \text{ if } 0 \in \tilde K\end{cases}$ and  $W_2=\begin{cases} W(D_{n-i_l}), & \text{ if } n \notin \tilde K \\ S_{n-i_l}, & \text{ if } n \in \tilde K.\end{cases}$

{\it Case} (i): $l=0$. 

If $\tilde K=\{0\}$, then $K$ is maximal special and the number of extreme element is $1$. 

If $\tilde K=\{0, n\}$, then $W_K \cong S_n$. In this case, the number of extreme elements equals to $2^r$. Thus the CCP condition is satisfied exactly when $r=1$. 

{\it Case} (ii): $l \ge 1$. 

Similarly to the argument in Subsection \ref{type-B}, if the CCP condition is satisfied, then $\l=\o^\vee_n$ or $\o^\vee_{n-1}$. 

Let $\epsilon=\#(\{0, n\} \cap \tilde K)$. If $\e=2$, then there are $2^{n-1}$ extreme elements. Since $n \ge 4$, we have $2^{n-1} \ge n>l$. Thus the CCP condition does not hold. 

If $\e \le 1$, then similarly to the argument in  Subsection \ref{type-B}, there are at least $2^l$ extreme elements. If the CCP condition is satisfied, then $2^l \le l+\e$. Hence $l=\e=1$. Thus $\tilde K=\{0, i\}$ for some $1 \le i \le n-1$. In this case, the number of extreme elements is $2^i$. Thus in this case, the CCP condition is satisfied if and only if $\tilde K=\{0, 1\}$. Note that $\o^\vee_n$ and $\o^\vee_{n-1}$ are permuted by an outer automorphism of the finite Dynkin diagram of $D_n$, which preserves $\{0, 1\}$.
\medskip

\subsection{Exceptional types} For exceptional types, we argue in a different way. 

Suppose that the extreme elements are $\l_1=\l, \l_2, \ldots, \l_k$. Then we have $W_0 \cdot \l=\sqcup_{i=1}^k W_K \cdot \l_i$. We denote by $W_\l \subset W_0$ the isotropy group of $\l$ and $W_{K, \l_i} \subset W_K$ the isotropy group of $\l_i$. Then we have 
\begin{equation}\label{}
\# W_0/W_\l=\sum_{i=1}^k \# W_K/W_{K, \l_i}.
\end{equation}

The trick here is that in most cases, we do not need to compute explicitly the coweights $\l_i$. Instead, we list the possible cardinalities $\# W_K/W_{K, \l_i}$. We then check that, in most cases, $\# W_0/W_\l$ does not equal to the sum of at most $\#\tilde K$ such numbers. Thus the CCP condition is not satisfied in these cases. 

\subsection{Type $\tilde G_2$} Note that $\l=\o^\vee_2$ and $\# W_0/W_\l=6$. Suppose that the CCP condition is satisfied. If $\#\tilde K=1$, then $\# W_K \ge 6$. This implies that $\tilde K=\{0\}$ or $\tilde K=\{1\}$. One may check directly that these two cases satisfy the CCP condition. If $\#\tilde K=2$, then $\# W_K \ge 3$, which is impossible. 

\subsection{Type $\tilde F_4$} Note that $\l=\o^\vee_1$. We refer to \cite[Plate VIII]{Bou} for the explicit description of the root system of type $F_4$. In particular, we have $\o^\vee_1=\epsilon^\vee_1+\epsilon^\vee_2$. Below is the list of maximal parahoric subgroups $K$ and $\# W_K/\# W_{K, \l}$.  
\smallskip
\smallskip

\begin{footnotesize}
	\[
	\begin{tabular}{|c|c|}
	\hline 
	$\tilde K$ & $\# W_K/W_{K, \l}$ \\
	\hline
	$\{0\}$ & $24$ \\
	\hline 
	$\{1\}$ & $2$ \\
	\hline
	$\{2\}$ & $6$ \\
	\hline
	$\{3\}$ & $12$ \\
	\hline
	$\{4\}$ & $24$ \\
	\hline
	\end{tabular}
	\]
\end{footnotesize}
\smallskip
\smallskip

Thus the CCP condition is satisfied if $\tilde K=\{0\}$ or $\{4\}$. 
\smallskip

a) Now suppose that $\#\tilde K=2$ and the CCP condition is satisfied. Then the (linear) action $W_K$ on $W_0 \cdot \l$ has exactly two orbits: the orbit of $\l$ and the orbit of another element $\l'$ with $\l' \in W_0 \cdot \l$. Then $\# W_K/W_{K, \l}+\# W_K/W_{K, \l'}=\# W_0/W_\l=24$. In particular, $24-\# W_K/W_{K, \l}$ divides $\# W_K$. This condition fails if $1 \in \tilde K$ or $2 \in \tilde K$, since in both cases $\# W_K/W_{K, \l} \le 6$. Thus $\tilde K$ must be $\{3, 4\}$, or $\{0, 4\}$, or $\{0, 3\}$. 

If $\tilde K=\{3, 4\}$, then we take $\l'=\epsilon^\vee_1-\epsilon^\vee_2$. By direct computation $\# W_K/W_{K, \l}+\# W_K/W_{K, \l'}=12+6 \neq 24$. 

If $\tilde K=\{0, 4\}$, then we take $\l'=-\epsilon^\vee_1+\epsilon^\vee_2$. By direct computation $\# W_K/W_{K, \l}+\# W_K/W_{K, \l'}=6+6 \neq 24$.

If $\tilde K=\{0, 3\}$, then $\# W_K/W_{K, \l}=3$ and $24-3=21$ does not divide $\# W_K$.
\smallskip

b) Now suppose that $\#\tilde K=3$. If the CCP condition is satisfied, then $\# W_K \ge 24/3=8$ and thus $\tilde K=\{0, 1, 4\}$. However, in this case $W_K=W_{K, \l}$ and thus for any $\l', \l''$, we have $\# W_K/W_{K, \l}+\# W_K/W_{K, \l'}+\# W_K/W_{K, \l''}<24$. That is a contradiction. 
\smallskip

c) If $\#\tilde K \ge 4$, then $\#\tilde K \cdot\# W_K \le 5 \cdot 2=10<24$. So the CCP condition is not satisfied in this case. 

\subsection{Type $\tilde E_6$} By the definition of a minuscule coweight, $\l$ is conjugate by an element of $W_K$ to the trivial coweight or a minuscule coweight $\l'$ of $W_K$. 

In the table below, we list all the numbers $\# W_K/W_{K, \l'}$, where $\l'$ is either trivial or a minuscule coweight of $W_K$. A direct product of Coxeter groups of type $A$ like $A_{n_1} \times \cdots \times A_{n_k}$ is abbreviated as $A_{n_1, \ldots, n_k}$. We use a double line to separate the subsets $\tilde K$ with different cardinalities. 

One may check case-by-case that if $27$ equals the sum of $\#\tilde K$ elements in the list of $\# W_K/W_{K, \l'}$ from the table below, then $W_K$ is of type $E_6$. 

\begin{footnotesize}
	\[
	\begin{tabular}{|c|c|c|c|c|c|}
	\hline 
	Type of $W_K$ & $\# W_K/W_{K, \l'}$ & Type of $W_K$ & $\# W_K/W_{K, \l'}$ & Type of $W_K$ & $\# W_K/W_{K, \l'}$ \\
	\hhline{|=|=|=|=|=|=|}
	$E_6$ & $27, 1$  & $A_{5, 1}$ & $20, 15, 6, 2, 1$ &
	$A_{2, 2, 2}$ & $3, 1$ \\
	\hhline{|=|=|=|=|=|=|}
	$D_5$ & $16, 10, 1$ &
	$A_5$ & $20, 15, 6, 1$ &
	$A_{4, 1}$ & $10, 5, 2, 1$ \\
	\hline 
	$A_{3, 1, 1}$ & $6, 4, 2, 1$ &
	$A_{2, 2, 1}$ & $3, 2, 1$ & & \\
	\hhline{|=|=|=|=|=|=|}
	$D_4$ & $8, 1$ & $A_4$ & $10, 5, 1$ & $A_{3, 1}$ & $6, 4, 2, 1$ \\
	\hline 
	$A_{2, 2}$ & $3, 1$ & $A_{2, 1, 1}$ & $3, 2, 1$ & & \\
	\hhline{|=|=|=|=|=|=|}
	$A_3$ & $6, 4, 1$ & $A_{2, 1}$ & $3, 2, 1$ & $A_{1, 1, 1}$ & $2, 1$ \\
	\hhline{|=|=|=|=|=|=|}
	$A_2$ & $3, 1$ & $A_{1, 1}$ & $2, 1$ & & \\
	\hhline{|=|=|=|=|=|=|}
	$A_1$ & $2, 1$ & & & & \\
	\hline
	\end{tabular}
	\]
\end{footnotesize}

\subsection{Type $\tilde E_7$} In the table below, we list all the numbers $\# W_K/W_{K, \l'}$, where $\l'$ is either trivial or a minuscule coweight of $W_K$. 

One may check case-by-case that, if $56$ equals the sum of $\#\tilde K$ elements in the list of $\# W_K/W_{K, \l'}$ from the table below, then $W_K$ is of type $E_7$, $A_7$, $E_6$ or $A_6$.

If $W_K$ is of type $A_7$, then $\# W_K/W_{K, \l}=8 \neq 56$. 

If $W_K$ is of type $E_6$, then $\# W_K/W_{K, \l}=1$ and thus there is no $\l'$ with 
\[\# W_K/W_{K, \l}+\# W_K/W_{K, \l'}=56.
\] 

If $W_K$ is of type $A_6$, then $\# W_K/W_{K, \l}=7$ or $1$. Thus there is no $\l'$ with 
\[
\# W_K/W_{K, \l}+\# W_K/W_{K, \l'}=56.
\]

\begin{footnotesize}
	\[
	\begin{tabular}{|c|c|c|c|c|c|}
	\hline
	Type of $W_K$ & $\# W_K/W_{K, \l'}$ & Type of $W_K$ & $\# W_K/W_{K, \l'}$ & Type of $W_K$ & $\# W_K/W_{K, \l'}$ \\
	\hhline{|=|=|=|=|=|=|}
	$E_7$ & $56, 1$ & $A_7$ & $70, 56, 28, 8, 1$ & $D_6 \times A_1$ & $32, 12, 2, 1$ \\ 
	\hline
	$A_{5, 2}$ & $20, 15, 6, 3, 1$ & $A_{3, 3, 1}$ & $6, 4, 2, 1$ & & \\
	\hhline{|=|=|=|=|=|=|}
	$E_6$ & $27, 1$ & $D_6$ & $32, 12, 1$ & $A_6$ & $35, 21, 7, 1$ \\
	\hline
	$D_5 \times A_1$ & $16, 10, 2, 1$ & $A_{5, 1}$ & $20, 15, 6, 2, 1$ & $D_4 \times A_1 \times A_1$ & $8, 2, 1$ \\
	\hline
	$A_{4, 2}$ & $10, 5, 3, 1$ & $A_{3, 3}$ & $6, 4, 1$ & $A_{3, 2, 1}$ & $6, 4, 3, 2, 1$ \\
	\hline
	$A_{3, 1, 1, 1}$ & $6, 4, 2, 1$ & $A_{2, 2, 2}$ & $3, 1$ & & \\
	\hhline{|=|=|=|=|=|=|}
	$D_5$ & $16, 10, 1$ & $A_5$ & $20, 15, 6, 1$ & $D_4 \times A_1$ & $8, 2, 1$ \\
	\hline
	$A_{4, 1}$ & $10, 5, 2, 1$ & $A_{3, 2}$ & $6, 4, 3, 1$ & $A_{3, 1, 1}$ & $6, 4, 2, 1$ \\
	\hline
	$A_{2, 2, 1}$ & $3, 2, 1$ & $A_{2, 1, 1, 1}$ & $3, 2, 1$ & $A_{1, 1, 1, 1, 1}$ & $2, 1$ \\
	\hhline{|=|=|=|=|=|=|}
	$D_4$ & $8, 1$ & $A_4$ & $10, 5, 1$ & $A_{3, 1}$ & $6, 4, 2, 1$ \\
	\hline 
	$A_{2, 2}$ & $3, 1$ & $A_{2, 1, 1}$ & $3, 2, 1$ & $A_{1, 1, 1, 1}$ & $2, 1$ \\
	\hhline{|=|=|=|=|=|=|}
	$A_3$ & $6, 4, 1$ & $A_{2, 1}$ & $3, 2, 1$ & $A_{1, 1, 1}$ & $2, 1$ \\
	\hhline{|=|=|=|=|=|=|}
	$A_2$ & $3, 1$ & $A_{1, 1}$ & $2, 1$ & & \\
	\hhline{|=|=|=|=|=|=|}
	$A_1$ & $2, 1$ & & & & \\
	\hline
	\end{tabular}
	\]
\end{footnotesize}

\section{Rationally strictly pseudo semi-stable reduction}\label{s:ratpsss}

In this section, we exclude the cases from the list in Theorem \ref{CCP} that do not have rationally SPSS reduction. By Lemma \ref{rpssPoinc}, we check if the Poincar\'e polynomial is symmetric. As we have seen, the Poincar\'e polynomial depends only on the enhanced Coxeter datum, not the enhanced Tits datum. We start the elimination process with the exceptional types. 

\subsection{The case $(\tilde G_2, \o^\vee_2, \{1\})$}\label{ss:G2} 

Here $t^\l=s_0 s_2 s_1 s_2 s_1 s_2$ and the unique maximal element in $W_{\le \l, K}$ is $w_{\l, K}=s_2 s_0 s_2 s_1 s_2 s_1$. The set $W_{\le \l, K}$ has a unique element of length $1$, which is $s_1$, but has two elements of length $5$, which are $s_2 w_{\l, K}$ and $s_0 w_{\l, K}$. Therefore the Poincar\'e polynomial is not symmetric and $\overline{K^\flat \l  K^\flat / K^\flat}$ is not rationally smooth. 

\subsection{The case $(\tilde F_4, \o^\vee_1, \{4\})$}\label{ss:F4} 
Here $t^\l=s_0 s_1 s_2 s_3 s_4 s_2 s_3 s_1 s_2 s_3 s_4 s_1 s_2 s_3 s_2 s_1$ and the unique maximal element in $W_{\le \l, K}$ is $w_{\l, K}=s_1 s_2 s_3 s_2 s_1 s_0 s_1 s_2 s_3 s_4 s_2 s_3 s_1 s_2 s_3 s_4$. The set $W_{\le \l, K}$ has a unique element of length $2$, which is $s_3 s_4$, but has at least two elements of colength $2$, which are $s_2 s_1 w_{\l, K}$ and $s_0 s_1 w_{\l, K}$. Therefore the Poincar\'e polynomial is not symmetric and $\overline{ K^\flat \l  K^\flat/ K^\flat}$ is not rationally smooth.

\subsection{The case $(\tilde F_4, \o^\vee_1, \{0\})$ (special parahoric)}\label{ss:F4sp}
Here the unique maximal element in $W_{\le \l, K}$ is $w_{\l, K}=t^{-\l}=w_0^{\{0, 1\}} w_0^{\{n\}} s_0$, where $w_0^{K'}$ is the longest element in $W_{K'}$ for $K'=\{0\}$ or $\{0, 1\}$. Thus 
\[
P_{\le \l, K}=1+q \frac{\sum_{x \in W_{\{0\}}} q^{\ell(x)}}{\sum_{x \in W_{\{0, 1\}}} q^{\ell(x)}}.
\] Note that $(\sum_{x \in W_{\{0\}}} q^{\ell(x)})(\sum_{x \in W_{\{0, 1\}}} q^{\ell(x)})^{-1}$ is a symmetric polynomial and not all the coefficients are equal to $1$. Thus $P_{\le \l, K}$ is not symmetric and $\overline{ K^\flat \l  K^\flat/ K^\flat}$ is not rationally smooth.

\subsection{The classical types} Now we consider the cases where $\tW$ is of classical type. Let $\Phi_{\rm af}$ be the affine root system of a split group whose associated affine Weyl group is $W_a$. Let $\Phi$ be the set of finite roots. As $\Phi_{\rm af}$ comes from a split group, $\Phi _{\rm af}=\{a+n \delta \, \mid\, a \in \Phi, n \in \BZ\}$. The positive affine roots are $\{a+n\delta \, \mid\, a \in \Phi^+, n \in \BZ_{>0}\} \sqcup \{-a+n \delta \, \mid\, a \in \Phi^+, n \in \BZ_{\ge 0}\}$. In other words, the element $t^\lambda \in \tW$ acts on the apartment by the translation $-\lambda$.

We have the following formula on the length function of an element in $\tW$ (see \cite{IM}). 

\begin{lemma}\label{length}
	For $w \in W_0$ and $\a \in \Phi$, set $$\delta_w(\a)=\begin{cases} 0, & \text{ if } w \a \in \Phi^+; \\ 1, & \text{ if } w \a \in \Phi^-. \end{cases}$$ Then for any $x, y \in W_a$ and any translation element $t^{\l'}$ in $\tilde W$, 
	 $$\ell(x t^{\l'} y^{-1})=\sum_{\a \in \Phi^+} \vert<\l', \a>+\delta_x(\a)-\delta_y(\a)\vert.$$
\end{lemma}

We also have the following well-known facts on the Bruhat order in $\tW$. 

\begin{lemma}\label{Bruhat-pm}
	Let $w \in \tW$. If $\a \in \Phi_{\rm af}$ is positive, and $w^{-1}(\a)$ is positive (resp., negative), then $s_{\a} w>w$ (resp., $s_{\a} w<w$).
\end{lemma}

\begin{remark}
In particular, if $K=\{0\}$, then $t^{-\l} \in \tW^K$ for any dominant $\l$. 
\end{remark}

\begin{lemma}\label{sw}
	Let $w \in \tW^K$ and $s \in \tilde S$. If $s w<w$, then $\ell(s w)=\ell(w)-1$ and $s w \in \tW^K$. 
\end{lemma} 

The following result on the maximal element $w_{\l, K}$ of $W_{\le \l, K}$ will also be useful in this section. 

\begin{lemma}\label{max}
The maximal element $w_{\l, K}$ in $W_{\le \l, K}$ is $t^{\l'}$, where $\l'$ is the unique element in $W_K\cdot \l$ such that $\<\l', \a\> \ge 0$ for any affine simple root $\a$  not in $\tilde K$. 
\end{lemma}

\begin{proof}
Let $w^K_0$ be the longest element in $W_K$. We claim that 

\smallskip

\noindent Claim: {\it  The element $t^{\l'} w^K_0$ is the maximal element in $W_K t^{\l'} W_K=W_K t^{\l} W_K$.} 

\smallskip
Let $\a$ be a simple affine root that is not in $\tilde K$. Then $(t^{\l'} w^K_0)^{-1} (\a)=(w^K_0)^{-1}(\a-\<\l', \a\> \delta)=(w^K_0)^{-1}(\a)-\<\l', \a\> \delta$ is a negative affine root. Hence by Lemma \ref{Bruhat-pm}, $s_\a t^{\l'} w^K_0<t^{\l'} w^K_0$. Similarly, $(t^{\l'} w^K_0)(\a)=t^{\l'} (w^K_0 (\a))=w^K_0(\a)+\<\l', w^K_0(\a)\> \delta$. Note that $w^K_0(\a)$ equals to $-\beta$ for some simple affine root $\beta$ that is not in $\tilde K$. Hence $(t^{\l'} w^K_0)(\a)$ is a negative affine root. By Lemma \ref{Bruhat-pm}, $t^{\l'} w^K_0 s_\a<t^{\l'} w^K_0$. The claim is proved.

Note that by the assumption on $\l'$, $t^{\l'}(\a)$ is a positive affine root for any simple affine root $\a$ that is not in $\tilde K$. Therefore $t^{\l'} \in \tW^K$. 
	
Since $t^{\l'}$ is the unique element contained in both $(t^{\l'} w^K_0) W_K$ and $\tW^K$, and $t^{\l'} w^K_0$ is the maximal element in $W_K t^{\l} W_K$, $t^{\l'}$ is the unique maximal element in $W_{\le \l, K}$. The statement is proved.  
\end{proof}

With these facts established, we can now continue our elimination process.

\subsection{The case $(\tilde B_n, \o^\vee_i, \{0\})$ (special parahoric)}\label{4.5}\label{ss:Bnfirstsp}
Here $n \ge 3$ and $2 \le i \le n-1$.

Note that the set $W_{\le \l, K}$ has a unique element of length $2$, which is $\tau s_2 s_0$.  Here $\tau$, as usual, is the unique length-zero element in $\tW$ with $t^{\l} \in W_a \tau$. By Lemma \ref{Bruhat-pm}, $s_i t^{-\l}<t^{-\l}$ and $s_{i-1} s_i t^{-\l}, s_{i+1} s_i t^{-\l} <s_i t^{-\l}$. By Lemma \ref{sw}, $s_i t^{-\l}$ is a colength-$1$ element in $W_{\le \l, K}$ and $s_{i-1} s_i t^{-\l}, s_{i+1} s_i t^{-\l}$ are colength-$2$ elements in $W_{\le \l, K}$. Hence the set $W_{\le \l, K}$ has at least two elements of colength $2$. Therefore the Poincar\'e polynomial is not symmetric and $\overline{K^\flat \l K^\flat/K^\flat}$ is not rationally smooth.

\subsection{The case $(\tilde C_n, \o^\vee_i, \{0\})$  (special parahoric)}\label{4.5}\label{ss:Cnfirstsp}
Here $n \ge 3$ and $2 \le i \le n-1$.

Note that the set $W_{\le \l, K}$ has a unique element of length $2$, which is $s_1 s_0$. By Lemma \ref{Bruhat-pm}, $s_i t^{-\l}<t^{-\l}$ and $s_{i-1} s_i t^{-\l}, s_{i+1} s_i t^{-\l} <s_i t^{-\l}$. By Lemma \ref{sw}, $s_i t^{-\l}$ is a colength-$1$ element in $W_{\le \l, K}$ and $s_{i-1} s_i t^{-\l}, s_{i+1} s_i t^{-\l}$ are colength-$2$ elements in $W_{\le \l, K}$. Hence the set $W_{\le \l, K}$ has at least two elements of colength $2$. Therefore the Poincar\'e polynomial is not symmetric and $\overline{K^\flat \l K^\flat/K^\flat}$ is not rationally smooth.

\subsection{The case $(\tilde C_n, l \o^\vee_n, \{i\})$}\label{4.5}\label{ss:Cnfirst}
 Here $n \ge 2$, $l=1$ or $2$ and $1 \le i \le n-1$.

By Lemma \ref{max}, $w_{\l, K}=t^{\l'}$, where the first $i$ entries of $\l'$ are $\frac{l}{2}$ and the last $n-i$ entries of $\l'$ are $-\frac{l}{2}$. Note that the set $W_{\le \l, K}$ has a unique element of length $1$, which is $\tau s_i$.  Here $\tau$, as usual,  is the unique length-zero element in $\tW$ with $t^{\l'} \in W_a \tau$. By Lemma \ref{Bruhat-pm}, $s_0 t^{\l'}<t^{\l'}$. By Lemma \ref{sw}, $s_0 t^{\l'}$ is a colength-$1$ element in $W_{\le \l, K}$. Similarly, $s_n t^{\l'}$ is a colength-$1$ element in $W_{\le \l, K}$. Thus the set $W_{\le \l, K}$ has at least two elements of colength $1$. Therefore the Poincar\'e polynomial is not symmetric and $\overline{K^\flat \l K^\flat/K^\flat}$ is not rationally smooth.

\subsection{The case $(\tilde B_n, \o^\vee_r, \{n\})$}\label{ss:Bnfirst}
 Here $n \ge 3$ and $1 \le r \le n-1$.

By Lemma \ref{max}, $w_{\l, K}=t^{\l'}$, where the first $n-r$ entries of $\l'$ are $0$ and the last $r$ entries of $\l'$ are $-1$. 

\subsubsection{The case $2 \le r \le n-2$} Note that the set $W_{\le \l, K}$ has a unique element of length $2$, which is $\tau s_{n-1} s_n$, where again $\tau$ is the unique length-zero element in $\tW$ with $t^{\l'} \in W_a \tau$. 

By Lemma \ref{Bruhat-pm}, $s_r t^{\l'}<t^{\l'}$ and $s_{r-1} s_r t^{\l'}, s_{r+1} s_r t^{\l'}<s_r t^{\l'}$. By Lemma \ref{sw}, $s_r t^{\l'}$ is a colength-$1$ element in $W_{\le \l, K}$ and $s_{r-1} s_r t^{\l'}, s_{r+1} s_r t^{\l'} \in W_{\le \l, K}$ are  colength-$2$ elements in $W_{\le \l, K}$. Hence the set $W_{\le \l, K}$ has at least two elements of colength $2$. Therefore the Poincar\'e polynomial is not symmetric and $\overline{K^\flat \l  K^\flat/ K^\flat}$ is not rationally smooth.

\subsubsection{The case $r=1$}\label{7.5.2} By direct computation, $w_{\l, K}=\tau (s_{n-1} s_{n-2} \cdots s_2) (s_0 s_1 \cdots s_n)$ and the Poincar\'e polynomial is $(1+q+\cdots+q^{2(n-1)})q+q^n+1$ which is not symmetric. Thus  $\overline{ K^\flat \l  K^\flat/ K^\flat}$ is not rationally smooth.

\subsubsection{The case $r=n-1$} The set $W_{\le \l, K}$ has a unique element of length $1$ which is $\tau s_n$, but has at least two elements of colength $1$, namely  $s_1 w_{\l, K}$ and $s_0 w_{\l, K}$. The Poincar\'e polynomial is not symmetric and $\overline{ K^\flat \l  K^\flat/ K^\flat}$ is not rationally smooth. 

\subsection{The case $(\tilde C_n, 2 \o^\vee_n, \{i, i+1\})$ with $0 \le i \le \frac{n}{2}-1$}\label{ss:Cnsecond}
Here $n \ge 2$.

By Lemma \ref{max}, $w_{\l, K}=t^{\l'}$, where the first $i+1$ entries of $\l'$ is $1$ and the last $n-i-1$ entries of $\l'$ is $-1$. In this case, $W_{\le \l, K}$ has exactly two elements of length $1$ which are $s_i$ and $s_{i+1}$. Similarly to the argument in \S\ref{4.5}, $s_n t^{\l'}$ and $s_0 t^{\l'}$ are colength $1$ elements in $W_{\le \l, K}$. By Lemma \ref{Bruhat-pm}, $s_{2 e_{i+1}+2 \delta} t^{\l'}=(i+1, -(i+1)) t^{\l''}<t^{\l'}$ and $s_{2 e_1+2 \delta} t^{\l'} \in \tW^K$. Here the first $i$-entries of $\l''$ is $1$ and the last $n-i$-entries of $\l''$ is $-1$. By Lemma \ref{length}, $\ell(s_{2 e_{i+1}+2 \delta} t^{\l'})=\ell(t^{\l'})-1$. Hence $s_{2 e_{i+1}+2 \delta} t^{\l'}$ is a colength-$1$ element in $W_{\le \l, K}$. Therefore $W_{\le \l, K}$ contains at least three elements of colength $1$ and the Poincar\'e polynomial of $W_{\le \l, K}$ is not symmetric. 

\subsection{The case $(\tilde C_n, \o^\vee_n, \{i, i+1\})$ with $1 \le i \le \frac{n}{2}-1$}
Here $n \ge 2$.

By Lemma \ref{max}, $w_{\l, K}=t^{\l'}$, where the first $i+1$ entries of $\l'$ is $\frac{1}{2}$ and the last $n-i-1$ entries of $\l'$ is $-\frac{1}{2}$. In this case, $W_{\le \l, K}$ has exactly two elements of length $1$ which are $\tau s_i$ and $\tau s_{i+1}$, where again $\tau$ is the unique length-zero element in $\tW$ with $t^{\l'} \in W_a \tau$. Similarly to the argument in \S\ref{4.5}, $s_n t^{\l'}$ and $s_0 t^{\l'}$ are colength $1$ elements in $W_{\le \l, K}$. By Lemma \ref{Bruhat-pm}, $s_{2 e_{i+1}+\delta} t^{\l'}=(i+1, -(i+1)) t^{\l''}<t^{\l'}$ and $s_{2 e_1+\delta} t^{\l'} \in \tW^K$. Here the first $i$-entries of $\l''$ is $\frac{1}{2}$ and the last $n-i$-entries of $\l''$ is $-\frac{1}{2}$. By Lemma \ref{length}, $\ell(s_{2 e_{i+1}+\delta} t^{\l'})=\ell(t^{\l'})-1$. Hence $s_{2 e_{i+1}+\delta} t^{\l'}$ is a colength-$1$ element in $W_{\le \l, K}$. Therefore $W_{\le \l, K}$ contains at least three elements of colength $1$ and the Poincar\'e polynomial of $W_{\le \l, K}$ is not symmetric. 

\subsection{The case $(\tilde B_n, \o^\vee_n, \{0, 1\})$}\label{ss:Bnsecond}
Here $n \ge 3$. 

By Lemma \ref{max}, $w_{\l, K}=t^{\l'}$, where $\l'=(1, -1, -1, \ldots, -1)$. In this case, $W_{\le \l, K}$ has exactly $3$ elements of length $2$ which are  $\tau s_0 s_1$, $\tau s_2 s_0$ and $\tau s_2 s_1$. Similarly to the argument in Subsection \ref{ss:Cnsecond}, $s_n t^{\l'}$ and $s_{\e_1+\delta} t^{\l'}=(1, -1) t^{-\o^\vee_n}$ are colength-$1$ elements in $W_{\le \l, K}$. 

By Lemma \ref{Bruhat-pm}, $s_{n-1} s_n t^{\l'}<s_n t^{\l'}$ and $s_n s_{\e_1+\delta} t^{\l'}<s_{\e_1+\delta} t^{\l'}$. By Lemma \ref{sw}, $s_{n-1} s_n t^{\l'}$ and $s_n s_{\e_1+\delta} t^{\l'}$ are colength-$2$ elements in $W_{\le \l, K}$. The next lemma produces two more elements of colength $2$. 

\begin{lemma}
The elements $s_{\e_1+\delta} t^{\l'} s_{\e_1-\e_n}$ and $s_{\e_1-\e_n+\delta} s_{\e_1+\delta} t^{\l'}$ are colength $2$ elements in $W_{\le \l, K}$.
\end{lemma}

\begin{proof}
We have $s_{\e_1+\delta} t^{\l'} \cdot (\e_1-\e_n)=(1, -1) \cdot (\e_1-e_n)=-\e_1-\e_n$. Thus $s_{\e_1+\delta} t^{\l'} s_{\e_1-\e_n}<s_{\e_1+\delta} t^{\l'}$. Note that $s_{\e_1+\delta} t^{\l'} s_{\e_1-\e_n}=(1, -1) (1, n) t^{-\o^\vee_n}$. 

For $2 \le i \le n-2$, we have $s_{\e_1+\delta} t^{\l'} s_{\e_1-\e_n} \cdot (\e_i-\e_{i+1})=(1, -1) (1, n) \cdot (\e_i-\e_{i+1})=\e_i-\e_{i+1}$. We have $s_{\e_1+\delta} t^{\l'} s_{\e_1-\e_n} \cdot (\e_{n-1}-\e_n)=(1, -1) (1, n) \cdot (\e_{n-1}-\e_n)=\e_1+\e_{n-1}$. We have $s_{\e_1+\delta} t^{\l'} s_{\e_1-\e_n} \cdot \e_n=(1, -1) (1, n) (\e_n-\delta)=-\e_1-\delta$. Therefore $s_{\e_1+\delta} t^{\l'} s_{\e_1-\e_n} \in \tW^K$. 

Finally, by Lemma \ref{length}, we have  $\ell(s_{\e_1+\delta} t^{\l'} s_{\e_1-\e_n})=\ell(t^{\l'})-2$. Therefore $s_{\e_1+\delta} t^{\l'} s_{\e_1-\e_n}$ is a colength $2$ element in $W_{\le \l, K}$. 

Similarly, we have $(s_{\e_1+\delta} t^{\l'})^{-1} \cdot (\e_1-\e_n+\delta)=-\e_1-\e_n-\delta$. Thus $s_{\e_1-\e_n+\delta} s_{\e_1+\delta} t^{\l'}<s_{\e_1+\delta} t^{\l'}$. Note that $s_{\e_1-\e_n+\delta} s_{\e_1+\delta} t^{\l'}=(1, n) (1, -1) t^{(0, -1, \ldots, -1, 0)}$. 

For $2 \le i \le n-2$, we have $s_{\e_1-\e_n+\delta} s_{\e_1+\delta} t^{\l'} \cdot (\e_i-\e_{i+1})=\e_i-\e_{i+1}$. We have $s_{\e_1-\e_n+\delta} s_{\e_1+\delta} t^{\l'} \cdot (\e_{n-1}-\e_n)=(1, n) (1, -1) \cdot (\e_{n-1}-\e_n-\delta)=\e_{n-1}-\e_1-\delta$. We have $s_{\e_1-\e_n+\delta} s_{\e_1+\delta} t^{\l'} \cdot \e_n=\e_1$. Therefore $s_{\e_1+\delta} t^{\l'} s_{\e_1-\e_n} \in \tW^K$. 

Finally, by Lemma \ref{length}, we have  $\ell(s_{\e_1-\e_n+\delta} s_{\e_1+\delta} t^{\l'})=\ell(t^{\l'})-2$. Therefore $s_{\e_1+\delta} t^{\l'} s_{\e_1-\e_n}$ is a colength $2$ element in $W_{\le \l, K}$. 
\end{proof}

Therefore $W_{\le \l, K}$ contains at least $4$ elements of colength $2$ and the Poincar\'e polynomial of $W_{\le \l, K}$ is not symmetric. 

\subsection{The case $(\tilde C_n, \omega^\vee_1, \{0, n\})$} This case is more complicated than the cases we have discussed earlier. In fact, the geometric special fiber has two irreducible components and the Poincar\'e polynomials of the irreducible components are both symmetric. However, the Poincar\'e polynomial of their intersection is not symmetric.

Let $\l=(1, 0, \ldots, 0)$ and $\l'=(0, 0, \ldots, 0, -1)$. The irreducible components of the geometric special fibers are $\overline{ K^\flat \l  K^\flat/ K^\flat}$ and $\overline{ K^\flat \l' K^\flat/ K^\flat}$.

We set $w_1=\max(W_K t^\l W_K)$ and $w_2=\max(W_K t^{\l'} W_K)$. By direct computation, $$w_1=(s_{n-1} s_{n-2} \cdots s_0) (s_1 s_2 \cdots s_n) w^K_0, \qquad w_2=(s_1 s_2 \cdots s_n) (s_{n-1} s_{n-2} \cdots s_0) w^K_0,$$ where $w^K_0$ is the longest element in $W_K$. Moreover, the set $\{w' \in \tW; w' \le w_1, w' \le w_2\}$ contains a unique maximal element $w w^K_0$, where $$w=(s_1 s_2 \cdots s_{n-1}) (s_{n-2} s_{n-3} \cdots s_1) s_0 s_n \in \tW^K.$$ Set $W_{\le w, K}=\{v \in \tW^K; v \le w\}$. The intersection of $\overline{ K^\flat \l K^\flat/ K^\flat}$ and $\overline{ K^\flat \l'  K^\flat/ K^\flat}$ is $\overline{ K^\flat w  K^\flat/ K^\flat}$ and the associated Poincar\'e polynomial is $$P_{\le w, K}(q)=\sum_{v \in W_{\le w, K}} q^{\ell(v)}.$$

We have $$W_{\le w, K}=\{1, v_1 s_0, v_2 s_n, v_3 s_0 s_n\},$$ where \begin{gather*} v_1 \in W_{\{0, n\}} \cap W^{\{0, 1, n\}}=\{1, s_1, s_2 s_1, \ldots, s_{n-1} s_{n-2} \cdots s_1\}, \\ v_2 \in W_{\{0, n\}} \cap W^{\{0, n-1, n\}}=\{1, s_{n-1}, s_{n-2} s_{n-1}, \ldots, s_1 s_2 \cdots s_{n-1}\}, \\ v_3 \in W_{\{0, n\}} \cap W^{\{0, 1, n-1, n\}}.\end{gather*} Note that $W_{\{0, n\}} \cap W^{\{0, 1, n-1, n\}}=W(A_{n-1})^{\{1, n-1\}}$, where $W(A_{n-1})$ is the finite Weyl group of type $A_{n-1}$. Thus $$\sum_{v_3 \in W_{\{0, n\}} \cap W^{\{0, 1, n-1, n\}}} q^{\ell(v_3)}=\frac{\sum_{v \in W(A_{n-1})} q^{\ell(v)}}{\sum_{v \in W(A_{n-3})} q^{\ell(v)}}=(1+q+\cdots+q^{n-2})(1+q+\cdots+q^{n-1}).$$ 

Hence \begin{align*}P_{\le w, K} &=1+2(1+q+\cdots+q^{n-1})q+(1+q+\cdots+q^{n-2})(1+q+\cdots+q^{n-1})q^2 \\ &=(1+2q+\cdots+n q^{n-1}+(n+1) q^n)+((n-1) q^{n+1}+(n-2) q^{n+2}+\cdots+q^{2n-1}).
\end{align*}

Note that the coefficient of $q^{n-1}$ is $n$ and the coefficient of $q^n$ is $n+1$. Therefore $P_{\le w, K}$ is not symmetric. 

\subsection{The case $(\tilde A_1, 2 \omega^\vee_1, \{0, 1\})$}\label{7.9} Finally, we consider the case $(\tilde A_1, 2 \omega^\vee_1, \{0, 1\})$. In this case, $ K^\flat=I^\flat$ and $\Adm(\l)=\{s_0 s_1, s_1 s_0, s_1, s_0, 1\}$. The irreducible components of the geometric special fiber are $\overline{ I^\flat s_0 s_1  I^\flat/ I^\flat}$ and $\overline{ I^\flat s_1 s_0  I^\flat/ I^\flat}$. Their intersection is $\overline{ I^\flat s_1  I^\flat/ I^\flat} \cup \overline{ I^\flat s_0  I^\flat/ I^\flat}$ and thus is not irreducible. 

\section{Proof of Theorem \ref{maingoodred}}\label{s:maingood}

In this section we assume $p\neq 2$. As already explained in Subsection \ref{ss:goodred}, we may assume that $G_\ad$ is absolutely simple. By Proposition \ref{proofLM} (iv), we may change $G$ arbitrarily, as long as $G_\ad$ is fixed. Let us check one implication.   If $K$ is hyperspecial, then  $\BM^\loc_K(G, \{\mu\})$ is smooth over $O_E$, cf. Proposition \ref{proofLM}, (i).  Let $G=\GU(V)$ be the group of unitary similitudes  for a hermitian space relative to a ramified quadratic extension $\tilde F/F$, and let $\{\mu\}=(1, 0, \ldots,0)$. If  $K$ is the stabilizer of a $\pi$-modular lattice $\Lambda$ if $\dim V$ is even, resp., is the stabilizer of an almost $\pi$-modular lattice $\Lambda$ if $\dim V$ is odd, then $E=\tilde F$ and  the local model $\BM^\loc_K(G, \{\mu\})$ is smooth over $O_{E}$, cf. \cite[Prop. 4.16]{Arz}, \cite[Thm. 2. 27, (iii)]{PRS}. Now let $G={\rm SO}(V)$ be the orthogonal group associated to an orthogonal $F$-vector space of even dimension $\geq 6$, $\{\mu\}$ the cocharacter corresponding to the orthogonal Grassmannian
of isotropic subspaces of maximal dimension,
and $K$ the parahoric stabilizer of an almost selfdual vertex lattice, as in \ref{ss:goodred} (2).
After an unramified extension of $F$, the set-up described in
\ref{exoticOrtho} applies; by the calculation in \ref{exoticOrtho} and Proposition \ref{proofLM} (ii) 
 the local model $\BM^\loc_K(G, \{\mu\})$ is smooth over $O_{E}$.
The general case of exotic good reduction type (which involves in addition an unramified restriction of scalars) follows.

Conversely, assume that $\Mloc$ is smooth over $O_E$. Then its geometric special fiber is irreducible, and hence the triple $(G, \{\mu\}, K)$ produces enhanced Coxeter data that appear in Theorem \ref{CCP} under the heading (1). By Subsections \ref{ss:G2} and \ref{ss:F4},  the exceptional types (d) and (e) do not have rationally SPSS reduction. Similarly, Subsection \ref{ss:Bnfirst} eliminates the case (b),  and Subsection \ref{ss:Cnfirst} eliminates the case (c). Therefore, the only remaining possibilities are in case (a), i.e., $\breve K$ is a special maximal parahoric. Hence the associated enhanced Tits datum $(\tilde \Delta, \{\lambda\}, \tilde K)$ is such that $\tilde K$ consists of a single \emph{special} vertex. From these cases, Subsection \ref{ss:F4sp} eliminates $(\tilde F_4, \omega_1^\vee, \{0\})$. 
Subsections \ref{ss:Bnfirstsp} and \ref{ss:Cnfirstsp}  eliminate  $(\tilde B_n, \omega^\vee_i, \{0\})$ and $(\tilde C_n, \omega^\vee_i, \{0\})$, for $n\geq 3$ and $2\leq i\leq n-1$. 
(In these cases, the special fiber is irreducible but, again, not rationally smooth).
When $\breve K$ is hyperspecial and $\lambda$ is minuscule we have good reduction.
When $\breve K$ is hyperspecial and $\lambda$ is not minuscule the reduction is not smooth by \cite{MOV}.
(This reference is for $k$ replaced by $\mathbb C$ but the same argument works; see also \cite[\S 6]{HR} for an explanation of the passage from $\BC$ to $k$.)
It remains to list the remaining cases in which $\breve K$ is special but not hyperspecial. Here are these remaining cases:
\medskip

1) $(\tilde B_n, \omega^\vee_1, \{0\})$, $n\geq 3$. 
\smallskip

Since we are only considering the cases in which $\{0\}$ is not hyperspecial, the local Dynkin diagram is $B$-$C_n$.  Since the non-trivial automorphism of $B$-$C_n$ 
does not preserve $\{0\}$, the Frobenius has to act trivially (see \cite{Gross}). The corresponding group is a quasi-split (tamely)
ramified unitary group $U(V)$ for $V$ of even dimension $2n$ (e.g. \cite[p. 22]{Gross}). The coweight  corresponds to $\{\mu\}=(1, 0, \ldots,0)$ and $K$ is the parahoric stabilizer of a $\tilde \pi$-modular lattice
(notations as in  \ref{exotgood} (1).) This is a case of unitary exotic good reduction.
\medskip

2) $(\tilde B_n, \omega^\vee_n, \{0\})$, $n\geq 3$. 
\smallskip

As above, the local Dynkin diagram is $B$-$C_n$ and the corresponding group a (tamely)
ramified unitary group $U(V)$ for $V$ of even dimension $2n$. The subgroup $K$ is the parahoric stabilizer of a $\tilde \pi$-modular lattice. In this case, the coweight corresponds to 
$\{\mu\}=(1^{(n)}, 0^{(n)})$ so this is the case of signature $(n, n)$. By \cite[(5.3)]{PR}, we see that the geometric special fiber of the local model contains an open affine subscheme which has the following properties: It is the reduced locus $C^{\rm red}$
 of an irreducible affine cone $C$ which is defined by homogeneous equations of degree $\geq 2$ and which is generically reduced. Then $C^{\rm red}$ is the affine cone over the integral projective variety $(C-\{0\})^{\rm red}/\sim$, also given by such equations, and is therefore not smooth. We see that, in this case, $\Mloc$ is not smooth.
\medskip

3) $(\tilde C_n, \omega^\vee_1, \{0\})$, $n\geq 2$. 
\smallskip

Since we are only considering the case in which $\{0\}$ is not hyperspecial, the local Dynkin diagram is $C$-$BC_n$ or $C$-$B_n$. In both cases, only the trivial automorphism can preserve $\{0\}$ so Frobenius acts trivially.
In the case $C$-$BC_n$, we have a ramified unitary group $U(V)$ for $V$ of odd dimension $2n+1$; here there are two possibilities for a corresponding enhanced Tits datum. There are three cases overall that also appear as cases (1-a), (1-b), (1-c) in the next section:

a) Ramified quasi-split  $U(V)$ for $V$ of odd dimension $2n+1$,
$\{\mu\}=(1, 0, \ldots, 0)$, and $K$ the parahoric stabilizer of an almost $\tilde \pi$-modular lattice. This is a case of unitary exotic good reduction. 

b)  Ramified quasi-split $U(V)$ for $V$ of odd dimension $2n+1$, $\{\mu\}=(1^{(n-1)}, 0^{(n+2)})$
and $K$ is the parahoric stabilizer of a selfdual lattice. Then the local model has non-smooth special fiber by Subsection \ref{0-b} (or one can employ an argument using \cite[(5.2)]{PR} as in 2) above). 

c) Ramified quasi-split  orthogonal group ${\rm SO}(V)$ for $V$ of even dimension $2n+2$,
$\mu$ is the cocharacter that corresponds to the quadric homogeneous space and $\tilde K$ is the
parahoric stabilizer of an almost selfdual vertex lattice. Then the local model is not smooth; this follows by combining Propositions \ref{quadricLoc} and \ref{propQ} (II).  
\medskip

4) $(\tilde C_n, \omega^\vee_n, \{0\})$, $n\geq 2$. 
\smallskip

The local Dynkin diagram is $C$-$B_n$. As above, we see that
we have a ramified quasi-split but not split orthogonal group $SO(V)$ for $V$ of even dimension $2n+2$,
$\{\mu\}$ corresponds to the orthogonal Grassmannian of isotropic subspaces of  dimension $n+1$ and $\tilde K$ is the parahoric stabilizer of an almost selfdual vertex lattice. This is the case of orthogonal exotic good reduction
(see also Subsection \ref{exoticOrtho}).
\medskip

5) $(\tilde C_n, 2\omega^\vee_n, \{0\})$, $n\geq 1$.
\smallskip

 The local Dynkin diagram is  $C$-$BC_n$.
We have a ramified unitary group $U(V)$ for $V$ of odd dimension $2n+1$,
$\{\mu\}=(1^{(n)}, 0^{(n+1)})$, and $K$ is the parahoric stabilizer of an almost $\tilde \pi$-modular lattice. Using \cite[(5.2)]{PR} and an argument as in 2) above, we see that the special fiber is not smooth when $n>1$. If $n=1$, then 
$\{\mu\}=(1, 0, 0)$ and this is a case of unitary exotic reduction.
\medskip

6) $(\tilde G_2, \omega^\vee_2, \{0\})$. 
\smallskip

Again Frobenius is trivial and we have the quasi-split ramified triality group of type ${}^3D_4$. The tameness assumption implies that $p\neq 3$. Therefore, the main result of Haines-Richarz \cite{HR} is applicable and implies that the special fiber is not smooth. (In principle, this non-smoothness statement can also be deduced using Kumar's criterion -- see Subsection \ref{formula-*} below. However, this involves a lengthy calculation that appears to require computer assistance, see the (simpler) case of $G_2$ in \cite[(7.9)-(7.12)]{MOV}.)
\smallskip
  
\section{Strictly pseudo semi-stable reduction}\label{s:psssred}

\subsection{Statement of the result}  Our goal here is to examine smoothness of the affine Schubert varieties contained in the geometric special fiber of $\Mloc$. By Theorem \ref{2.9}, this fiber can be identified with the admissible locus ${\mathfrak A}_K(G,\{\mu\})$ in the partial affine flag variety for a group $\tilde G^\fl$ which is isogenous to $G^\fl$ but which has simply connected derived group. In the rest of this section, we will omit the tilde from the notation;   but it is understood that the affine Schubert varieties will  be for a group with simply connected derived group. This issue did not appear in our discussion of rational smoothness since this is defined via $\ell$-adic cohomology which is insensitive to radicial morphisms.
In fact, the rational smoothness of the affine Schubert variety $\overline{I^\flat w I^\flat/I^\flat}$ only depends on the element $w$ in the Iwahori-Weyl group, and does not depend on the reductive group itself. On the other hand, the smoothness of the affine Schubert variety $\overline{I^\flat w I^\flat/I^\flat}$ depends on the reductive group, not only the associated Iwahori-Weyl group. 
In other words, smoothness of the affine Schubert varieties in question (assuming simply connected derived group) 
depends on the enhanced Tits datum, and not only on the enhanced Coxeter datum.  

In this section, we consider the enhanced Tits data associated to the enhanced Coxeter data $(\tilde C_n, \omega^\vee_1, \{0\})$, $(\tilde B_n, \o^\vee_1, \{0, n\})$ and $(\tilde C_n, \o^\vee_n, \{0, 1\})$. They are as follows:  

\begin{enumerate}
\item The triple $(\tilde C_n, \o^\vee_1, \{0\})$ with $n \ge 2$:

\smallskip
\smallskip

\begin{center}
	
	\begin{footnotesize}

		\begin{tabular}{|c|c|c|}
			\hline
			\text{Label} & \text{Enhanced Tits datum} & \text{Linear algebra datum} \\
			\hline
			(1-a) &
			\begin{tikzpicture}[baseline=0]
			\node at (0,0) {$\bullet$};
			\node[above] at (0, 0) {$0$};
			\draw[implies-, double equal sign distance] (0.1, 0) to (0.9,0);
			\node  at (1,0) {$\circ$};
			\node[above] at (1, 0) {$1 \, \times$};
			\draw (1.1, 0) to (1.9,0);
			\node at (2,0) {$\circ$};		
			\node[above] at (2, 0) {$2$};
			\draw[dashed] (2.1, 0) to (2.9, 0);
			\node at (3, 0) {$\circ$}; 		
			\node[above] at (3, 0) {$n-1$};
			\node at (4, 0) {$\circ$};
			\draw[implies-, double equal sign distance] (3.1, 0) to (3.9,0);
			\node [above] at (4, 0) {$n$};
			\end{tikzpicture}
			& \text{Nonsplit $U_{2n+1}, r=1, \Lambda_0$}\\			\hline
			(1-b) &
			\begin{tikzpicture}[baseline=0]
			\node at (0,0) {$\circ$};
			\node[above] at (0, 0) {$0$};
			\draw[implies-, double equal sign distance] (0.1, 0) to (0.9,0);
			\node  at (1,0) {$\circ$};
			\node[above] at (1, 0) {$1$};
			\draw (1.1, 0) to (1.9,0);
			\node at (2,0) {$\circ$};		
			\node[above] at (2, 0) {$2$};
			\draw[dashed] (2.1, 0) to (2.9, 0);
			\node at (3, 0) {$\circ$}; 		
			\node[above] at (3, 0) {$n-1 \, \times$};
			\node at (4, 0) {$\bullet$};
			\draw[implies-, double equal sign distance] (3.1, 0) to (3.9,0);
			\node [above] at (4, 0) {$n$};
			\end{tikzpicture}
			& \text{Nonsplit $U_{2n+1}, r=n-1, \Lambda_n$}\\
			\hline
			(1-c) &
			\begin{tikzpicture}[baseline=0]
			\node at (0,0) {$\bullet$};
			\node[above] at (0, 0) {$0$};
			\draw[implies-, double equal sign distance] (0.1, 0) to (0.9,0);
			\node  at (1,0) {$\circ$};
			\node[above] at (1, 0) {$1 \, \times$};
			\draw (1.1, 0) to (1.9,0);
			\node at (2,0) {$\circ$};		
			\node[above] at (2, 0) {$2$};
			\draw[dashed] (2.1, 0) to (2.9, 0);
			\node at (3, 0) {$\circ$}; 		
			\node[above] at (3, 0) {$n-1$};
			\node at (4, 0) {$\circ$};
			\draw[-implies, double equal sign distance] (3.1, 0) to (3.9,0);
			\node [above] at (4, 0) {$n$};
			\end{tikzpicture}
			& \text{Nonsplit $SO_{2n+2}, r=1, \Lambda_0$}\\
			\hline
		\end{tabular}
\end{footnotesize}

\end{center}
\smallskip
	
\item The triple $(\tilde B_n, \o^\vee_1, \{0, n\})$ with $n \ge 3$:

\smallskip
\smallskip

\begin{center}

\begin{footnotesize}
\begin{tabular}{|c|c|c|}
\hline
\text{Label} & \text{Enhanced Tits datum} & \text{Linear algebra datum} \\
\hline
(2-a) &
\begin{tikzpicture}[baseline=0]
		\node at (0,0) {$\bullet$};
		\node[above] at (0, 0) {$n$};
		\draw[implies-, double equal sign distance] (0.1, 0) to (0.9,0);
		\node  at (1,0) {$\circ$};
		\node[above] at (1, 0) {$n-1$};
		\draw (1.1, 0) to (1.9,0);
		\node at (2,0) {$\circ$};	
		\node[above] at (2, 0) {$n-2$};
		\draw[dashed] (2.1, 0) to (2.9, 0);
		\node at (3.0, 0) {$\circ$};
		\node[above] at (3, 0) {$2$};
		\draw[-] (3.1,0.1) to (3.9, 0.35);
		\node at (4, 0.4) {$\bullet$};
		\node [above] at (4, 0.4) {$0$};
		\draw[-] (3.1, - 0.1) to (3.9, -0.35);
		\node at (4,-0.4) {$\circ$};
		\node [above] at (4, -0.4) {$1 \, \times$};
	\end{tikzpicture}
	& \text{Split $SO_{2n+1}$, $r=1$, $(\Lambda_0,  \Lambda_n)$} \\
\hline
(2-b) & \begin{tikzpicture}[baseline=0]
		\node at (0,0) {$\bullet$};
		\node[above] at (0, 0) {$n$};
		\draw[-implies, double equal sign distance] (0.1, 0) to (0.9,0);
		\node  at (1,0) {$\circ$};
		\node[above] at (1, 0) {$n-1$};
		\draw (1.1, 0) to (1.9,0);
		\node at (2,0) {$\circ$};	
		\node[above] at (2, 0) {$n-2$};
		\draw[dashed] (2.1, 0) to (2.9, 0);
		\node at (3.0, 0) {$\circ$};
		\node[above] at (3, 0) {$2$};
		\draw[-] (3.1,0.1) to (3.9, 0.35);
		\node at (4, 0.4) {$\bullet$};
		\node [above] at (4, 0.4) {$0$};
		\draw[-] (3.1, - 0.1) to (3.9, -0.35);
		\node at (4,-0.4) {$\circ$};
		\node [above] at (4, -0.4) {$1 \, \times$};
	\end{tikzpicture}
	& \text{$U_{2n}$, $r=1$, $(\Lambda_0,  \Lambda_n)$} \\
\hline
\end{tabular}
\end{footnotesize}

\end{center}
\smallskip

\item The triple $(\tilde C_n, \o^\vee_n, \{0, 1\})$ with $n \ge 2$:

\smallskip
\smallskip

\begin{center}

\begin{footnotesize}

\begin{tabular}{|c|c|c|}
\hline
\text{Label} & \text{Enhanced Tits datum} & \text{Linear algebra datum} \\
\hline
(3-a) &
\begin{tikzpicture}[baseline=0]
		\node at (0,0) {$\bullet$};
		\node[above] at (0, 0) {$0$};
		\draw[-implies, double equal sign distance] (0.1, 0) to (0.9,0);
		\node  at (1,0) {$\bullet$};
		\node[above] at (1, 0) {$1$};
		\draw (1.1, 0) to (1.9,0);
		\node at (2,0) {$\circ$};		
		\node[above] at (2, 0) {$2$};
		\draw[dashed] (2.1, 0) to (2.9, 0);
		\node at (3, 0) {$\circ$}; 		
		\node[above] at (3, 0) {$n-1$};
		\node at (4, 0) {$\circ$};
		\draw[implies-, double equal sign distance] (3.1, 0) to (3.9,0);
		\node [above] at (4, 0) {$n \, \times$};
	\end{tikzpicture}
	& \text{Split $Sp_{2n}, r=n, (\Lambda_0, \Lambda_1)$}\\
\hline
(3-b) &
\begin{tikzpicture}[baseline=0]
		\node at (0,0) {$\bullet$};
		\node[above] at (0, 0) {$0$};
		\draw[implies-, double equal sign distance] (0.1, 0) to (0.9,0);
		\node  at (1,0) {$\bullet$};
		\node[above] at (1, 0) {$1$};
		\draw (1.1, 0) to (1.9,0);
		\node at (2,0) {$\circ$};		
		\node[above] at (2, 0) {$2$};
		\draw[dashed] (2.1, 0) to (2.9, 0);
		\node at (3, 0) {$\circ$}; 		
		\node[above] at (3, 0) {$n-1$};
		\node at (4, 0) {$\circ$};
		\draw[-implies, double equal sign distance] (3.1, 0) to (3.9,0);
		\node [above] at (4, 0) {$n \, \times$};
	\end{tikzpicture}
	& \text{Nonsplit $SO_{2n+2}, r=n+1, (\Lambda_0, \Lambda_2)$}\\
\hline
\end{tabular}
\end{footnotesize}

\end{center}
\end{enumerate}
\medskip

Here the numbers above the vertices of the Dynkin diagrams are the labelings. 
The main result of this section is 

\begin{proposition}\label{notSPSS}
The cases (1-b), (1-c), (2-b) and (3-b) are not strictly pseudo semi-stable reduction. 
\end{proposition}

We prove Proposition \ref{notSPSS} by showing that at least one of the irreducible components of the geometric fiber is not smooth. 

\subsection{Kumar's criterion}\label{formula-*} 

Note that for any $x \in \tW$ and parahoric subgroup $K$, the smoothness of $\overline{ K^\flat x  K^\flat/ K^\flat}$ is equivalent to the smoothness  of $\overline{ I^\flat w I^\flat/ I^\flat}$, where we set $w=\max\{W_K x W_K\}$. To study the case (1-b), we use Kumar's criterion \cite{Ku}, which we recall here.  

Let $Q$ be the quotient field of the symmetric algebra of the root lattice. Following \cite{B}, we fix a reduced expression $\underline w=\tau s_{\a_1} \cdots s_{\a_l}$ of $w$, where $\tau$ is a length-zero element in $\tW$ and $\a_1, \ldots, \a_l$ are affine simple roots. For any $x \le w$, we define 
\begin{equation}\label{polQ}
e_x X(w)=\sum_{(s_1, \ldots, s_l)} \prod_{j=1}^l s_1 \cdots s_j(\a_j)^{-1} \in Q,
\end{equation}
where the sum runs over all sequences $(s_1, \ldots, s_l)$ such that $s_j=1$ or $s_{\a_j}$ for any $j$ and $s_1 \cdots s_l=x$. We call such sequences the \emph{subexpressions} for $x$ in $\underline w$. It is known that $e_x X(w)$ is independent of the choice of the reduced expression $\underline w$  of $w$.  Kumar's criterion gives a necessary and sufficient condition for the Schubert variety to be singular in terms of $e_1 X(w)$ when the field is $\BC$.  It is not known if a similar result holds in positive characteristic. However, one implication can be shown following \cite[\S 6]{HR}. The statement we will use is the following:

\begin{theorem}\label{ku}
If $e_1 X(w) \neq \prod_{\{\a \in\Phi^+_{\rm af}\mid  s_\a \le w\}} \a^{-1}$, then the Schubert variety $\overline{ I^\flat w  I^\flat/ I^\flat}$ is singular.
\end{theorem}

\subsection{The case (1-b)}\label{0-b} Set $\l'=(-1, -1, \ldots, -1, 0)$. By Lemma \ref{max}, $t^{\l'}$ is the maximal element in $W_{\le \l', K}$. Set $w=\max(W_K t^{\l'} W_K)$. By direct computation, $$w=(s_{n-1} s_{n-2} \cdots s_0) (s_1 s_2 \cdots s_n) w^K_0,$$ where $w^K_0$ is the longest element in $W_K$. 

As $s_n$ is the reflection of a long root, and the other simple reflections are reflections of short roots, in any expression of $1$, the simple reflection $s_n$ must appear an even number of times. Note that in a reduced expression of $w$, the simple reflection $s_n$ appears only once, thus $s_n$ does not appear in the subexpression for $1$. Moreover, the reduced expression $\underline w$ of $w$ may be chosen to be of the form $\ldots s_{n-1} s_n s_{n-1} \ldots$. Thus any subexpression $(s_1, \ldots, s_j)$ for $1$ in $\underline w$ is of the form $(\ldots, 1, 1, 1, \ldots)$, $(\ldots, s_{n-1}, 1, s_{n-1}, \ldots)$, $(\ldots, s_{n-1}, 1, 1, \ldots)$, $(\ldots, 1, 1, s_{n-1}, \ldots)$. Here the three terms in the middle are the subexpressions of $s_{n-1} s_n s_{n-1}$ in which $s_n$ does not appear. A direct computation for the rank-two Weyl groups shows that \begin{equation}\label{rank2}
\begin{aligned}
e_1 X(s_{n-1} s_n s_{n-1}) &=-e_{s_{n-1}} X(s_{n-1} s_n s_{n-1}) =\\
&=\frac{-\<\a_n, \a_{n-1}^\vee\>}{\a_n \a_{n-1} s_{n-1}(\a_n)}=\frac{2}{\a_n \a_{n-1} s_{n-1}(\a_n)}.
\end{aligned}
\end{equation}
We rewrite the formula \eqref{polQ} as 
\begin{equation*}
\begin{aligned}e_1 X(w) &=\sum_{(\ldots, 1, 1, 1, \ldots)} \prod_{j=1}^l s_1 \cdots s_j(\a_j)^{-1}+\sum_{(\ldots, s_{n-1}, 1, s_{n-1}, \ldots)} \prod_{j=1}^l s_1 \cdots s_j(\a_j)^{-1} \\ &+\sum_{(\ldots, s_{n-1}, 1, 1, \ldots)} \prod_{j=1}^l s_1 \cdots s_j(\a_j)^{-1}+\sum_{(\ldots, 1, 1, s_{n-1}, \ldots)} \prod_{j=1}^l s_1 \cdots s_j(\a_j)^{-1}.
\end{aligned}\end{equation*}
By \eqref{rank2}, all  coefficients in the first and in the second line are multiples of $2$. By Theorem \ref{ku}, $\overline{ I^\flat w  I^\flat/ I^\flat}$ is not smooth.		

\subsection{The case (1-c)}\label{1-c} The special fiber is irreducible but not smooth. As was also mentioned in Section  \ref{s:maingood}, Case (3),  this follows by combining Propositions \ref{quadricLoc} and \ref{propQ} (II).

\subsection{The case (2-b)}\label{1-b} Set $\l'=(0, 0, \ldots, 0, 1)$. By Lemma \ref{max}, $t^{\l'}$ is the maximal element in $W_{\le \l', K}$. Set $w=\max(W_K t^{\l'} W_K)$. By direct computation, $$w=\tau (s_{n-1} s_{n-2} \cdots s_2) (s_0 s_1 \cdots s_n) w^K_0,$$ where $\tau$ is the unique length-zero element in $\tW$ with  $t^{\l'} \in W_a \tau$ and $w^K_0$ is the longest element in $W_K$. 

Note that in a reduced expression of $w$, the simple reflection $s_n$ appears only once, thus $s_n$ does not appear in the subexpression for $1$. Similar to the argument in the subsection \ref{0-b}, $\overline{ I^\flat w  I^\flat/ I^\flat}$ is not smooth. 

\subsection{The case (3-b)}\label{case(2-b)}
 Set $\l'=(-\frac{1}{2}, -\frac{1}{2}, \ldots, -\frac{1}{2})$. By Lemma \ref{max}, $t^{\l'}$ is the maximal element in $W_{\le \l', K}$. By direct computation, $$\max(W_K t^{\l'} W_K)=\tau w_0^{\{n\}} w_0^{\{0, n\}} w^K_0,$$ where $\tau$ is the unique length-zero element in $\tW$ with  $t^{\l'} \in W_a \tau$  and where $w_0^{K'}$ is the longest element in $W_{K'}$ for $\tilde K'=\{0, 1\}, \{n\}$ or $\{0, n\}$. Note that $$\overline{ K^\flat t^{\l'}  K^\flat/ K^\flat}=\overline{ I^\flat t^{\l'} K^\flat/ K^\flat} \cong \overline{ I^\flat w_0^{\{n\}} w_0^{\{0, n\}} K^\flat/ K^\flat} \cong \overline{ I^\flat w_0^{\{n\}} w_0^{\{0, n\}}  K^\flat_1/ K^\flat_1} \subset  K^\flat_2/ K^\flat_1,$$ where $\tilde K_1=\{0, 1, n\}$ and $\tilde K_2=\{n\}$. 

Let $U_{K^\flat_2}$ be the pro-unipotent radical of $ K^\flat_2$ and $G'= K^\flat_2/U_{ K^\flat_2}$  the reductive quotient of $ K^\flat_2$. Note that $G'_{\ad}$ is the adjoint group of type $B_n$ over $k$. Let $P= K^\flat_1/U_{ K^\flat_2}$. This is a standard parabolic subgroup of $G'$. We have $ K^\flat_2/ K^\flat_1 \cong G'/P$. This is a partial flag variety of finite type. 

\begin{center}
\begin{footnotesize}
\begin{tabular}{|c||c|}
\hline
\text{Group} & \text{Affine/Finite Dynkin diagram} \\
\hline
$G$ & \begin{tikzpicture}[baseline=0]
		\node at (0,0) {$\bullet$};
		\node[above] at (0, 0) {$0$};
		\draw[implies-, double equal sign distance] (0.1, 0) to (0.9,0);
		\node  at (1,0) {$\bullet$};
		\node[above] at (1, 0) {$1$};
		\draw (1.1, 0) to (1.9,0);
		\node at (2,0) {$\circ$};		
		\node[above] at (2, 0) {$2$};
		\draw[dashed] (2.1, 0) to (2.9, 0);
		\node at (3, 0) {$\circ$}; 		
		\node[above] at (3, 0) {$n-1$};
		\node at (4, 0) {$\bullet$};
		\draw[-implies, double equal sign distance] (3.1, 0) to (4,0);
		\node [above] at (4, 0) {$n$};
	\end{tikzpicture} \\
\hline
$G'$ & \begin{tikzpicture}[baseline=0]
		\node at (0,0) {$\bullet$};
		\node[above] at (0, 0) {$n$};
		\draw[implies-, double equal sign distance] (0.1, 0) to (0.9,0);
		\node  at (1,0) {$\bullet$};
		\node[above] at (1, 0) {$n-1$};
		\draw (1.1, 0) to (1.9,0);
		\node at (2,0) {$\circ$};		
		\node[above] at (2, 0) {$n-2$};
		\draw[dashed] (2.1, 0) to (2.9, 0);
		\node at (3, 0) {$\circ$}; 		
		\node[above] at (3, 0) {$1$};
	\end{tikzpicture} \\
	\hline
\end{tabular}
\end{footnotesize}
\end{center}
\smallskip

In the table, the parahoric subgroup $\breve{K}_1$ of $G$ and the parabolic subgroup $P$ of $G'$ correspond to the set of vertices filled with black color  in the corresponding diagram. 

The finite Dynkin diagram of $G'$ is obtained from the local Dynkin diagram of $G$ by deleting the vertex labeled $n$. The labeling of the Dynkin diagram is not inherited from the local Dynkin diagram of $G$, but is the standard labeling of the finite Dynkin diagram in \cite{Bou}. The reason is that we will apply the smoothness criterion for finite Schubert varieties, and we follow the convention for finite Dynkin diagrams and finite Weyl groups. We identify the finite Weyl group $W_{G'}$ of $G'$ with the group of permutations of $\{\pm 1, \pm 2, \ldots, \pm n\}$.

Under the natural isomorphism $ K^\flat_2/ K^\flat_1 \cong G'/P$, the closed subset of the affine partial flag variety $\overline{ I^\flat w_0^{\{n\}} w_0^{\{0, n\}}  K^\flat_1/ K^\flat_1}$ is isomorphic to the closed subvariety $\overline{B' w' P/P}$ of the finite type partial flag variety, where $B'= I^\flat/U_{ K^\flat_2}$ is a Borel subgroup of $G'$ and $w' \in W_{G'}$ is the permutation $(1, -n) (2, -(n-1)) \cdots$. 

The smoothness of $\overline{B' w' P/P}$ is equivalent to the smoothness of $\overline{B' w' w_0^P B'/B'}$, where $w_0^P$ is the longest element in the Weyl group $W_P$ of $P$. The element $w' w_0^P$ is the permutation of $\{\pm 1, \pm 2, \ldots, \pm n\}$ sending $1$ to $-2$, $2$ to $-3,\dots, n-1$ to $-n$ and $n$ to $-1$. By the pattern avoidance criterion (see \cite[Thm.~8.3.17]{BL}), $\overline{B' w' w_0^P B'/B'}$ is not smooth. Hence $\overline{ K^\flat {\l'}  K^\flat/ K^\flat }$ is not smooth. 

\section{Proof of one implication in Theorem \ref{mainssred}}\label{s:oneimpl}
Assume that $\Mloc$ has strictly semi-stable reduction.  Inspection of all cases considered in the previous section shows that then $(G, \{\mu\}, K)$ appears in the list of Theorem \ref{mainssred}. In the remaining section of the paper, we show that indeed for all triples $(G, \{\mu\}, K)$ on this list the corresponding associated local models have semi-stable reduction. As a consequence of this assertion, we obtain the following somewhat surprising result.
\begin{corollary}
Let $(G, \{\mu\}, K)$ be a triple over $F$ such that $G$ splits over a tame extension of $F$. Assume $p\neq 2$.  Assume also that the group $G$ is adjoint and absolutely simple. 
 If  $\Mloc$ has strictly pseudo semi-stable reduction, then $\Mloc$ has (strictly) semi-stable reduction, in particular, it is a regular scheme with reduced special fiber. \qed
\end{corollary}

\section{Proof of the other implication of Theorem \ref{mainssred}}\label{s:converse1}
In this section, we go through the list of Theorem \ref{mainssred}, and produce in each case an LM triple  $(G, \{\mu\}, K)$ in the given central isogeny type which has semi-stable reduction. By Lemma \ref{devislem}, we may indeed assume that $G$ is a central extension of the adjoint group appearing in the list of Theorem \ref{mainssred}. So, for instance, in this section, we work with $\GL$ instead of $\SL$, ${\rm GSp}$ instead of $\Sp$, and, in some instances, with ${\rm GO}$ instead of $\SO$.

We precede this by the following remarks. The first remark is that the  locus where  $\Mloc$ has semi-stable reduction is open and $\CG$-invariant. Therefore, in order to show that  $\Mloc$ has semi-stable reduction, it suffices to check this in a closed point of the unique closed $\CG\otimes_{O_F}k$-orbit of the special fiber. 

The second remark is that we may always make an unramified field extension $F'/F$. This implies that, in checking semi-stable reduction, we may assume that in the LM triple $(G, \{\mu\}, K)$, the group $G$ is residually split. 

The third remark is that in most of the cases which we treat, the LM triples are of ``EL or PEL type." Then the corresponding local models of \cite{PZ} have a more standard/classical description, as closed subschemes of linked classical (i.e., not affine) Grassmannians. This description, which was in fact given in earlier works, is established in  \cite[7.2, 8.2]{PZ}; we use this in our analysis, sometimes without further mention. There are two   cases which are different: These are the LM triples for (special) orthogonal groups 
and the coweight $\omega^\vee_1$ (i.e., $r=1$). The corresponding local models have as generic fiber a quadric hypersurface. These are just of ``Hodge type" and, for them, we have to work harder to first establish a standard description. Most of this is done in Subsection
\ref{ss:quadricLM} with the key statement being Proposition \ref{quadricLoc}.

\subsection{Preliminaries on  $\GL_n$}\label{ss:GLnprelim}
In this subsection, we consider the LM triple  
\[
(G=\GL_n,\{\mu\}=\mu_r:=(1^{(r)},0^{(n-r)}), K_I)
\]
 for some $r\geq 1$, where $K_I$ is the stabilizer of a lattice chain $\Lambda_I$ for some non-empty subset $I\subset \{0, 1,\ldots,n-1\}$. We use the notation $(\GL_n, \mu_r, I)$. We follow G\"ortz \cite[\S 4.1]{Goertz} for the description of the local model in this case (the \emph{standard local model})  and of an open subset $U$ around the worst point, cf. \cite[Prop.4.5]{Goertz}. 

The local model  $\BM^\loc_I(\GL_n, \mu_r)$  represents the following functor on $O_F$-schemes. Write $I=\{i_0< i_1<\cdots<i_{m-1}\}$. Then  $\BM^\loc_I(\GL_n, \mu_r)(S)$ is the set of commutative diagrams 
\begin{equation}
\begin{aligned}
\xymatrix{
	&\Lambda_{i_0, S}  \ar[r] & \Lambda_{i_1,S}  \ar[r]  & \cdots 
	\ar[r] & \Lambda_{i_{m-1},S}  \ar[r]^\pi & \Lambda_{i_0,S} \\
	&\CF_0  \ar@{^{(}->}[u] \ar[r] & \CF_1 \ar@{^{(}->}[u]  \ar[r] \ar@{^{(}->}[u]  & \cdots  \ar[r] & \CF_{m-1} \ar@{^{(}->}[u] 
	\ar[r] & \CF_0 \ar@{^{(}->}[u] }
\end{aligned}
\end{equation}
where $\Lambda_i$ is the lattice generated by $e^i_1:=\pi^{-1}e_1, \ldots, e^i_i:=\pi^{-1}e_i, e^i_{i+1}:=e_{i+1},\ldots,e^i_n:=e_n$, $\Lambda_{i,S}$ is $\Lambda_i \otimes_{O_F} \CO_S$, $\pi$ is a fixed uniformizer of $F$, and where 
the $\CF_\kappa$ are locally free $\CO_S$-submodules of rank $r$ which
Zariski-locally on $S$ are direct summands of $\Lambda_{i_\kappa,S}$.

\subsection{The case $(\GL_n, r=1, I)$,  $I$ arbitrary }
That in this case we have semi-stable reduction is well-known and follows from\footnote{More precisely, in loc.~cit., the case $I=\{0,\ldots, n-1\}$ is considered, but the case of an arbitrary subset $I$ is the same. } \cite[Prop. 4.13]{Goertz}.

\subsection{Preliminaries on  $(\GL_n, r, \{0, \kappa\})$,  $r$ arbitrary }  Note that to verify the remaining case for $\GL_n$, we only need the case $\kappa=1$, cf.  Subsection \ref{11.4}. However, as we will see later, in order to verify the cases for other classical groups, we need to describe the incidence relation between $0$ and $\kappa$ for some other $\kappa$. So we discuss arbitrary $\kappa$ here.  

In terms of the bases $\{e^i_1, \ldots, e^i_n\}$ of $\Lambda_{i, S}$, the transition maps $\Lambda_{0, S}\to \Lambda_{\kappa, S}$, resp., $\pi\colon \Lambda_{\kappa, S}\to \Lambda_{0, S}$ are given by the diagonal matrices
\begin{equation}
\phi_{0, \kappa}=\diag(\pi^{(\kappa)}, 1^{(n-\kappa)}) , \text{ resp., } \phi_{ \kappa, 0}=\diag(1^{(\kappa)}, \kappa^{(n-\kappa)}) .
\end{equation}

For the open subset $U$ around the worst point we take the pair of subspaces $\CF_0$ of $\Lambda_{0, S}$, resp., $\CF_\kappa$ of $\Lambda_{\kappa, S}$,  given by the $n\times r$-matrices
\begin{eqnarray*}
	&& \CF_0 {=} \left( \begin{array}{cccc} 
		1     &          &        &          \\
		&     1    &        &          \\
		&          & \ddots &          \\
		&          &        & 1        \\
		a^0_{11} & a^0_{12} & \cdots & a^0_{1r} \\
		\vdots  & \vdots   &        &  \vdots  \\
		a^0_{n-r,1} & a^0_{n-r,2} & \cdots & a^0_{n-r,r}
	\end{array}  \right), \\
	&& \CF_\kappa {=} \left( \begin{array}{cccc} 
		a^\kappa_{n-r-\kappa+1,1} & a^\kappa_{n-r-\kappa+1,2} & \cdots & a^\kappa_{n-r-\kappa+1,r} \\
		\vdots & \vdots & & \vdots \\
		a^\kappa_{n-r,1} & a^\kappa_{n-r,2} & \cdots & a^\kappa_{n-r,r} \\
		1     &          &        &          \\
		&     1    &        &          \\
		&          & \ddots &          \\
		&          &        & 1        \\
		a^\kappa_{11} & a^\kappa_{12} & \cdots & a^\kappa_{1r} \\
		\vdots  & \vdots   &        &  \vdots  \\
		a^\kappa_{n-r-\kappa,1} & a^\kappa_{n-r-\kappa,2} & \cdots & a^\kappa_{n-r-\kappa,r}
	\end{array}  \right).
\end{eqnarray*} 
Then the incidence relation from $0$ to $\kappa$ is given by 
$$
\phi_{0, \kappa}\cdot\CF_0=\CF_\kappa\cdot N_0 ,
$$ and the incidence relation from $\kappa$ to $0$ is given by
$$
\phi_{ \kappa, 0}\cdot\CF_\kappa=\CF_0\cdot N_\kappa ,
$$
where $N_0, N_\kappa\in \GL_r(\CO_S)$ are uniquely defined matrices. These equations can now be evaluated and lead  to the following description of $U$:
\begin{proposition}[{\cite[\S 4.4.5]{Goertz}}]Let $\kappa\leq r$. Let 
$$
A = (a^0_{i,j})_{i, j = 1, \dots, \kappa}, B = (a^\kappa_{i,j})_{i = 1, \dots, \kappa, j = r-\kappa+1, \dots, r}$$
	be $\kappa \times \kappa$-matrices of indeterminates.
	Then
	$$ U \cong \Spec O_F[A,B]/(BA-\pi, AB-\pi) \times V, $$
	where 
\[V =\Spec O_F[a^0_{i,j}]_{ i = 1, \dots, r, j = \kappa+1, \dots, n-r} \times   \Spec O_F[a^\kappa_{i,j}]_{ i=1, \dots, r-\kappa, j = 1, \dots, \kappa}\]
 is an affine space $ \BA^{(n-r)r-\kappa^2}$ over $O_F$.\qed
\end{proposition}
Something analogous holds in the case when $\kappa>r$, cf.~loc.~cit.. 

\subsection{The case $(\GL_n, r=1, I=\{i, i+1\})$,  $r$ arbitrary}\label{11.4}  After changing the basis,  we may assume that $i=0$. Then the above proposition implies that $U$ is a product of $\Spec O_F[X, Y]/(XY-\pi)$ and an affine space $\BA^{(n-r)r-1}$. Hence $U$ is regular and the special fiber is the union of two smooth divisors crossing normally along a smooth subscheme of codimension $2$. Hence $U$ has semi-stable reduction.
\begin{remark}\label{oneenough}
	In contrast to the case of a general subset $I$, in this case the incidence condition from $\CF_0$ to $\CF_1$ automatically implies the incidence relation from $\CF_1$ to $\CF_0$.  
\end{remark}

\subsection{The case $({\rm GSp}_{2n}, r=n, \{0, 1\})$} In the case of ${\rm GSp}_{2n}$ there is only one non-trivial minuscule coweight $\{\mu\}=\mu_n$. Let $e_1,\ldots, e_{2n}$ be a symplectic basis of $F^{2n}$, i.e., $\langle e_i, e_{2n-j+1}\rangle=\pm\delta_{ij}$ for $i, j\leq 2n$ (with sign $+$ if $i= j\leq n$ and sign $-$ if $n+1\leq i= j$). Then the standard lattice chain  is self-dual, i.e., $\Lambda_i$ and $\Lambda_{2n-1}$ are paired by a perfect pairing. In this case, a parahoric subgroup $K$ is  the stabilizer of a \emph{selfdual} periodic lattice chain $\Lambda_I$ , i.e., $I$ satisfies $i\in I\iff 2n-i\in I$.  In this case, the local model is contained in the closed subscheme $\BM^{\rm naive}_I({\rm GSp}_{2n}, \mu_n)$ of the local model $\BM^\loc_I(\GL_{2n}, \mu_n)$ given by the condition that 
\begin{equation}\label{perpcondSp}
\CF_i=\CF_{2n-i}^\perp, \quad i\in I .
\end{equation}
In fact, by \cite{Go2}, it is equal to this closed subscheme but we will not use this fact. 

Now let $I_0=\{0, 1, 2n-1\}$. Then, since $\CF_{2n-1}$ is determined by $\CF_1$ via the identity \eqref{perpcondSp}, we obtain a closed embedding $\BM^{\rm naive}_{I_0}({\rm GSp}_{2n}, \mu_n)\subset \BM^\loc_{\{0, 1\}}(\GL_{2n}, \mu_n)$.

As open subset $U$ of the worst point we take the scheme of $(\CF_0, \CF_1)$, where 
\begin{eqnarray*}
	&& \CF_0 {=} \left( \begin{array}{cccc} 
		1     &          &        &          \\
		&     1    &        &          \\
		&          & \ddots &          \\
		&          &        & 1        \\
		c_{11} & c_{12} & \cdots & c_{1n} \\
		\vdots  & \vdots   &        &  \vdots  \\
		c_{n1} & c_{n2} & \cdots & c_{nn}
	\end{array}  \right), \\
	&& \CF_1 {=} \left( \begin{array}{cccc} 
		a_{n1} & a_{n2} & \cdots & a_{nn} \\
		1     &          &        &          \\
		&     1    &        &          \\
		&          & \ddots &          \\
		&          &        & 1        \\
		a_{11} & a_{12} & \cdots & a_{1n} \\
		\vdots  & \vdots   &        &  \vdots  \\
		a_{n-1,1} & a_{n-1,2} & \cdots & a_{n-1,n}
	\end{array}  \right).
\end{eqnarray*}
The condition that $\CF_0$ be a totally isotropic subspace of $\Lambda_{0, S}$ is expressed by
\begin{equation}
c_{\mu \nu}=c_{n-\nu+1, n-\mu+1} . 
\end{equation}
The incidence relation from $\CF_0$ to $\CF_1$ is given by the following system of equations,
$$\begin{pmatrix}
\pi&0&\ldots&0\\
0&1&\ldots&0\\
\vdots&&\ddots&\\
0&0&\ldots&1\\
c_{11}&c_{12}&\ldots&c_{1n}\\
\vdots&\vdots&&\vdots\\
c_{n1}&c_{n2}&\ldots&c_{n n}
\end{pmatrix}
=
\begin{pmatrix}
a_{n 1}&a_{n2}&\ldots&a_{n n}\\
1&0&\ldots&0\\
0&1&\ldots&0\\
\vdots&&\ddots&\\
0&0&\ldots&1\\
a_{11}&a_{12}&\ldots&a_{1n}\\
\vdots&\vdots&&\vdots\\
a_{n-1, 1}&a_{n-1,2}&\ldots&a_{n-1, n}
\end{pmatrix}
\cdot
\begin{pmatrix}
0&1&0&\ldots&0\\
0&0&1&\ldots&0\\
\vdots&&&\ddots\\
0&0&0&\ldots&1\\
c_{11}&c_{12}&c_{13}&\ldots&c_{1n}
\end{pmatrix}
$$
The first row of this matrix identity gives
\begin{equation}\label{aceq}
a_{n n}c_{11}=\pi ,
\end{equation}
and allows one to eliminate $a_{n 1}, \ldots,a_{n, n-1}$. The last $n-1$ entries of the $n+2$-th row allow one to eliminate $a_{11},\ldots, a_{1, n-1}$, the last $n-1$ entries of the $n+3$-th row allow one to eliminate $a_{21},\ldots, a_{2, n-1}$, etc., until the last $n-1$ entries of  the $2n$-th row eliminate $a_{n-1, 1},\ldots, a_{n-1, n-1}$. Finally, the first column of these rows allows one to eliminate $c_{2 1},,\ldots,c_{n1}$.  All in all, we keep the entries $a_{1 n},\ldots, a_{n n}$, $c_{1 1}$,  and $ c_{\mu \nu}$ for $\mu\geq\nu>1$, which are subject to equation \eqref{aceq}. 

Let ${\rm Grass}^{\rm lagr}(\Lambda_0)\times {\rm Grass}(\Lambda_1)$ be the product of the Grassmannian variety of Lagrangian subspaces of $\Lambda_0$ and of the Grassmannian variety of subspaces of dimension $n$ of $\Lambda_1$.  Let $\BM$ denote its closed subscheme of elements $(\CF_0, \CF_1)$ such that $\CF_0$ is incident to $\CF_1$. Note that $\BM$ and $\BM^\loc_{I_0}({\rm GSp}_{2n}, \mu_n)$ have identical generic fibers. We have a chain of closed embeddings
\begin{equation}
\BM^\loc_{I_0}({\rm GSp}_{2n}, \mu_n)\subset \BM^{\rm naive}_{I_0}({\rm GSp}_{2n}, \mu_n)\subset \BM .
\end{equation}
But we just proved that $\BM$ has semi-stable reduction, and is therefore flat over $O_F$. Hence all inclusions are equalities, since we can identify all three schemes with the flat closure of the generic fiber of $\BM^\loc_{I_0}({\rm GSp}_{2n}, \mu_n)$ in $\BM^\loc_{\{0,1\}}({\rm GL}_{2n}, \mu_n)$. In particular, $\BM^\loc_{I_0}({\rm GSp}_{2n}, \mu_n)$ has semi-stable reduction (in fact, with special fiber the union of two smooth divisors meeting transversally in a smooth subscheme of codimension $2$). 

\begin{remark}\label{oneenoughSp}
	Again, as in the case of $\GL_n$,  in this case  the incidences  from $\CF_1$ to $\CF_{2n-1}$ and from $\CF_{2n-1}$ to $\CF_{0}$ are automatic.  
\end{remark}

\subsection{The case $({\bf split }\,{\rm GO}_{2n}, r=n, \{1\})$ } In this subsection, we assume $p\neq 2$. Let $e_1,\ldots, e_{2n}$ be a Witt basis of $F^{2n}$, i.e., $\langle e_i, e_{2n-j+1}\rangle=\delta_{ij}$ for $i, j\leq 2n$. Then the standard lattice chain  is self-dual, i.e., $\Lambda_i$ and $\Lambda_{2n-i}$ are paired by a perfect pairing. In this case $K$ is the parahoric stabilizer of a \emph{selfdual} periodic lattice chain $\Lambda_I$ , i.e., $I$ has the property $i\in I\iff 2n-i\in I$.    In this case, by \cite[8.2.3]{PZ}, the local model is contained in the closed subscheme $\BM^{\rm naive}_I({\rm GO}_{2n}, \mu_n)$ of the local model $\BM^\loc_I(\GL_{2n}, \mu_n)$ given by the condition that 
\begin{equation}\label{perpcond}
\CF_i=\CF_{2n-i}^\perp, \quad i\in I .
\end{equation}

Now let $I_0=\{1, 2n-1\}$. Then, since $\CF_{2n-1}$ is determined by $\CF_1$ via the identity \eqref{perpcond}, we obtain a closed embedding $\BM^{\rm loc}_{I_0}({\rm GO}_{2n}, \mu_n)\subset \BM^\loc_{\{1\}}(\GL_{2n}, \mu_n)$.

As open subset $U$ of the worst point we take the scheme of $( \CF_1, \CF_{2n-1})$, where 
\begin{eqnarray*}
	&& \CF_1 {=} \left( \begin{array}{cccc} 
		a_{n,1} & a_{n,2} & \cdots & a_{n,n} \\
		1     &          &        &          \\
		&     1    &        &          \\
		&          & \ddots &          \\
		&          &        & 1        \\
		a_{11} & a_{12} & \cdots & a_{1n} \\
		\vdots  & \vdots   &        &  \vdots  \\
		a_{n-1,1} & a_{n-1,2} & \cdots & a_{n-1,n}
	\end{array}  \right).\\
	&& \CF_{2n-1} {=} \left( \begin{array}{cccc} 
		
		0&1&0\ldots&0\\
		0&0&1\ldots&0\\
		\vdots&&&\vdots\\
		0&0&0\ldots&1\\
		b_{11} & b_{12} & \cdots & b_{1n} \\
		\vdots  & \vdots   &        &  \vdots  \\
		b_{n,1} & b_{n,2} & \cdots & b_{n,n}\\
		1& 0 & \cdots & 0 \\
	\end{array}  \right).
\end{eqnarray*}
The condition that $\CF_1$ and $\CF_{2n-1}$ be orthogonal  is expressed by
\begin{equation}\label{baorth}
b_{\mu \nu}=-a_{n-\nu+1, n-\mu+1} .
\end{equation} 
Recall the spin condition on $\CF_1$, cf. \cite[7.1, 8.3]{PR}.  This is a set of conditions stipulating the vanishing  of certain linear forms on $\wedge^n \Lambda_{1}$ on the line $\wedge^n\CF_1$ in $\wedge^n \Lambda_{1, S}$. These linear forms are enumerated by certain subsets $E\subset\{1,\ldots,2n\}$ of order $n$.  For a subset $E\subset\{1,\ldots,2n\}$ of order $n$, set $E^\perp=(2n+1-E)^c$. Also, to such a subset $E$ is associated a permutation $\sigma_E$ of $S_{2n}$, cf. \cite[7.1.3]{PR}. We call the \emph{weak spin condition} the vanishing of the linear forms corresponding to subsets $E$ with the property
\begin{equation}
E=E^\perp,\quad \vert E\cap \{2, 3,\ldots, n+1\}\vert=n-1, \quad \sgn\sigma_E=1 . 
\end{equation}
It is easy to see that there are precisely the following subsets satisfying this condition: $\{1, \ldots, n\}$ and $\{2, \ldots, n-1, n+1, 2n\}$. 

\begin{lemma}\label{lemweak}
	The weak spin condition on $\CF_1$ implies $a_{n-1, n-1}=a_{n n}=0$.
\end{lemma}
\begin{proof}
	Indeed, the linear forms for $E=\{1, \ldots, n\}$, resp., $E=\{2, \ldots, n-1, n+1, 2n\}$, are the minors of size $n$ consisting of the rows $\{1, \ldots, n\}$, resp., $\{2, \ldots, n-1, n+1, 2n\}$. 
\end{proof}

The incidence relation between $\CF_{2n-1}$ and $\CF_1$ is given by the following system of equations,
$$\begin{pmatrix}
0&\pi&0&\ldots&0\\
0&0&1&\ldots&0\\
\vdots&&&\ddots&\\
0&0&\ldots&0&1\\
b_{11}&b_{12}&\ldots&&b_{1n}\\
\vdots&&&&\vdots\\
b_{n1}&b_{n2}&\ldots&&b_{n n}\\
\pi&0&\ldots&&0
\end{pmatrix}
=
\begin{pmatrix}
a_{n 1}&a_{n2}&\ldots&a_{n n}\\
1&0&\ldots&0\\
0&1&\ldots&0\\
\vdots&&\ddots&\\
0&0&\ldots&1\\
a_{11}&a_{12}&\ldots&a_{1n}\\
\vdots&&&\vdots\\
a_{n-1, 1}&a_{n-1,2}&\ldots&a_{n-1, n}
\end{pmatrix}
\cdot
\begin{pmatrix}
0&0&1&\ldots&0\\
\vdots&&&\ddots&\\
0&0&\ldots&0&1\\
b_{11}&b_{12}&b_{13}&\ldots&b_{1n}\\
b_{21}&b_{22}&b_{23}&\ldots&b_{2n}
\end{pmatrix}
$$
Using \eqref{baorth} and Lemma \ref{lemweak}, we may also write the equations for the closed sublocus $U^{\rm wspin}$  where, in addition to the incidence relation from $\CF_0$ to $\CF_1$ and the isotropy condition on $\CF_0$,  the weak spin condition is  satisfied, as an identity of $n\times n$-matrices,

\begin{equation*}
\begin{aligned}
\begin{pmatrix}
0&\pi&\ldots&&0\\
-a_{n,n-2}&-a_{n-1, n-2}&\ldots&&-a_{1,n-2}\\
-a_{n,n-3}&-a_{n-1, n-3}&\ldots&&-a_{1,n-3}\\
\vdots&&&&\vdots\\
-a_{n,1}&-a_{n-1,1}&\ldots&&-a_{11}\\
\pi&0&\ldots&&0
\end{pmatrix}
=
\begin{pmatrix}
a_{n 1}&\ldots&a_{n, n-1}&0\\
a_{11}&\ldots&a_{1, n-1}&a_{1,n}\\
a_{21}&\ldots&a_{2, n-1}&a_{2,n}\\
\vdots&&&\vdots\\
a_{n-2, 1}&\ldots&a_{n-2, n-1}&a_{n-2, n}\\
a_{n-1, 1}&\ldots&0&a_{n-1, n}
\end{pmatrix}\\
\cdot
\begin{pmatrix}
0&0&1&\ldots&0\\
0&0&0&\ldots&0\\
\vdots&&&\ddots&\\
0&0&\ldots&0&1\\
0&-a_{n-1,n}&-a_{n-2,n}&\ldots&-a_{1,n}\\
-a_{n, n-1}&0&-a_{n-2, n-1}&\ldots&-a_{1, n-1}
\end{pmatrix}
\end{aligned}
\end{equation*}

It implies (look at the $(1,2)$ entry, which also equals the $(n,1)$ entry)
\begin{equation}\label{piO}
a_{n-1, n}\cdot a_{n, n-1}=-\pi . 
\end{equation}
Let us call $E_{i j}$ the polynomial identity among the $a_{\mu \nu}$ that is given by the entry $i, j$ of the above matrix identity.  Then $E_{1, j}$ for $3\leq j\leq n$ is of the form
$$
a_{n,j-2}=P_{1, j}(a_{\bullet, n-1}, a_{\bullet, n}) .
$$ 
The identities $E_{n, j}$ for $3\leq j\leq n$ are of the form
$$
a_{n-1, j}=P_{n, j}(a_{\bullet, n-1}, a_{\bullet, n}) .
$$
The identities $E_{i,1}$ for $2\leq i\leq n-1$ are of the form
$$
a_{n, n-i}=P_{i,1}(a_{\bullet, n-1}, a_{\bullet, n}) .
$$
The identities $E_{i,2}$ for $2\leq i\leq n-1$ are of the form
$$
a_{n-1, n-i}=P_{i,2}(a_{\bullet, n-1}, a_{\bullet, n}) .
$$
The identities $E_{i, j}$ for $2\leq i\leq n-1$ and $3\leq j\leq n$ are of the form
$$
a_{i-1, j-2}+ a_{n+1-j, n-i}=P_{i, j}(a_{\bullet, n-1}, a_{\bullet, n}) .
$$
We also note the following identities
\begin{equation*}
\begin{aligned}
E_{i, j}=& E_{n+2-j, n+2-i}, &\text{for } i\in[2, n-1], j\in[3, n];\\
E_{1,j}=&E_{n+2-j, 1},  &\text{for } j\in[3, n-2];\\
E_{n,j}=&E_{n-j, 2},   &\text{for } j\in[3, n-2].
\end{aligned}
\end{equation*}
We keep the $2(n-2)$ variables $a_{\bullet, n-1}$ and $a_{\bullet, n}$, but eliminate $a_{n, j}$ and $a_{n-1, j}$, for $j\in [1, n-2]$. Then the remaining variables $a_{i, j}$ with $i, j\in [1, n-2]$ satisfy  the identities
$$
a_{i, j}+ a_{n-1-j, n-1-i}=Q_{i, j}(a_{\bullet, n-1}, a_{\bullet, n}) .
$$ It follows that 
\begin{equation}
U^{\rm wspin}\simeq \Spec O_F[X, Y]/(XY-\pi)\times \BA^{\frac{n(n-1)}{2}-1} . 
\end{equation}
We obtain the semi-stability of $\BM^{\rm loc}_{I_0}({\rm GO}_{2n}, \mu_n)$ as in the case of the symplectic group via the chain of closed embeddings 
\[
\BM^{\rm loc}_{I_0}({\rm GO}_{2n}, \mu_n)\subset \BM^{\rm wspin}_{I_0}({\rm GO}_{2n}, \mu_n)\subset  \BM^\loc_{\{1\}}(\GL_{2n}, \mu_n).
\]
\begin{remark}\label{oneenoughSO}
	Again, as in the case of $\GL_n$,  in this case we can see, using flatness, that the further incidence relation from $\CF_1$ to $\CF_{2n-1}$ and from $\CF_{2n-1}$ to $\CF_{1}$, as well as the full spin condition are automatically satisfied. 
\end{remark}

\subsection{Quadric local models}\label{ss:quadricLM}

Let  $V$ be an $F$-vector space of dimension $d=2n$ or $2n+1$, with a non-degenerate symmetric $F$-bilinear form
$\langle\ ,\ \rangle$. Assume that $d\geq 5$ and that $p\neq 2$. We will consider a minuscule coweight $\mu$ of ${\rm SO}(V)(F)$ (i.e., defined over $F$) that
corresponds to cases with $r=1$. 

Recall the notion of a \emph{vertex lattice} in $V$: this is an $O_F$-lattice $\Lambda$ in $V$ such that $\Lambda\subset \Lambda^\vee\subset \pi^{-1}\Lambda$. We will say that an orthogonal vertex lattice $\Lambda$ in $V$ 
is \emph{self-dual} if $\Lambda=\Lambda^\vee$, and \emph{almost self-dual}  if the length $ \lg (\Lambda^\vee/\Lambda)=1$.  We list this as two cases: 

\smallskip

I) $\Lambda$ is self-dual, i.e., $\Lambda^\vee=\Lambda$. 

\smallskip
II) We have
$
\Lambda \subset \Lambda^\vee\subset \pi^{-1}\Lambda,
$
and  $\lg (\Lambda^\vee/\Lambda)=1$.

\medskip

Now let us consider the following cases:
\medskip

a) $d=2n+1$,  there is a  basis $e_i$ with $\langle e_i, e_{d+1-j} \rangle=\delta_{ij}$,
and $\Lambda=\oplus_{i=1}^d O_F\cdot e_i$, so that $\Lambda^\vee=\Lambda$.
\medskip

b) $d=2n$, there is a basis $e_i$ with $\langle e_i, e_{d+1-j} \rangle=\delta_{ij}$,
and $\Lambda=\oplus_{i=1}^d O_F\cdot e_i$, so that we have $\Lambda^\vee=\Lambda$.
\medskip

c) $d=2n+1$, there is a basis $e_i$ with $\langle e_i, e_{d+1-j} \rangle=\pi\delta_{ij}$, and $\Lambda=(\oplus_{i=1}^{n} O_F\cdot \pi^{-1}e_i)\oplus (\oplus_{i=n+1}^dO_F\cdot e_i)$, so that  $
\Lambda \subset \Lambda^\vee\subset \pi^{-1}\Lambda
$.
\medskip

d) $d=2n$, there is a  basis $e_i$ with $\langle e_i, e_{d+1-j} \rangle=\delta_{ij}$,
if $i, j\neq n, n+1$, and $\langle e_n, e_n\rangle =\pi$, $\langle e_{n+1}, e_{n+1} \rangle=1$, $\langle e_{n}, e_{n+1} \rangle=0$, 
and  $\Lambda=\oplus_{i=1}^d O_F\cdot e_i$, so that  $
\Lambda \subset \Lambda^\vee\subset \pi^{-1}\Lambda
$.
\medskip

In all these cases, we take $\mu: \BG_m\to {\rm SO}(V)$ to be given by $\mu(t)={\rm diag}(t^{-1},1,\ldots , 1, t)$
(under the embedding into the group of matrices by giving the action on the basis).
\medskip

It follows from the classification of quadratic forms over local fields \cite{Ger} that for each $(V, \langle\ ,\ \rangle, \Lambda)$ with $\lg(\Lambda^\vee/\pi\Lambda)\leq 1$, and $\mu$ as in the beginning of this  subsection, there is an unramified 
finite field extension $F'/F$ such that the base change of $(V, \langle\ ,\ \rangle, \Lambda)$ to $F'$ affords a basis as in 
one of the cases (a)-(d), and $\mu$ is given as above. In fact, we can also consider similarly cases of $(V, \langle\ ,\ \rangle, \Lambda)$ with 
$\pi\Lambda\subset\Lambda^\vee\subset \Lambda$, with $\lg(\Lambda^\vee/\pi\Lambda)\leq 1$, by changing the
form $\langle\ ,\ \rangle$ to $\pi\langle\ ,\ \rangle$; these two forms have the same orthogonal group.

In what follows, for simplicity we set $O_F=O$.

\subsubsection{Quadrics} We will now consider the quadric $Q(\Lambda)$ over $\Spec O$ which, by definition, is the projective hypersurface in $\BP^{d-1}_{O}$
whose $R$-valued points parametrize isotropic lines, i.e., $R$-locally free rank $1$ direct summands
\[
\CF\subset \Lambda_R:=\Lambda\otimes_{O}R,
\]
with $\langle \CF, \CF\rangle_R=0$. Here $\langle \ ,\ \rangle_R$ is
the symmetric $R$-bilinear form $\Lambda_R\times \Lambda_R\to R$ obtained 
from $\langle \ ,\ \rangle$ restricted to $\Lambda\times\Lambda$ by base change.

\begin{proposition}\label{propQ} Set $\BP^{d-1}_{\br O}= {\rm Proj}(\breve {O}[x_1,\ldots,x_d])$.

	  In case (I), $Q(\Lambda)\otimes_{O}\breve {O}$ is isomorphic to the closed subscheme of $\BP^{d-1}_{\br O} $ given by 
	\[
	\sum\nolimits_{i=1}^{d}x_ix_{d+1-i}=0\]
	and the scheme $Q(\Lambda)$ is smooth over ${O}$. 
	\medskip
	
	 In case (II), $Q(\Lambda)\otimes_{O}\breve {O}$ is isomorphic to   the closed subscheme of $\BP^{d-1}_{\br O}$ given by
	\[ \sum\nolimits_{i=1}^{d-1}x_ix_{d-i}+\pi x^2_d=0.\]
	Then $Q(\Lambda)$ is regular with normal special fiber which is singular only at the point 
	$(0;\ldots;1)$; this point corresponds to $\CF_0=\pi\Lambda^\vee/\pi\Lambda\subset \Lambda/\pi\Lambda=\Lambda\otimes_{O}k$.
\end{proposition}

\begin{proof}
	Follows from the classification of quadratic forms over $\breve F$, by expressing $Q(\Lambda)$ in cases (a)-(d).
	(To get the equations in the statement we have to rearrange the basis vectors.)
\end{proof}

\subsubsection{Quadrics and PZ local models}\label{sss:PZmod} 
We can now extend our data to ${O}[u, u^{-1}]$. We set $\BV=\oplus_{i=1}^{d} {O}[u, u^{-1}]\cdot \underline e_i$
with $\langle\ ,\ \rangle: \BV\times \BV\to {O}[u, u^{-1}]$ a symmetric ${O}[u, u^{-1}]$-bilinear form 
for which $\langle\underline e_j ,\underline e_j\rangle$ is given as $\langle e_i , e_j \rangle$ above, but with 
$\pi$ replaced by $u$. We define $\underline \mu:\BG_m\to {\rm SO}(\BV)$ by $\underline\mu(t)={\rm diag}(t^{-1},1,\ldots , 1, t)$ by using the basis $\underline e_i$.

Similarly, we define $\BL$ to be  the free ${O}[u]$-submodule of $\BV$ spanned by $\underline e_i$
(or $u^{-1} \underline e_i$ and $\underline e_j$) as above, following the pattern 
of the definition of $\Lambda$ in each case. Then, the base change of $(\BV, \langle \ ,\ \rangle, \BL)$  
from   ${O}[u, u^{-1}]$ to $F$ given by $u\mapsto \pi$ are $(V, \langle \ ,\ \rangle, \Lambda)$.

We can now define the local model $\BM^{\rm loc}=\BM^{\rm loc}(\Lambda)=\BM^{\rm loc}_K({\rm SO}(V),\{\mu\})$ for the LM triple $({\rm SO}(V), \{\mu\}, K)$ where $K$ is the parahoric stabilizer of $\Lambda$, as in \cite{PZ}.  
Consider the smooth affine group scheme $\underline\CG$
over ${O}[u]$ given by $g\in {\rm SO}(\BV)$ that also preserve $\BL $ and $\BL ^\vee$.
Base changing by $u\mapsto \pi$ gives the Bruhat-Tits group scheme $\CG$ of ${\rm SO}(V)$ which is the 
stabilizer of the lattice chain $\Lambda\subset \Lambda^\vee\subset \pi^{-1}\Lambda$.
This is a hyperspecial subgroup when $\Lambda^\vee=\Lambda$. If $\lg(\Lambda^\vee/\Lambda)=1$, 
we can see that $\CG$ has special fiber with $\BZ/2\BZ$ as its group of connected components. The corresponding parahoric group scheme is its connected component $\CG^0$. The construction of \cite{PZ} produces the
group scheme $\underline\CG^0$ that extends $\CG^0$. By construction, there is a group scheme immersion
$\underline \CG^0\hookrightarrow \underline\CG$.

As in \cite{PZ}, one can see that the Beilinson-Drinfeld style (``global") affine Grassmannian ${\rm Gr}_{\underline\CG,{O}[u]}$ over ${O}[u]$ represents the functor that sends the ${O}[u]$-algebra $R$ given by $u\mapsto r$ to the set of
projective finitely generated $R[u]$-submodules $\CL$ of $\mathbb V\otimes_{{O}}R$
which are $R$-locally free such that $(u-r)^N\mathbb L\subset \CL\subset (u-r)^{-N}\mathbb L$
for some $N>>0$ and satisfy $\CL^\vee=\CL$ in case (I), resp., 
\[
u\CL^\vee \buildrel{d-1}\over\subset\CL\buildrel{1}\over\subset \CL^\vee\buildrel{d-1}\over\subset u^{-1}\CL
\]
 in case (II), with all graded quotients $R$-locally free and of the indicated rank.

By definition, the PZ local model $\BM^{\rm loc}$ is a closed subscheme of the base change ${\rm Gr}_{\CG^0,{O}}={\rm Gr}_{\underline\CG^0, {O}[u]}\otimes_{{O}[u]}{O}$ by ${O}[u]\to {O}$ given by $u\mapsto \pi$.
Consider the ${O}$-valued point $[\CL(0)]$ given by 
\[
\CL(0)=\underline\mu(u-\pi)\BL .
\]
By definition, the PZ local model
$\BM^{\rm loc}$ is the reduced Zariski closure of the orbit of $[\CL(0)]$ in ${\rm Gr}_{\CG^0,{O}}$;
it inherits an action of the group scheme $\CG^0=\underline \CG^0\otimes_{{O}[u]}{O}$.
As in \cite[8.2.3]{PZ}, we can see that the natural morphism ${\rm Gr}_{\underline\CG^0,{O}[u]}\to {\rm Gr}_{\underline\CG,{O}[u]}$ induced by  $\underline \CG^0\hookrightarrow \underline\CG$
identifies $\BM^{\rm loc}$ with a closed subscheme of ${\rm Gr}_{\underline\CG, O}:={\rm Gr}_{\underline\CG,{O}[u]}\otimes_{O[u]}O$.

\begin{proposition}\label{quadricLoc}
	In each of the above cases (a)-(d) with $\lg (\Lambda^\vee/\Lambda)\leq 1$, there is a $\CG$-equivariant isomorphism
	\[
	\BM^{\rm loc}(\Lambda)\xrightarrow{\sim} Q(\Lambda)
	\]
	between the PZ local model as defined above and the quadric.
\end{proposition}

\begin{proof}
	Note that since, by definition, $\CG$ maps to ${\rm GL}(\Lambda)$ and preserves the form $\langle\ ,\ \rangle$,
	it acts on the quadric $Q(\Lambda)$. By the definition of $\CL(0)$, we have 
	\[
	\xy
	(-21,7)*+{(u-\pi)\BL \subset   \CL(0)\cap \BL\ \ \ \ \ };
	(-4.5,3.5)*+{\rotatebox{-45}{$\,\, \subset\,\,$}};
	(-4.5,10.5)*+{\rotatebox{45}{$\,\, \subset\,\,$}};
	(0,14)*+{\BL};
	(0,0)*+{\CL(0)};
	(4,3.5)*+{\rotatebox{45}{$\,\, \subset\,\,$}};
	(4,10.5)*+{\rotatebox{-45}{$\,\, \subset\,\,$}};
	(23.5,7)*+{\ \ \ \ \ \BL+\CL(0) \subset   (u-\pi)^{-1}\BL,};
	\endxy 
	\]
	where the quotients along all slanted inclusions are ${O}$-free of rank $1$. Consider the subfunctor $M$ of ${\rm Gr}_{\CG,{O_F}}$ parametrizing $\CL$ such that 
	\[
	(u-\pi)\BL\subset \CL\subset (u-\pi)^{-1}\BL .
	\]
	 Then $M$ is given by a closed subscheme
	of ${\rm Gr}_{\CG, {O}}$ which contains the orbit of $\CL(0)$; therefore the local model 
	$\BM^{\rm loc}$ is a closed subscheme of $M$ and $\BM^{\rm loc}$ is the reduced Zariski closure
	of its generic fiber in $M$. 
	
	We now consider another projective scheme $P(\Lambda)$  which
	parametrizes pairs $(\CF, \CF')$ where $\CF\subset \Lambda_R$, $\CF'\subset \Lambda^\vee_R$ are both $R$-lines,
	such that $\CF$ is isotropic for $\langle\ ,\ \rangle_R$ and $\CF'$ is isotropic
	for $\pi\langle\ ,\ \rangle_R$, and such that $\CF$, $\CF'$ are linked by both natural $R$-maps 
	$\Lambda_R\to \Lambda^\vee_R$ and  $\pi: \Lambda^\vee_R\to \Lambda_R$.

	\begin{lemma} In case (I), the forgetful morphism $f:P(\Lambda)\xrightarrow{\sim} Q(\Lambda)$ is an isomorphism.
		In case (II), denote by $P(\Lambda)^{\rm fl}$ the flat closure of $P(\Lambda)$. Then the forgetful morphism $f: P(\Lambda)^{\rm fl}\to Q(\Lambda)$, given by $(\CF, \CF')\mapsto \CF$,
		can be identified
		with the blow-up of $Q(\Lambda)$ at the unique singular closed point of its special fiber. In particular, it 
		is an isomorphism away from the closed point given by $\CF=$ the radical of the form $\langle\ ,\ \rangle$
		on $\Lambda_{\kappa_F}$. If $\CF$ is the radical, then $\CF'$ lies in the radical of 
		the form $\pi\langle\ ,\ \rangle$ on $\Lambda^\vee\otimes_{O}\kappa_F$. Since  this radical has dimension $d-1$, the exceptional locus is isomorphic to $\BP^{d-2}_{\kappa_F}$.
	\end{lemma}
	
	\begin{proof} In case (I), we have $\CF'=\CF$, so $P(\Lambda)\simeq Q(\Lambda)$. Assume we are in case (II).
		Using the universal property of the blow-up, we see it is enough to show the statement after base changing to $\breve {O}$. For convenience we rearrange the basis of $V$ such that 
		$\Lambda^\vee/\Lambda$ is generated by $\pi^{-1}e_d$. Set
		\[
		\CF=(\sum_{i=1}^d x_ie_i), \quad \CF'=(\sum_{i=1}^{d-1} y_i e_i+  y_d \pi^{-1}e_d).
		\]
		Since $\CF$ maps to $\CF'$ by $\Lambda_R\to \Lambda^\vee_R$ and $\CF'$ maps to  $\CF$ by $\pi: \Lambda_R^\vee\to \Lambda_R$, there are $\lambda$, $\mu\in R$ such that $\lambda\mu=\pi$ and
		\[
		x_1=\lambda y_1,\ldots, x_{d-1}=\lambda y_{d-1},\quad \pi x_d=\lambda   y_d
		\]
		\[
		\pi y_1=\mu x_1, \ldots , \pi y_{d-1}=\mu x_{d-1},\quad y_d=\mu x_d.
		\]
		The isotropy conditions are
		\[
		\sum_{i=1}^{d-1} x_i x_{d-i}+\pi x_d^2=0, \quad \sum_{i=1}^{d-1} \pi y_iy_{d-i}+y_d^2=0.
		\]
		Let us consider the inverse image of the affine chart with $x_d=1$ under the forgetful morphism $f: P(\Lambda)\to Q(\Lambda)$. We obtain $y_d=\mu$ and the equations become
		\[
		\lambda(\sum_{i=1}^{d-1}  x_{i}y_{d-i}+ \mu)=0, \quad \mu(\sum_{i=1}^{d-1}  x_i y_{d-i} +\mu )=0.
		\]
		The flat closure $P(\Lambda)^{\rm fl}$ is given by $\lambda\mu=\pi$ and $\sum_{i=1}^{d-1}  x_{i}y_{d-i}+ \mu=0$.
		Since $x_i=\lambda y_i$, we get $\lambda \sum_{i=1}^{d-1}  y_{i}y_{d-i}+ \mu=0$.
		Eliminating $\mu$ gives
		\[
		-\lambda^2(\sum_{i=1}^{d-1}  y_{i}y_{d-i})=\pi.
		\]
		An explicit calculation shows that this coincides with the blow-up of the affine chart $x_d=1$ in the quadric $Q(\Lambda)$ at the 
		point $(0:\ldots :1)$ of its special fiber.
		In fact, we   see that $P(\Lambda)^{\rm fl}$ is regular and that the special fiber $P(\Lambda)^{\rm fl}_k$    has two irreducible components: The smooth blow up of the special fiber $Q(\Lambda)\otimes_{O}\kappa_F$ at the singular point and
		the exceptional locus  $\BP^{d-2}_{\kappa_F}$ for $\lambda=0$; they intersect along a smooth quadric of dimension $d-3$ over $\kappa_F$. The exceptional locus has multiplicity $2$ in the special fiber of $P(\Lambda)^{\rm fl}$. The rest of the statements  follow easily.
	\end{proof} 
	
	We now continue with the proof of Proposition \ref{quadricLoc}.
	Assume that $(\CF, \CF')$ gives an $R$-valued point of $P(\Lambda)$.
	The pair $(\CF, \CF')$  uniquely determines lattices $\CL(\CF)$ and $\CL'(\CF')$ 
	with
	\[
	(u-\pi)\BL\subset \CL'(\CF')\buildrel{1}\over\subset \BL\buildrel{1}\over\subset \CL(\CF)\subset (u-\pi)^{-1}\BL, 
	\]
	by 
	\begin{equation*}
	\begin{aligned}
	\CL(\CF)&=\text{ the inverse image of $\CF$ under $u-\pi: (u-\pi)^{-1}\BL\to \BL/(u-\pi)\BL=\Lambda_R$,}\\
	\CL'(\CF')&=\text{ the inverse image of $\CF'^\perp\subset \Lambda_R$ under $\BL\to \BL/(u-\pi)\BL=\Lambda_R$.}
	\end{aligned}
	\end{equation*}
	The other conditions translate to $(u-\pi)\CL(\CF)\subset \CL'(\CF')\subset \CL(\CF)$ and
	\[
	\CL'(\CF')\subset \CL(\CF)^\vee\subset u^{-1}\CL'(\CF'),
	\]
	\[
	\CL(\CF)\subset \CL'(\CF')^\vee
	\subset u^{-1}\CL(\CF).
	\]
	Note that we obtain a symmetric $R$-bilinear form by interpreting the value  $\langle\ ,\ \rangle$ in $(u-\pi)^{-1}R[u]/R[u]\simeq R$,
	\[
	h: \CL(\CF)/\CL'(\CF')\times \CL(\CF)/\CL'(\CF')\to R
	\]
	Consider the scheme $Z\to P(\Lambda)^{\rm fl}\subset P(\Lambda)$ classifying isotropic lines in the rank $2$ symmetric space $\CL(\CF^{\rm univ})/\CL'(\CF'^{\rm univ})$ over $P(\Lambda)$. One of these isotropic lines is always $\BL/\CL'(\CF')$.
	Suppose $R=k$. When $\CF$ is the radical in $\Lambda_k$, then $\CL(\CF)=\BL^\vee$. Then $\BL/\CL'(\CF')$ is the radical of $h$ and gives the unique isotropic line.
	If $\CF$ is not the radical in $\Lambda_k$, then $h$ is non-degenerate and there are two distinct isotropic lines,  one of which is
	$\BL/\CL'(\CF')$. 
	
	We first consider case (I), i.e.,  $\Lambda=\Lambda^\vee$. Then $\CF=\CF'$, and $P(\Lambda)=Q(\Lambda)$. The form
	$h$ is perfect (everywhere non-degenerate) and $Z$ is the disjoint union $Z=Z^0\sqcup Z^1$, where $Z^0$ is the component
	where the isotropic line is $\BL/\CL'(\CF')$. Each component $Z^i$ projects isomorphically to $Q(\Lambda)$. We can give a morphism $g: Z^1\simeq P(\Lambda)=Q(\Lambda)\to \BM^{\rm loc}$ by sending $\CF$ to $\CL$ characterized by  the condition that $\CL/\CL'(\CF')\subset \CL(\CF)/\CL'(\CF')$ is the tautological isotropic line over $Z^1$. The morphism $g$ is the desired equivariant isomorphism
	$Q(\Lambda)\simeq \BM^{\rm loc}(\Lambda)$.
	
	We now consider case (II), i.e.,  $\lg(\Lambda^\vee/\Lambda)=1$. Then 
	the scheme $Z\to P(\Lambda)^{\rm fl}$ has two irreducible components $Z^0$, $Z^1$, where 
	$Z^0$ is the irreducible component over which the isotropic line 
	is  $\BL/\CL(\CF')$ and where  $Z^1$ is the irreducible component over which the isotropic line 
	generically is not $\BL/\CL(\CF')$. By the above, the two components intersect over the exceptional locus 
	of the blow-up $P(\Lambda)^{\rm fl}\to Q(\Lambda)$. Each irreducible component maps isomorphically to $P(\Lambda)^{\rm fl}$.
	(Note that $P(\Lambda)^{\rm fl}$ is normal and each morphism $Z^i\to P(\Lambda)^{\rm fl}$ is clearly birational and finite.)
	We can now produce a morphism $g: Z^1\simeq P(\Lambda)^{\rm fl}\to \BM^{\rm loc}$
	by sending $(\CF, \CF')$ to $\CL$  characterized by  the condition that  $\CL/\CL(\CF')\subset \CL(\CF)/\CL(\CF')$ is the tautological isotropic line over $Z^1$, 
	\[
	\xy
	(-21,7)*+{(u-\pi)\BL \subset   \CL'(\CF')\  };
	(-4.5,3.5)*+{\rotatebox{-45}{$ \subset$}};
	(-4.5,10.5)*+{\rotatebox{45}{$ \subset$}};
	(0,14)*+{\BL};
	(0,0)*+{\CL};
	(4,3.5)*+{\rotatebox{45}{$ \subset$}};
	(4,10.5)*+{\rotatebox{-45}{$ \subset$}};
	(23.5,7)*+{\  \CL(\CF) \subset   (u-\pi)^{-1}\BL.}
	\endxy 
	\]
	When $\CL\neq \BL$, then $\CL(\CF)=\BL+\CL$ and $\CL'(\CF')=\BL\cap\CL$ and so $(\CF, \CF')$ is uniquely determined
	by its image $\CL$ in $\BM^{\loc}$.  Hence, $g$ is an isomorphism over the open subscheme of $\BM^{\rm loc}$ where
	$\CL\neq\BL$. When $\CL=\BL$, $\BL/\CL(\CF')$ is isotropic in $\CL(\CF)/\CL(\CF')$ and, as above, $\CF=$ the radical of the form on $\Lambda_k$. This shows that the inverse image of $g: P(\Lambda)\to \BM^{\rm loc}$ over $[\BL]$ agrees with exceptional locus of the blow-up $f:P(\Lambda)^{\rm fl}\to Q(\Lambda)$ over the point $\CF$ given by the radical. Since   $Q(\Lambda)$, $\BM^{\rm loc}$ are both normal, we can conclude that the birational map 
	$f\circ g^{-1}: \BM^{\rm loc}\dashrightarrow Q(\Lambda)$ is an isomorphism; it is $\CG$-equivariant since
	this is true on the generic fibers. This completes the proof of Proposition \ref{quadricLoc}.
\end{proof}

\subsubsection{More orthogonal local models}\label{sss:orthMod}  We will now consider orthogonal local models
associated to the self-dual chains generated by two lattices $\Lambda_0$, $\Lambda_n$ and their duals, where $\Lambda_0$, $\Lambda_n$  are both self-dual or  almost self-dual vertex lattices, $\Lambda_0$ for the form $\langle\ ,\ \rangle$ and $\Lambda_n$ for its multiple $\pi\langle\ ,\ \rangle$. In all cases, the self-dual 
lattice chain has the form
\[
\cdots \subset\Lambda_0\buildrel{r}\over\subset \Lambda^\vee_0\subset \Lambda_n\buildrel{s}\over\subset \pi^{-1}\Lambda_n^\vee\subset \pi^{-1}\Lambda_0\subset \cdots
\]
with each $r$, $s$ either $0$ or $1$. Again, after an unramified extension of $F$, we can 
reduce to the following cases:
\medskip

1) $d=2n+1$, there is a basis $e_i$ with   $\langle e_i, e_{d+1-j} \rangle=\delta_{ij}$,
$\Lambda_0=\oplus_{i=1}^d {O}\cdot e_i$ so that $\Lambda^\vee_0=\Lambda_0$
and $\Lambda_n=(\oplus_{i=1}^{n} {O}\cdot \pi^{-1}e_i)\oplus (\oplus_{i=n+1}^d{O}\cdot e_i)$ so that $\Lambda_n \subsetneq \pi^{-1}\Lambda_n^\vee$.
\medskip

2) $d=2n$, there is a basis $e_i$ with $\langle e_i, e_{d+1-j} \rangle=\delta_{ij}$,
$\Lambda_0=\oplus_{i=1}^d {O}\cdot e_i$ so that $\Lambda^\vee_0=\Lambda_0$,
and $\Lambda_n=(\oplus_{i=1}^{n} {O}\cdot \pi^{-1}e_i)\oplus (\oplus_{i=n+1}^d{O}\cdot e_i)$
so that $\Lambda_n = \pi^{-1}\Lambda_n^\vee$.
\medskip

3) $d=2n$, there is a  basis $e_i$ with $\langle e_i, e_{d+1-j} \rangle=\delta_{ij}$,
if $i, j\neq n, n+1$, and $\langle e_n, e_n\rangle =\pi$, $\langle e_{n+1}, e_{n+1} \rangle=1$,  $\langle e_{n}, e_{n+1} \rangle=0$, 
and  $\Lambda_0=\oplus_{i=1}^d {O}\cdot e_i$ so that  $
\Lambda_0 \subsetneq \Lambda^\vee_0\subset \pi^{-1}\Lambda_0
$ and $\Lambda_n=(\oplus_{i=1}^{n} {O}\cdot \pi^{-1}e_i)\oplus (\oplus_{i=n+1}^d{O}\cdot e_i)$ so that $
\Lambda_n^\vee\subset \Lambda_n \subsetneq \pi^{-1}\Lambda_n^\vee$.
\medskip

We extend $(V, \langle\ ,\ \rangle)$ and $\Lambda_j$   to $(\BV, \langle\ ,\ \rangle)$ and $\BL_j$ over ${O}[u]$ as in \ref{sss:PZmod}. We consider the (smooth, affine) group scheme $\underline\CG=\underline\CG_{\BL_\bullet}$ over ${O}[u]$ 
given by $g\in {\rm SO}(\BV)$ that also preserve the chain
\[
\BL_\bullet: \cdots\subset \BL_0 \subset\BL_0^\vee \subset \BL_n \subset u^{-1}\BL^\vee_n\subset u^{-1}\BL_0\subset\cdots
\] 
The base change of $\underline \CG$ by $u\mapsto \pi$ 
is the the Bruhat-Tits group scheme $\CG$ for ${\rm SO}(V)$ that preserves the chain  
\[
\Lambda_\bullet: \cdots\subset \Lambda_0 \subset \Lambda_0^\vee \subset \Lambda_n \subset \pi^{-1}\Lambda^\vee_n\subset \pi^{-1}\Lambda_0 \subset
\cdots .
\]
This is connected and hence parahoric in cases 1) and 2), since it is contained in the hyperspecial stabilizer of $\Lambda_0$. In case 3), the group of connected components of its special fiber is $\BZ/2\BZ$. The corresponding parahoric is the connected component of $\CG$. 
We can now see that the diagonal embedding gives a closed immersion
\[
\underline\CG_{\BL_\bullet}\hookrightarrow \underline\CG_{\BL_0 }\times \underline\CG_{ \BL_n}
\]
of group schemes over ${O}[u]$. Similarly, we have a compatible closed immersion of the global affine Grassmannian for $\underline\CG_{\BL_\bullet}$ into the product of the ones for 
$\underline\CG_{\BL_0 }$ and $\underline\CG_{\BL_n}$. 

The global affine Grassmannian for $\underline\CG_{\BL_\bullet}$ represents the functor which sends the ${O}[u]$-algebra $R$ given by $u\mapsto r$ to the set of pairs of
projective finitely generated $R[u]$-submodules $(\CL_0, \CL_n)$ of $\mathbb V\otimes_{{O}}R$
which are $R$-locally free, such that $(u-r)^N\mathbb L\subset \CL_i\subset (u-r)^{-N}\mathbb L$
for some $N>>0$, and
\[
\CL_0 \buildrel{r}\over\subset \CL^\vee_0 \buildrel{n-r}\over\subset \CL_n \buildrel{s}\over\subset u^{-1}\CL_n^\vee \buildrel{n-s}\over\subset u^{-1}\CL_0
\]
with all graded quotients $R$-locally free and of the indicated rank. From this and the discussion before Proposition \ref{quadricLoc} it easily follows 
that there is an equivariant closed embedding of local models
\[
\BM^{\rm loc}(\Lambda_\bullet)\hookrightarrow\BM^{\rm loc}(\Lambda_0 )\times\BM^{\rm loc}(\Lambda_n )
\]
which restricts to the diagonal morphism on the generic fibers. Proposition \ref{quadricLoc} now 
gives equivariant isomorphisms $\BM^{\rm loc}(\Lambda_0 )\simeq Q(\Lambda_0, \langle,\ ,\ \rangle)$ and $\BM^{\rm loc}(\Lambda_n)\simeq Q(\Lambda_n, \pi\langle,\ ,\ \rangle)$. In fact, we can now see that the
construction of these isomorphisms in the proof of this proposition implies that the image of the resulting closed embedding
\[
\BM^{\rm loc}(\Lambda_\bullet)\hookrightarrow Q(\Lambda_0, \langle,\ ,\ \rangle) \times Q(\Lambda_n, \pi\langle,\ ,\ \rangle)
\]
lies in the closed subscheme of the product of the two quadrics where the two lines $(\CF_0, \CF_n)$ are linked in the same manner as for the corresponding local model for ${\rm GL}_d$,
i.e., $\CF_0\subset \Lambda_{0, R}$ maps to $\CF_n\subset \Lambda_{n, R}$ via $\Lambda_{0, R}\to \Lambda_{n, R}$ induced by $\Lambda_0\subset \Lambda_n$ and $\CF_n$ maps to $\CF_0$ under $\Lambda_{n, R}\to \Lambda_{0, R}$ induced by $\pi: \Lambda_n\to \Lambda_0$. Indeed, the reason is that $\BM^{\rm loc}(\Lambda_\bullet)$, as a closed subscheme of the Beilinson-Drinfeld Grassmannian, classifies $(\CL_0, \CL_n)$ which in particular satisfy $\CL_0\subset \CL_n\subset u^{-1}\CL_0$. Therefore, we have that $\BL_0+\CL_0\subset \BL_n+\CL_n \subset u^{-1}(\BL_0+\CL_0)$. But, as the proof shows, on the open dense non-singular part of the quadrics, the  sums $\BL_0+\CL_0$ and $\BL_n+\CL_n$ determine the lines $\CF_0$ and $\CF_n$ and we easily see that  the linkage inclusions as above are satisfied.

\subsection{The  case $({\bf split }\,{\rm SO}_{2n}, r=1, \{0, n\})$ }\label{ss: spliteven}

This corresponds to case 2) in \ref{sss:orthMod}. We continue to assume $p\neq 2$. We have $V=\oplus_{i=1}^{2n}F\cdot e_i$ with symmetric $F$-bilinear form determined by
$
\langle e_i, e_{2n+1-j}\rangle=\delta_{ij}.
$
For $1\leq j\leq n$, set 
\[
\Lambda_j=(\pi^{-1}e_1,\ldots, \pi^{-1}e_j, e_{j+1}, \ldots, e_{2n})\simeq O^{2n}\subset V;
\]
Then $\Lambda_0= \Lambda^\vee_0$,  $\pi\Lambda_n= \Lambda^\vee_n$.
In this case, the local model is contained in the closed subscheme $\BM^{\rm naive}_{\{0,n\}}({\rm SO}_{2n}, \mu_1 )$ of the local model $\BM^\loc_{\{0, n\}}(\GL_{2n}, \mu_1)$ given by the condition 
\begin{equation}\label{perpcond}
\CF_0\subset \CF_{0}^\perp, \quad \CF_n\subset \CF_{n}^\perp .
\end{equation}
As open subset $U$ of the worst point we take the scheme of $( \CF_0, \CF_{n})$, where
$$ 
\CF_0=(e_1+  a_1 e_2+\cdots +a_{2n-1}e_{2n}),
$$
$$
 \CF_n =(b_n\pi^{-1}e_1+\cdots +b_{2n-1}\pi^{-1}e_n+e_{n+1}+b_1e_{n+1}+\cdots +b_{n-1}e_{2n}).
$$
The incidences from $\CF_0$ to $\CF_n$, resp.,  from $\CF_n$ to $\CF_0$ are given by the following matrix relations 
$$
\begin{pmatrix}
\pi\\ 
\pi a_1\\
\vdots\\
\pi a_{n-1}\\
a_n\\
\vdots\\
a_{2n-1}
\end{pmatrix} =
\begin{pmatrix}
b_n\\
b_{n+1}\\
\vdots\\
b_{2n-1}\\
1\\
b_1\\
\vdots\\
b_{n-1}
\end{pmatrix}\cdot a_n ,
\quad
\begin{pmatrix}
b_n\\
\vdots\\
b_{2n-1}\\
\pi\\
\pi b_1\\
\vdots\\
\pi b_{n-1}
\end{pmatrix}=
\begin{pmatrix}
1
\\a_1\\
\vdots\\
a_{n}\\
\vdots\\
a_{2n-1}
\end{pmatrix}\cdot b_n  .
$$
We deduce that
\begin{equation}\label{anbnpi}
a_nb_n=\pi ,
\end{equation}
and the following identities for $i=1,\ldots, n-1$, 
\begin{equation}\label{aandb}
\begin{aligned}
a_nb_i=&\, a_{n+i}, \quad &b_n a_i=b_{n+i} , \\
a_nb_{n+i}=&\, \pi a_{i}, \quad &b_n a_{n+i}=\pi b_{i}  .
\end{aligned}
\end{equation}
The isotropy conditions on $\CF_0$, resp., $\CF_n$, are given by the following equations,
\begin{equation}\label{isotro}
\begin{aligned}
a_{2n-1}+a_1a_{2n-2}+\cdots+a_{n-1} a_n&=0 , \\
b_{2n-1}+b_1b_{2n-2}+\cdots+b_{n-1} b_n&=0  .
\end{aligned}
\end{equation}
We use the first lines of \eqref{aandb} to eliminate $a_{n+1}, \ldots, a_{2n-1}$ and $b_{n+1}, \ldots, b_{2n-1}$. Then the second lines in \eqref{aandb} are automatically
satisfied (use \eqref{anbnpi}),
\[
a_nb_{n+i}-\pi a_i=a_n b_n a_i-\pi a_i=a_i(a_n b_n -\pi)=0;
\]
\[
b_n a_{n+i}-\pi b_i=b_n a_n b_i-\pi b_i=b_i(a_n b_n -\pi)=0.
\]
Expressing $a_{n+1}, \ldots, a_{2n-1}$ in terms of $b_1,\ldots, b_{n-1}$ in the first equation of \eqref{isotro}, we obtain the equation
$$
a_n(b_{n-1}+a_1b_{n-2}+\cdots+a_{n-2}b_{n-1}+a_{n-1})=0 .
$$Similarly, the second equation of \eqref{isotro} gives 
$$
b_n(b_{n-1}+a_1b_{n-2}+\cdots+a_{n-2}b_{n-1}+a_{n-1})=0 .
$$These equations also hold in the generic fiber of $U$; but by \eqref{anbnpi}, both $a_n$ and $b_n$ are units in the generic fiber, and hence we obtain the following equation, first in the generic fiber but then by flatness on all of $U$,
\begin{equation}
b_{n-1}+a_1b_{n-2}+\cdots+a_{n-2}b_{n-1}+a_{n-1}=0. 
\end{equation} 
We can now eliminate $b_{n-1}$ and remain only with equation \eqref{anbnpi} among the indeterminates $a_1,\ldots,a_n, b_1, \ldots, b_{n-2}, b_n$. Hence 
$$
U\simeq\Spec O_F[X, Y]/(XY-\pi)\times \BA^{2n-3}
$$
has semi-stable reduction.

\subsection{The  case $({\bf split }\,{\rm SO}_{2n+1}, r=1, \{0, n\})$ }  
This corresponds to case 1) in \ref{sss:orthMod}. We continue to assume $p\neq 2$. Again we denote by $e_1, \ldots, e_{2n+1}$ a Witt basis, i.e., $\langle e_i, e_{2n+2-j}\rangle=\delta_{ij}$ for $i, j\leq 2n+1$.  

The local model is contained in the closed subscheme $\BM^{\rm naive}_{\{0, n\}}({\rm SO}_{2n+1}, \mu_1)$ of the local model $\BM^\loc_{\{0, n\}}(\GL_{2n+1}, \mu_1)$ given by the condition 
\begin{equation}\label{perpcond}
\CF_0\subset \CF_{0}^\perp,\quad  \CF_n\subset \CF_n^\perp   
\end{equation}
where the second $\perp$ is for  the form $\pi\langle\ , \ \rangle$ on $\Lambda_{n, R}$.

As open subset $U$ of the worst point we take the scheme of $( \CF_0, \CF_{n})$, where 
$$ 
\CF_0=(e_1+  a_1 e_2+\cdots +a_{2n}e_{2n+1}),
$$
$$
 \CF_n =(b_1\pi^{-1}e_1+\cdots +b_{n}\pi^{-1}e_n+e_{n+1}+b_{n+1}e_{n+1}+\cdots +b_{2n}e_{2n}).
$$
The incidences from $\CF_0$ to $\CF_n$, resp.,  from  $\CF_n$ to $\CF_0$ are given by the following matrix relations 
$$
\begin{pmatrix}
\pi\\ 
\pi a_1\\
\vdots\\
\pi a_{n-1}\\
a_n\\
\vdots\\
a_{2n}
\end{pmatrix} =
\begin{pmatrix}
b_1\\
b_{2}\\
\vdots\\
b_{n}\\
1\\
b_{n+1}\\
\vdots\\
b_{2n}
\end{pmatrix}\cdot a_n ,
\quad
\begin{pmatrix}
b_1\\
\vdots\\
b_{n}\\
\pi\\
\pi b_{n+1}\\
\vdots\\
\pi b_{2n}
\end{pmatrix}=
\begin{pmatrix}
1
\\a_1\\
\vdots\\
a_{n}\\
\vdots\\
a_{2n}
\end{pmatrix}\cdot b_1  .
$$
We deduce that
\begin{equation}\label{oddanbnpi}
a_nb_1=\pi ,
\end{equation}
and the following identities , 
\begin{equation}\label{oddaandb}
\begin{aligned}
a_nb_{n+i}=&a_{n+i}, \text{ for $i=1,\ldots, n$ }\quad &b_1 a_{i}=b_{i+1} , \text{ for $i=1,\ldots, n-1$ }\\
a_nb_{i+1}=&\pi a_{i}, \text{ for $i=1,\ldots, n-1$ }\quad &b_1   a_{n+i}=\pi b_{n+i}   \text{ for $i=1,\ldots, n$ }.
\end{aligned}
\end{equation}
The isotropy conditions on $\CF_0$, resp., $\CF_n$, are given by the following equations,
\begin{equation}\label{oddisotro}
\begin{aligned}
2a_{2n}+2a_1a_{2n-1}+\cdots+2a_{n} a_{n+1}+a_n^2&=0 , \\
\pi+2b_1b_{2n}+2b_2b_{2n-1}+\cdots+2b_{n} b_{n+1}&=0  .
\end{aligned}
\end{equation}
We use the first lines of \eqref{oddaandb} to eliminate $a_{n+1}, \ldots, a_{2n}$ and $b_{2}, \ldots, b_{n}$. Then the second lines in \eqref{oddaandb} are automatically
satisfied (use \eqref{anbnpi}),
\[
a_nb_{i+1}-\pi a_i=a_n b_1 a_i-\pi a_i=a_i(a_n b_1 -\pi)=0; 
\]
\[
b_1 a_{n+i}-\pi b_{n+i}=b_1 a_n b_{n+i}-\pi b_{n+i}=b_{n+i}(a_n b_1 -\pi)=0.
\]
Expressing $a_{n+1}, \ldots, a_{2n-1}$ in terms of $b_1,\ldots, b_{n-1}$ in the first equation of \eqref{oddisotro}, we obtain the equation
$$
a_n(2b_{n}+2a_1b_{2n-1}+\cdots+2a_{n-1}b_{n+1}+a_{n})=0 .
$$Similarly, the second equation of \eqref{aandb} gives 
$$
b_1(2b_{2n}+a_1b_{2n-1}+\cdots+2a_{n-1}b_{n+1}+a_{n})=0 .
$$These equations also hold in the generic fiber of $U$; but by \eqref{anbnpi}, both $a_n$ and $b_1$ are units in the generic fiber, and hence we obtain the following equation, first in the generic fiber but then by flatness on all of $U$,
\begin{equation}
2b_{2n}+a_1b_{2n-1}+\cdots+2a_{n-1}b_{n+1}+a_{n} =0.
\end{equation} 
We can now eliminate $b_{2n}$ and remain only with equation \eqref{anbnpi} among the indeterminates $a_1,\ldots,a_{n}, b_1, b_{n+1},\ldots, b_{2n-1}$. Hence 
$$
U\simeq\Spec O_F[X, Y]/(XY-\pi)\times \BA^{2n-2}
$$
has semi-stable reduction.

\subsection{The  case $({\bf nonsplit }\,{\rm SO}_{2n}, r=1, \{0, n\})$ } 
This corresponds to case 3) in \ref{sss:orthMod}. We continue to assume $p\neq 2$.
Considering this case is not essential for the 
proof of the main result, since it has already been excluded in Subsection \ref{case(2-b)}. However, we include it here since it fits the pattern of the 
previous cases.  

We have $V=\oplus_{i=1}^{2n}F\cdot e_i$ with symmetric $F$-bilinear form determined by
\[
\quad (e_i, e_{2n+1-j})=\delta_{ij}, \ \hbox{\rm for $i$, $j\neq n, n+1$}, \  (e_n, e_n)=\pi,\, (e_{n+1}, e_{n+1})=1, \, (e_{n}, e_{n+1})=0.\]
Here, $\Lambda_0\subset \Lambda^\vee_0\subset \pi^{-1}\Lambda_0$, $\pi\Lambda_n\subset \Lambda^\vee_n\subset \Lambda_n$ with the quotients $\Lambda^\vee_0/\Lambda_0$, $\Lambda^\vee_n/\pi\Lambda_n$ both of length one. 

We consider the functor which to an $O$-algebra $R$, associates the set of $\CF_0\subset \Lambda_0\otimes_OR$, $\CF_n\subset \Lambda_n\otimes_OR$,
both $R$-locally direct summands of rank $1$ that are isotropic for the symmetric forms induced by $(\ ,\ )$ on $\Lambda_0\otimes_OR$, resp., by $\pi(\ ,\ )$ on $\Lambda_n\otimes_OR$, and which are linked, i.e.,
$\Lambda_0\otimes_OR\to \Lambda_n\otimes_OR$ maps $\CF_0$ to $\CF_n$ and $\pi: 
\Lambda_n\otimes_OR\to \Lambda_0\otimes_OR$ maps $\CF_n$ to $\CF_0$. This functor is represented
by a closed subscheme 
\[
\BM^{\rm naive}_{\Lambda_\bullet}({\rm SO}(V),\mu_1 )\subset Q(\Lambda_0, (\ ,\ ))\times Q(\Lambda_n, \pi(\ ,\ ))
\]
of the product of the two quadrics. The local model $\BM^{\rm loc}(\Lambda_\bullet)$ is the flat closure of the generic fiber 
of this subscheme.
Set 
\[
\CF_0=(\sum_{i=1}^{2n} x_ie_i),\qquad \CF_n=(\sum_{i=1}^{n} y_i\pi^{-1}e_i+\sum_{i=n+1}^{2n}y_i e_i).
\]
The isotropy conditions translate to:
\begin{equation}
x_1x_{2n}+\cdots + x_{n-1}x_{n+2}+\pi x_n^2+x_{n+1}^2=0,
\end{equation}
\begin{equation}
y_1y_{2n}+\cdots + y_{n-1}y_{n+2}+ y_n^2+\pi y_{n+1}^2=0.
\end{equation}
(Here $(x_1;\ldots ; x_{2n})$, $(y_1;\ldots ;y_{2n})$ are homogeneous coordinates.)
Linkage translates to the existence of $\lambda$, $\mu\in R$ with 
\[
\sum_{i=1}^{n}  x_ie_i+\sum_{i=n+1}^{2n} x_i e_i=\lambda\cdot (\sum_{i=1}^{n} y_i\pi^{-1}e_i+\sum_{i=n+1}^{2n}y_i e_i),
\]
\[
\sum_{i=1}^{n} y_i e_i+\sum_{i=n+1}^{2n}\pi y_i e_i=\mu\cdot (\sum_{i=1}^{n}  x_ie_i+\sum_{i=n+1}^{2n} x_i e_i).
\]
This gives
\[
\pi x_1=\lambda y_1, \ldots, \pi x_n=\lambda y_n, \ x_{n+1}=\lambda y_{n+1},\ldots, x_{2n}=\lambda y_{2n},
\]
\[
y_1=\mu x_1, \ldots , y_n=\mu x_n, \ \pi y_{n+1}=\mu x_{n+1}, \ldots, \pi y_{2n} =\mu x_{2n}.
\]
We obtain $\lambda\mu=\pi$. Now the two isotropy conditions become:
\[
x_1\lambda y_{2n}+\cdots + x_n\lambda y_n +\lambda^2y_{n+1}^2=0,
\]
\[
x_1\mu y_{2n}+\cdots + \mu^2 x^2_n +x_{n+1}\mu y_{n+1}=0.
\]
These give:
\[
\lambda\cdot (x_1y_{2n}+\cdots + x_n y_n +\lambda y_{n+1}^2)=\mu\cdot (x_1y_{2n}+\cdots + \mu x_n^2 +x_{n+1}y_{n+1})=0.
\]
Since $x_{n+1}=\lambda y_{n+1}$, $y_n=\mu x_n$, the expressions in both parentheses are the same, and are equal to
\[
x_1y_{2n}+\cdots +  x_n y_n +x_{n+1}y_{n+1}.
\] 
By flatness,   $x_1y_{2n}+\cdots +  x_n y_n +x_{n+1}y_{n+1}=0$ holds on $\BM^{\rm loc}(\Lambda_\bullet)$.  In fact,
the worst point lies in the affine chart $U$ with $x_{n}=1$ and $y_{n+1}=1$. Then $\mu=y_n$ and $\lambda=x_{n+1}$
and we can see
\[
U\simeq \Spec O_F[X, x_1,\ldots , x_{n-1}, y_{n+2}, \ldots , y_{2n}]/(X(X+x_1y_{2n}+\cdots +x_{n-1}y_{n+2})+\pi)
\]
with $X=x_{n+1}$. The special fiber $U_{\kappa_F}$ has two irreducible components that are both isomorphic to $\BA^{2n-2}_{\kappa_F}$.
Their intersection is isomorphic to 
$$\Spec \kappa_F[x_1,\ldots , x_{n-1}, y_{n+2}, \ldots , y_{2n}]/( x_1y_{2n}+\cdots +x_{n-1}y_{n+2}) ,
$$ which is singular. Therefore, in this case, the local model $\BM^{\rm loc}(\Lambda_\bullet)$ indeed does not have pseudo
semi-stable reduction. 

Note that $\BM^{\rm naive}_{\Lambda_\bullet}({\rm SO}(V),\mu_1 )$ is not flat; the special fiber contains $\lambda=\mu=0$
and $x_{n+1}=\cdots =x_{2n}=0$, $y_1=\cdots =y_{n}=0$. This shows that there is an extra irreducible component isomorphic to
$\BP^{n-1}_{\kappa_F}\times \BP^{n-1}_{\kappa_F}$ given by $(x_1;\ldots ;x_n, 0;\ldots ; 0)\times (0; \ldots ;0, y_{n+1};\ldots ;y_{2n})$.
On this component, the equation 
$x_1y_{2n}+\cdots +  x_n y_n +x_{n+1}y_{n+1}=0$ becomes $x_1y_1+\cdots + x_{n-1}y_{n+2}=0$
and it is not satisfied.

\subsection{The  case $({\bf nonsplit }\,{\rm SO}_{2n}, r=n, \{0\})$ }\label{exoticOrtho} 

Here $n\geq 2$. We have $V=\oplus_{i=1}^{2n}F\cdot e_i$ with symmetric $F$-bilinear form determined by
\[
\quad (e_i, e_{2n+1-j})=\delta_{ij}, \ \hbox{\rm for $i$, $j\neq n, n+1$}, \  (e_n, e_n)=\pi,\, (e_{n+1}, e_{n+1})=1, \, (e_{n}, e_{n+1})=0.\]
Then $\Lambda_0\subset^1 \Lambda^\vee_0\subset \pi^{-1}\Lambda_0$, the quotient $\Lambda^\vee_0/\Lambda_0$ is of length one. In this case, $K$ is the parahoric stabilizer of the selfdual periodic lattice chain
\[
\cdots\subset \pi\Lambda_0^\vee\subset \Lambda_0\subset \Lambda^\vee_0\subset \pi^{-1}\Lambda_0\subset\cdots
\]
The reflex field $E$ is the ramified quadratic extension of $F$ obtained by adjoining the square root of $\pi$.

Set $I=\{0, 1\}$. In this case, by \cite[8.2.3]{PZ}, the local model is contained in the closed subscheme $\BM^{\rm naive}_I({\rm GO}_{2n}, \mu_n)$ of the local model $\BM^\loc_I(\GL_{2n}, \mu_n)\otimes_{O}{O_E}$ described by
\begin{equation}\label{perpcond}
\CF_1=\CF_{0}^\perp.
\end{equation}
(Note that the group ${\rm GO}_{2n}$ is not connected and so the discussion in \cite[p. 215]{PZ} applies.)
As open subset $U$ of the worst point we take the scheme of $( \CF_0, \CF_{1})=(\CF_0, \CF_0^\perp)$, where 
\begin{equation*}
\ \  \CF_0 {=} \left(  \begin{matrix}
		1     &          &        &          \\
		&     1    &        &          \\
		&          & \ddots &          \\
		&          &        & 1        \\
		a_{11} & a_{12} & \cdots & a_{1n} \\
		\vdots  & \vdots   &        &  \vdots  \\
		a_{n,1} & a_{n,2} & \cdots & a_{n,n}
	\end{matrix}  \right).
\end{equation*}
We can now see that $U$ is a  subscheme of the closed subscheme of $\Spec(O_E[a_{i, j}]_{1\leq i, j\leq n})$ defined by the equations
\[
 a_{1n}^2=\pi, \ \ \hbox{\rm and}, 
\]
\[
a_{n+1-i, j}+a_{n+1-j, i}+a_{1i}a_{1j}=0,
\]
(if at least one of $i$ or $j$ is not equal to $n$). This has two irreducible components defined by setting $a_{1n}=\sqrt\pi$, or $a_{1n}=-\sqrt\pi$ respectively. As we can see from the equations, each component is isomorphic to affine space over $O_E$ in the coordinates $a_{i, j}$ with $i+j\leq n$, and is therefore smooth over $O_E$. The generic fiber $U\otimes_{O_E}E$ has two isomorphic connected components, given by the generic fibers of these two irreducible components
and the two irreducible components above are the Zariski closures of these two connected components. Our discussion implies that the corresponding local model, which has an open affine given by the Zariski closure of one of these connected components, is smooth.


\begin{thebibliography}{A99}
 \bibitem{Arz}{K. Arzdorf, \textit{On local models with special parahoric level structure}, Michigan Math. J. \textbf{58} (2009), no. 3, 683--710.}
 
  \bibitem{BL} S. Billey, V. Lakshmibai, Singular loci of Schubert varieties. Progress in Mathematics, {\bf 182}. Birkh\"auser Boston, Inc., Boston, MA, 2000. xii+251 pp. 
  
\bibitem{B} M. Brion, \textit{ Equivariant cohomology and equivariant intersection theory},  Notes by Alvaro Rittatore. NATO Adv. Sci. Inst. Ser. C Math. Phys. Sci., 514, Representation theories and algebraic geometry (Montreal, PQ, 1997), 1--37, Kluwer Acad. Publ., Dordrecht, 1998.
\bibitem{B2} M. Brion, \textit{Rational smoothness and fixed points of torus actions}, Transform. Groups {\bf 4} (1999), no. 2-3, 127--156.

\bibitem{BTII}
E. Bruhat, J. Tits, \textit{Groupes r\'eductifs sur un corps local. {II}. {S}ch\'emas en
  groupes. {E}xistence d'une donn\'ee radicielle valu\'ee}, Publ. Math. IHES  {\bf 60} (1984), 197--376.
  
  \bibitem{BTIII}
E. Bruhat, J. Tits, \emph{Groupes r\'eductifs sur un corps local. {III},  }  J. Fac. Sci. Univ. Tokyo Sect. IA Math. {\bf 34} (1987), no. 3, 671--698.

\bibitem{Bou} N. Bourbaki, {Lie groups and Lie algebras}, Chapters 4--6. Elements of Mathematics. Springer Verlag, Berlin, 2002, xii+300 pp.

\bibitem{Car}
J. B.~Carrell, \emph{The Bruhat graph of a Coxeter group, a conjecture of Deodhar, and rational smoothness of Schubert varieties}, Algebraic groups and their generalizations: classical methods (University Park, PA, 1991), 53--61, Proc. Sympos. Pure Math., 56, Part 1, Amer. Math. Soc., Providence, RI, 1994. 

\bibitem{Conrad} B. Conrad, \textit{Reductive group schemes}, 
 Autour des sch\'emas en groupes. Vol. I, 93--444, Panor. Synth\`eses, 42/43, Soc. Math. France, Paris, 2014.  (Notes for the SGA 3 Summer School, Luminy 2011.)

 \bibitem{D} P. Deligne, \textit{Vari\'et\'es de Shimura: interpr\'etation modulaire, et techniques de construction de mod\`eles canoniques},  Automorphic forms, representations and L-functions, Proc. Sympos. Pure Math., XXXIII,  Part 2, pp. 247--289, Amer. Math. Soc., Providence, R.I., 1979. 
 
 \bibitem{Dr} V.~G.~Drinfeld, \textit{Coverings of p-adic symmetric domains}, Funkcional. Anal. i Prilozen. {\bf 10} (1976), no. 2, 29--40  (Russian).
 
 
 \bibitem{EM} S. Evens, I. Mirkovi\'c, \textit{ Characteristic cycles for the loop Grassmannian and nilpotent orbits}, Duke Math. J. {\bf 97} (1999), no. 1, 109--126.
 
 \bibitem{F} G. Faltings, \textit{The category $\CM\CF$ in the semistable case},  Izv. Ross. Akad. Nauk Ser. Mat. {\bf 80} (2016), no. 5, 41--60; translation in Izv. Math. {\bf 80} (2016), no. 5, 849--868. 
 
\bibitem{GenTil} A. Genestier,  J. Tilouine, 
\textit{Syst\`emes de Taylor-Wiles pour $GSp_4$.  } 
Formes automorphes. II. Le cas du groupe $GSp(4)$. 
Ast\'erisque {\bf 302} (2005), 177--290. 
 
 \bibitem{Ger}
L.~Gerstein, {Basic quadratic forms}. Graduate Studies in Mathematics,
  vol.~90, American Mathematical Society, Providence, RI, 2008. 
  
\bibitem{Goertz}{U. G\"ortz, \textit{On the flatness of models of certain Shimura varieties of PEL-type}, Math. Ann. \textbf{321} (2001), no. 3, 689--727.}

\bibitem {Go2} U.\ G\"ortz, \textit{On the flatness of local models for the symplectic group},  Adv. Math. {\bf 176} (2003), 89--115.

\bibitem{Goertz-He}{U. G\"ortz and X. He, \textit{Basic loci in Shimura varieties of Coxeter type}, Camb. J. Math. {\bf 3} (2015), no. 3, 323--353.}

\bibitem{GHN} U. G\"ortz, X. He, and S. Nie,  \textit{Fully Hodge-Newton decomposable Shimura varieties}, arXiv:1610.05381 

\bibitem{Gross} B. Gross, \textit{Parahorics}, available at http://www.math.harvard.edu/$\sim$gross/eprints.html

\bibitem{H} Th. Haines, \textit{Introduction to Shimura varieties with bad reduction of parahoric type}, Clay Math. Proc. {\bf 4}, 583--642 (2005).

\bibitem{HR} Th. Haines, T. Richarz, \textit{Smooth  Schubert varieties in twisted affine Grassmannians}, arXiv:1809.08464

\bibitem{HR2} Th. Haines, T. Richarz, \textit{Normality and Cohen-Macaulayness of parahoric local models},
 arXiv:1903.10585

\bibitem{HaT} M. Harris,  R. Taylor,
\textit{Regular models of certain Shimura varieties},
Asian J. Math. {\bf 6} (2002), no. 1, 61--94. 

\bibitem{He} X. He, \textit{Kottwitz-Rapoport conjecture on unions of affine Deligne-Lusztig varieties}, Ann. Scient. ENS {\bf 49} (2016), 1125--1141.

\bibitem{HeR} X. He, M. Rapoport, \textit{Stratifications in the reduction of Shimura varieties}, Manuscripta Math. {\bf 152} (2017), no. 3-4, 317--343. 

\bibitem{IM}  N. Iwahori, H. Matsumoto, \textit{On some Bruhat decomposition and the structure of the Hecke rings of $p$-adic Chevalley groups}, Inst. Hautes Etudes Sci. Publ. Math. {\bf 25} (1965), 5--48. 

\bibitem{KM}{N. Katz, B. Mazur, {Arithmetic moduli of elliptic curves}. Annals of Mathematics Studies, vol. \textbf{108}, Princeton University Press, Princeton, NJ, 1985.}

\bibitem{KL} D. Kazhdan, G. Lusztig, \textit{Representations of Coxeter groups and Hecke algebras}, Invent. Math. {\bf 53} (1979), 165--184.

\bibitem{KP} M. Kisin, G. Pappas, \textit{Integral models of Shimura varieties with parahoric level structure},  Publ. Math. IHES   \textbf{128} (2018), 121--218. 

\bibitem{Kr} N.\ Kr\"amer, \textit{Local models for ramified unitary groups},  Abh. Math. Sem. Univ. Hamburg {\bf 73} (2003), 67--80.

\bibitem{Ku} Sh. Kumar, \textit{The nil Hecke ring and singularity of Schubert varieties},  Invent. Math. {\bf 123} (1996), no. 3, 471--506.

\bibitem{Ku2} Sh. Kumar, Kac-Moody groups, their flag varieties and representation theory. Progress in Mathematics, {\bf 204}. Birkh\"auser Boston, Inc., Boston, MA, 2002.

\bibitem{Levin} B.\ Levin, \textit{Local models for Weil-restricted groups.} Compos. Math. {\bf 152} (2016), no. 12, 2563--2601.

\bibitem{MOV} A. Malkin, V. Ostrik, and M. Vybornov, \textit{The minimal degeneration singularities in the affine Grassmannians.} Duke Math. J. {\bf 126} (2005), no. 2, 233--249.

\bibitem{McG} W. McGovern, \textit{Rational singular loci of nilpotent varieties}, unpublished manuscript; e-mail to M.~Rapoport from 12/01/2018.

\bibitem{Milne} J. Milne, \textit{The points on a Shimura variety modulo a prime of good reduction},  pp. 151-253, The zeta functions of Picard modular surfaces, Univ. Montreal, Montreal, QC, 1992.

\bibitem{MS} J. Milne, K.-Y. Shi, \textit{Langlands' construction of the Taniyama group}, 229--260, in:  P. Deligne, J. Milne, A. Ogus, and K.-Y. Shih,  Hodge cycles, motives, and Shimura varieties. Lecture Notes in Mathematics, {\bf 900}. Springer-Verlag, Berlin-New York, 1982. ii+414 pp. 

\bibitem{Pappas} G. Pappas, \textit{On the arithmetic moduli schemes of PEL Shimura varieties.}, J. Alg. Geom. {\bf 9} (2000), 577--605. 

\bibitem{Pappasicm}G. Pappas, \textit{Arithmetic models for Shimura varieties},   Proceedings of the International Congress of Mathematicians---Rio de Janeiro 2018. Vol. II. Invited lectures, 377--398, World Sci. Publ., Hackensack, NJ, 2018. arXiv:1804.04717. 

\bibitem{PCan} G. Pappas, \textit{On integral models of Shimura varieties},  Preprint.

\bibitem{PR1} G. Pappas, M. Rapoport, \textit{Local models in the ramified case. I. The EL-case}, J. Algebraic Geom. {\bf 12} (2003), no. 1, 107--145. 

\bibitem{PRTwisted}
G.~Pappas, M.~Rapoport, \emph{Twisted loop groups and their affine flag varieties}, Adv. Math.
  \textbf{219} (2008), no.~1, 118--198. With an appendix by Th. Haines and
  Rapoport.

\bibitem{PR} G. Pappas, M. Rapoport, \textit{Local models in the ramified case. III. Unitary groups}, J. Inst. Math. Jussieu {\bf 8} (2009), 507-564.

\bibitem{PRS}{G. Pappas, M. Rapoport, and B. Smithling, \textit{Local models of Shimura varieties, I. Geometry and combinatorics}. In \textit{Handbook of Moduli. Vol. III}, Adv. Lect. Math. (ALM), vol. \textbf{26}, Int. Press, Somerville, MA, 2013, pp. 135--217.}

\bibitem{PZ} G. Pappas, X. Zhu, \textit{Local models of Shimura varieties and a conjecture of Kottwitz}, Invent. math. {\bf 194} (2013), 147--254.
\bibitem  {R} M.\ Rapoport, \textit{A guide to the reduction modulo $p$ of Shimura varieties},  Ast\'erisque {\bf 298} (2005), 271--318.

\bibitem  {RV} M.\ Rapoport, E.\ Viehmann, \textit{Towards a theory of local Shimura varieties},  M\"unster J. of Math. {\bf 7} (2014), 273--326.

\bibitem  {R-Z} M.\ Rapoport, Th.\ Zink,  { Period spaces for $p$--divisible groups}.
Ann.\ of Math. Studies {\bf 141}, Princeton University Press, Princeton, NJ, 1996. 
\bibitem {Ri} T.\ Richarz, \textit {Schubert varieties in twisted affine flag varieties and local models},  J. Algebra {\bf 375} (2013), 121--147.

\bibitem{Schber} P. Scholze, J. Weinstein, \textit{Berkeley lectures on $p$-adic geometry}, preprint 2017, www.math.uni-bonn.de/people/scholze/Berkeley.pdf. 

\bibitem{Stacks} The Stacks Project, stacks.math.columbia.edu

\bibitem{T} J.~Tits,  \textit{Reductive groups over local fields}, Automorphic forms, representations and L-functions, Proc. Sympos. Pure Math., XXXIII,  part 1, 29--69, Amer. Math. Soc., Providence, R.I., 1979.

 \bibitem{VasiuZink} A. Vasiu, Th. Zink, \textit{ 
Purity results for $p$-divisible groups and abelian schemes over regular bases of mixed characteristic, }
Doc. Math. {\bf 15} (2010), 571-599. 

 \bibitem{Zhou} R. Zhou, \textit{Mod-p isogeny classes on Shimura varieties with parahoric level structure}, arXiv:1707.09685 
 
 
  \bibitem{Zhufixed} X.~Zhu, \textit{Affine Demazure modules and $T$-fixed point subschemes in the affine Grassmannian}, Adv. Math. {\bf 221} (2009), no. 2, 570--600.
  
 \bibitem{ZhuPRConj} X.~Zhu, \emph{On the coherence conjecture of {P}appas and {R}apoport}, Ann. of Math. (2) \textbf{180} (2014), no.~1, 1--85.
 
  \bibitem{ZhugeoSat} X.~Zhu, \emph{The geometric Satake correspondence for ramified groups},  Ann. Scient. ENS  \textbf{48} (2015), no. 2, 409--451. 

\end{thebibliography}
 \end{document}